\documentclass[12pt,a4paper]{article}
		\usepackage{amsmath,amsfonts,amssymb,amsthm} 
		\usepackage[english]{babel}
		\usepackage[utf8]{inputenc} 
		\usepackage[pdftex]{graphicx}
		\usepackage{pict2e}
		\usepackage{eurosym}
		\usepackage{hyperref}
		\usepackage{epsfig,epstopdf,titling,url,array}
		\usepackage{mathtools} 
		\usepackage{tikz} 
		\usetikzlibrary{decorations.pathreplacing} 
		\usepackage{algorithm2e}
		\usepackage{stmaryrd}
		\usepackage{pdfpages}
		\advance\textheight by \topskip
		\usepackage{enumitem}
		\allowdisplaybreaks 
		\def\Xint#1{\mathchoice
			{\XXint\displaystyle\textstyle{#1}}%
			{\XXint\textstyle\scriptstyle{#1}}%
			{\XXint\scriptstyle\scriptscriptstyle{#1}}%
			{\XXint\scriptscriptstyle\scriptscriptstyle{#1}}%
			\!\int}
		\def\XXint#1#2#3{{\setbox0=\hbox{$#1{#2#3}{\int}$}
				\vcenter{\hbox{$#2#3$}}\kern-.5\wd0}}
		
		\def\dashint{\Xint-}
		\newcommand{\cpt}[2]{#2}
		
		\theoremstyle{plain}
		\newtheorem{theorem}{Theorem}[section]
		\newtheorem{btheorem}{Theorem}
		\newtheorem{lemma}[theorem]{Lemma}
		\newtheorem{proposition}[theorem]{Proposition}
		\newtheorem{corollary}[theorem]{Corollary}
		
		\newtheorem*{theorem*}{Theorem}
		\theoremstyle{definition}
		\newtheorem{definition}[theorem]{Definition}
		\newtheorem{assumption}[theorem]{Assumption}
		\newtheorem{example}[theorem]{Example}
		\newtheorem{remark}[theorem]{Remark}
		\newtheorem{fact}[theorem]{Fact}
		
		\DeclareMathOperator{\geod}{geod}
		\DeclareMathOperator{\Cay}{Cay}

		\DeclareMathOperator{\cut}{cut}
		\DeclareMathOperator{\sep}{sep}
		\DeclareMathOperator{\Sep}{sep}
		\DeclareMathOperator{\SL}{SL}
		\DeclareMathOperator{\Var}{Var}
		\DeclareMathOperator{\support}{support}
		\DeclareMathOperator{\Map}{Map}
		\DeclareMathOperator{\range}{range}
		\DeclareMathOperator{\diam}{diam}
		
		\DeclareMathOperator{\anch}{anch}

		\DeclareMathOperator{\dist}{dist}
		
		\DeclarePairedDelimiter\abs{\lvert}{\rvert}
		\newcommand{\commarque}[2]{}

		\newcommand{\eps}{\epsilon}
		\newcommand{\norm}[1]{\left\| #1 \right\|_{p}}
		\newcommand{\card}[1]{\left| #1 \right|}
		\newcommand{\twonorm}[1]{\left\| #1 \right\|_{2}}
		
		\newcommand{\set}[1]{\left\{ #1 \right\}}

		\newcommand{\R}{\mathbf{R}}
		\newcommand{\Z}{\mathbf{Z}}
		\newcommand{\N}{\mathbf{N}}
		\newcommand{\HH}{\mathbf{H}}

		\newcommand{\fonction}[5]{\begin{array}{lrcl}
				#1: & #2 & \longrightarrow & #3 \\
				& #4 & \longmapsto & #5 \end{array}}

\title{Poincar\'e profiles of lamplighter diagonal products}
\author{Corentin Le Coz\thanks{Supported by projects ANR-14-CE25-0004 GAMME and ANR-16-CE40-0022-01 AGIRA.}\\Université Paris-Saclay, CNRS,\\Laboratoire de mathématiques d'Orsay,\\ 91405, Orsay, France.\\\texttt{corentin.le-coz@universite-paris-saclay.fr}
}
\begin{document}
\maketitle
\begin{abstract}
We exhibit finitely generated groups with prescribed Poincar\'e profiles. It can be prescribed for functions between $n/\log n$ and linear, and is sharp for functions at least $n/(\log\log n)$. These profiles were introduced by Hume, Mackay and Tessera in 2019 as a generalization of the separation profile, defined by Benjamini, Schramm and Tim\'ar in 2012. The family of groups used is based on a construction of Brieussel and Zheng. As applications, we show that there exists bounded degrees graphs of asymptotic dimension one that do not coarsely embed in any finite product of bounded degrees trees, exhibit hyperfinite sequences of graphs of arbitrary large distortion in $L^p$-spaces, and prove the existence of a continuous family of pairwise uncomparable amenable groups.

\end{abstract}
\tableofcontents
\section{Introduction}

The \textbf{separation profile} was introduced by Benjamini, Schramm \& Tim\'ar~\cite{benjaminischrammtimar2012}. As remarked by Hume~\cite{hume2017}, the separation profile of an (infinite) graph $G$ at $n\geq0$ can be defined by
\[\sep_G(n)=\sup\set{\card{V\Gamma}h(\Gamma) \colon \Gamma \subset G, \card{V\Gamma}\leq n},\]
where $h(\Gamma)$ denotes the Cheeger constant of the graph $\Gamma$. Hume, Mackay and Tessera generalized this profile by defining, for any $ p\in\left[0,\infty\right] $ the \textbf{$L^p$-Poincar\'e profile} of an (infinite) graph $G$ by: 
	\[  \Pi_{G,p}(n) = \sup \set{\card{V\Gamma}h_p\left(\Gamma\right) \colon \Gamma \subset G, \card{V\Gamma}\leq n},\]
where $h_p\left(\Gamma\right)$ denotes the $ L^p $-Cheeger constant of the graph $ \Gamma $ (see \cpt{Chapter}{Section}~\ref{s_up_bd} for details). For graphs of bounded degree, the $L^1$-Poincar\'e profile and the separation profile are equivalent up to constants.

A map between graphs of bounded degree is called \textbf{regular} if it is Lipschitz and if the preimage of singletons have a uniformly bounded cardinality. For example, coarse embeddings and quasi-isometric embeddings are regular maps. Separation and Poincar\'e profiles have the property to be \textit{monotone} under regular maps, see Theorem~\ref{thm:mntncty_regmps}. In this generality, the only other invariants known to have this property are volume growth and asymptotic dimension.\\

Separation and Poincar\'e profiles have interesting relations with other known properties or invariants: hyperbolicity~\cite{benjaminischrammtimar2012,humemackaytessera,humemackay_poorly}, volume growth~\cite{humemackaytessera,gournaylecoz}, finite Assouad-Nagata dimension~\cite{hume2017}, isoperimetric profile~\cite{gournaylecoz}. Nevertheless, these profiles are able to give new information: here, we compute a variety of Poincaré profiles for groups all having exponential growth and asymptotic dimension one. On the other hand, the separation profile doesn't always detect the amenability of groups: for example polycyclic groups and product of free groups both have a separation profile $\simeq \frac n{\log n}$, and hyperbolic spaces $\HH^d$ have the same separation profile as $\Z^{d-1}$, when $d$ is at least three. In the latter example, it is worth noticing that Poincar\'e profiles can make a distinction between $\HH^d$ and $\Z^{d-1}$.


It is clear from the definition that any Poincar\'e profile is at least constant and at most linear. It is then natural to ask what are the possible profiles within this range. Here, we obtain any Poincar\'e profile between $\frac n{\log\log n}$ and $n$, see Theorem~\ref{thm_presc_sep} (the lower bounds on Poincar\'e profiles are only valid along a subsequence). To our knownledge, these are the first examples of amenable groups with profiles strictly between $\frac n{\log n}$ and $n$; it is worth noticing that our lower bounds are only valid along a subsequence. Our examples come from Brieussel and Zheng~\cite{brieusselzheng2015} and are amenable groups with exponential growth and asymptotic dimension one. This shows that amenable groups can have a variety of behaviours with respect to Poincar\'e profiles, even within families of groups that are indistinguishable by these classical invariants. As a corollary, we obtain a continuum of amenable groups with pairwise distinct regular classes, see Theorem~\ref{thm:cntinuum}.

%
Our main result is the following.

\begin{btheorem}\label{thm_presc_sep}
	There exist two universal constants $ \kappa_{1} $ and $ \kappa_{2} $ such that the following is true. Let $\rho\colon\R_{\geq 1}\rightarrow\R_{\geq 1}$ be a non-decreasing function  such that $ \frac{x}{\rho(x)} $ is non-decreasing and $ \lim_{\infty}\rho = \infty $. We assume that $\rho$ is injective and that there exists some $ \alpha > 0 $ such that $ \frac{\rho^{-1}(x)}{\exp(x^{\alpha})} $ is non-decreasing. Then, there exists a finitely generated elementary amenable group $ \Delta $ of exponential growth and of asymptotic dimension one such that for any $p\in\left[1,\infty\right)$,
	\begin{align*}
	\Pi_{\Delta,p}(n) &\leq\kappa_1\frac{n}{\rho(\log n)}\quad\text{for any $ n $,}\\
	\text{and}\quad\Pi_{\Delta,p}(n)&\geq 4^{-p}\kappa_2\frac{n}{\rho(\log n)}\quad\text{for infinitely many $ n $'s.}
	\end{align*}
\end{btheorem}
This theorem applies for example with $\rho = \log$. These groups are built using the construction of Brieussel and Zheng in~\cite{brieusselzheng2015}. As it is shown in this paper, the group $ \Delta $ of Theorem~\ref{thm_presc_sep} also have prescribed speed and entropy of random walk equivalent to $ \frac{n}{\rho(\sqrt{n})} $, $ \ell^p $-isoperimetric profile equivalent to $ \rho(\log(n))^{-p} $, a return probability defined implicitly with $ \rho $, and an $ L^p $-equivariant compression gap of the form $ \left(\frac{\rho}{\log^{1+\epsilon}(\rho)},\rho\right) $. See~\cite[Theorem 1.1]{brieusselzheng2015} for details.

Unfortunately, we were not able to make our upper and lower bounds match each other in all cases, but only on \textit{high} separation profiles. In general, we have the following statement.

\begin{btheorem}\label{thm_presc_sep_approx}
	There exist two universal constants $ \kappa_{1} $ and $ \kappa_{2} $ such that the following is true. Let $\rho\colon\R_{\geq 1}\rightarrow\R_{\geq 1}$ be a non-decreasing function  such that $ \frac{x}{\rho(x)} $ is non-decreasing and $ \lim_{\infty}\rho = \infty $. Then, there exists a finitely generated elementary amenable group $ \Delta $ of exponential growth and of asymptotic dimension one such that for any $p\in\left[1,\infty\right)$,
\begin{align*}
	\Pi_{\Delta,p}(n) &\leq\kappa_1\frac{n}{\rho(\log n)}\quad\text{for any $ n $,}\\
	\text{and}\quad\Pi_{\Delta,p}(n) &\geq 4^{-p}\kappa_2\frac{n}{\rho(\log n)^2}\quad\text{for infinitely many $ n $'s.}
\end{align*}
\end{btheorem}
The lower bound of Theorem~\ref{thm_presc_sep_approx} can be improved for functions $ \rho $ that grow slower than $\sqrt{x}$. This is the following theorem:
\begin{btheorem}\label{thm_better_bound}
	Under the setting of Theorem~\ref{thm_presc_sep_approx}, there exists a universal constant $ \kappa_3 > 0 $ such that if $\rho$ is injective and there exists $ a\in\left(0,1/2\right) $ such that $ \frac{\rho^{-1}(x)}{x^{1/a}} $ is non-decreasing, then, for any $p\in[1,\infty)$,
	\[ \Pi_{\Delta,p}(n) \geq 4^{-p}\kappa_3 \frac{n}{\rho(\log n)^{\frac{1}{1-a}}}\quad\text{for infinitely many $ n $'s.}\]
\end{btheorem}
	See Theorem~\ref{thm_better_bound_gener} for a more general statement.\\

The upper bounds are obtained using compression in $ L^p $ spaces. The compression of a $1$-Lipschitz embedding $f\colon G\to L^p$ is defined by	
\[\rho_f(t)=\inf\set{\norm{f(g)-f(h)}\mid d_G(g,h) \geq t}.\]
The upper bounds of Theorems~\ref{thm_presc_sep} and~\ref{thm_presc_sep_approx} are obtained from the following more general statement:
\begin{btheorem}\label{thm_borne_sup1}
	Let $ G $ be a graph of bounded degree. Then there exists two constants $ c_1,c_2 > 0 $, depending only on the maximum degree in $G$, such that if $ f\colon VG \rightarrow L^p $ is a $ 1 $-Lipschitz map, then
\[\Pi_{G,p}(n) \leq c_1 \frac{n}{\rho_f(c_2\log n)},\]
for all $p\in \left [ 1,\infty \right) $ and $n\geq0$.
\end{btheorem}

This theorem is of independent interest, since it holds in great generality. Moreover, this inequality is known to be sharp for finite products of bounded degree trees.
Indeed, they can be embedded in $L^p$ spaces with compression function $\rho\succeq t^{1-\epsilon}$ (see~\cite[Corollary 2]{Tes11}). Then, Theorem~\ref{thm_borne_sup1} gives that their Poincar\'e profiles satisfy $\Pi_p\preceq\frac n{(\log n)^{1-\epsilon}}$ (for $p=1$, one can actually take $\rho=t$).
This is quite optimal since on the other hand, we have~$\Pi_p\succeq_p\frac n{\log n}$, as soon as at least two of the trees coarsely contain the infinite binary tree, see~\cite{benjaminischrammtimar2012} and Theorem~\ref{t:comppoincsep}. 

More generally, the same reasoning applies to finite products of finitely generated hyperbolic groups (Tits alternative).

Other cases are examined in the more precise statement Theorem~\ref{thm_borne_sup}.

\subsection{About the proofs}
\paragraph{Lower bounds}
The lower bounds of Theorems~\ref{thm_presc_sep}, \ref{thm_presc_sep_approx} and~\ref{thm_better_bound} are obtained by exhibiting particular subgraphs of the groups $ \Delta $. These subgraphs are compared to Cartesian powers of finite graphs. Along the way, we make a general study of these graphs in \cpt{section}{subsection}~\ref{s:cheegercartesian}. In particular, we prove the following proposition, that might be of independent interest:
\begin{proposition}
	Let $ G $ be a connected regular graph. Let $ k $ be a positive integer and $ G^{k} = \underbrace{G\times \dots \times G }_{k \text{ terms} }$ the Cartesian product of $k$ copies of $ G $. Then
	\[ \frac{a}{k} \leq h(G^k) \leq  \frac{b}{\sqrt{k}},\]	
	with $ a = \left(\frac{h(G)}{2\deg G}\right)^2 $ and $ b =(2\sqrt{2}+2)\sqrt{\deg(G)h(G)}.$
\end{proposition}We recall that for any finite graph $H$, $h(H)$ denotes the Cheeger constant of $H$ (see Definition~\ref{d:cbntrlchgrcnst}). Since $G^k$ can have an arbitrary large degree, it is important to remark that Cheeger constants are defined using extern-vertex boundary, see Proposition~\ref{p:prop_prod}. The proof relies on classical spectral graph theory, and results of Bobkov, Houdr\'{e} and~Tetali \cite{bobkovbhoudretetali2000} on vertex-isoperimetry and $ L^{\infty} $-spectral gap.

\paragraph{Upper bounds} As mentioned before, the upper bounds are obtained mapping graphs in $L^p$ spaces. The basic idea is to use such an embedding as a ``test'' function in the definition of the $L^p$-Cheeger constant (see Definition~\ref{d:lpchgrcnst}, Proposition~\ref{p:cheeger-lp-valued}, Theorem~\ref{thm_borne_sup}). In the particular case of the groups studied in this paper, the upper bounds of Theorems~\ref{thm_presc_sep} and~\ref{thm_presc_sep_approx} follow from explicit embeddings given in~\cite{brieusselzheng2015}.

\subsection{Applications}
We present here some applications of the preceding statements.
\paragraph{A continuum of distinct regular classes}
Given two graphs of bounded degree $G$ and $H$, let us recall that a map from $G$ to $H$ is called \textbf{regular} if it is Lipschitz and if the preimage of singletons of $H$ have a uniformly bounded cardinality (see Definition~\ref{d:regmps}).
The following theorem is a corollary of Theorem~\ref{thm_better_bound_gener}, which is the technical version of Theorem~\ref{thm_presc_sep}. 
\begin{btheorem}\label{thm:cntinuum}
There exists an uncountable family of amenable groups of asymptotic dimension one $\left(G_r\right)_{r\in \R}$ such that for any $r\neq s$ there is no regular map from $G_s$ to $G_r$.
\end{btheorem}
Let us recall that quasi-isometric and coarse embeddings are regular maps. As stated above, this result is new. See Hume~\cite[Theorem 1.2]{hume2017} for an analog statement, with $C'(1/6)$ small cancelation groups. Our proof will use the following fact:
\begin{fact}\label{fact_continuum}
	Let $g$ be a function satisfying the hypothesis of Theorem~\ref{thm_presc_sep}. Then, there exists a sequence of integers $(v_n)_{n\ge0}$ such that the following is true: for any function $f$ satisfying the assumptions of Theorem~\ref{thm_presc_sep} and such that $f\ge g$, there exists a group $\Delta_f$ and a sequence of integers $(u_m)_{m\ge0}$ such that:
\begin{itemize}
	\item $\Pi_{\Delta_f,p}(n)\leq\kappa_1\frac{n}{f(\log n)}$ for any $ n $,
	\item $\Pi_{\Delta_f,p}(u_m)\geq4^{-p}\kappa_2\frac{u_m}{f(\log u_m)}$ for any $m$ and $p\in\left[1,\infty\right)$,

	\item  for any large enough integer $n$, there exists an integer $m$ such that $u_m\in\left[v_n,v_{n+1}\right]$.
\end{itemize}
\end{fact}

This fact relies on the proof of Theorem~\ref{thm_better_bound_gener}. We refer the reader to Remark~\ref{remkfactcontinuum} for details.

\begin{proof}[Proof of Theorem~\ref{thm:cntinuum}]
We will use a well known process, that comes at least from Grigorchuk~\cite[Theorem B.1, statement 4]{MR764305}. Let $(v_n)_{n\ge0}$ be a sequence satisfying the lower bounds on the Poincar\'e profiles of Theorem~\ref{thm_presc_sep} for $\rho=\log$. Up to extracting a subsequence, we can assume that we have, for any $n$,

\[\label{e_suite_extraite}\stepcounter{equation}\tag{\theequation}
\log(v_{n+1})\le\left(\log v_{n}\right)^2.\] Let $f_0=(\log n)^2$ and $f_1=(\log n)^3$. For any sequence $(\omega_n)_{n\ge0}\in\set{0,1}^{\N}$, we claim that there exists a function $\rho_\omega$ such that for any $n\ge0$ and any $x\in\left[v_{2n},v_{2n+1}\right]$, we have $\rho_\omega(x)=f_{\omega_n}(x)$, and satisfying the assumptions of Theorem~\ref{thm_presc_sep}. To construct such a function, one just need to say what needs to be done when $\omega_n$ changes of value:
\begin{itemize}
	\item If $\omega_n=0$ and $\omega_{n+1}=1$, then one can set $\rho_\omega(x)=\min\set{{\frac{\log^4 x}{(\log v_{2n+1})^2},\log^3x}}$, for every $x\in\left[v_{2n+1},v_{2n+2}\right]$.
	\item If $\omega_n=1$ and $\omega_{n+1}=0$, then one can set $\rho_\omega(x)=\max\set{(\log v_{2n+1})^2\log x,\log^2x}$, for every $x\in\left[v_{2n+1},v_{2n+2}\right]$.
\end{itemize}
The assumption~\eqref{e_suite_extraite} on the sequence $(v_n)_{n\ge0}$ ensures that this gives a well-defined function, satisfying the assumptions of Theorem~\ref{thm_presc_sep}, and such that~$\rho_\omega\ge\rho=\log$. Then, for each sequence $(\omega_n)_{n\ge0}$, we obtain a group $\Delta_\omega$ from Theorem~\ref{thm_presc_sep}.
Each $\Delta_{\omega}$ is a finitely gererated amenable group of asymptotic dimension one. 

If, for some sequences $\omega$ and $\omega'$, there exists a regular map from $\Delta_\omega$ to $\Delta_{\omega'}$, then, from the monotonicity of Poincar\'e profiles (see~\ref{s:regmps}), we have $\Pi_{\Delta_\omega,1}\preceq \Pi_{\Delta_{\omega'},1} $. From the conclusion of Theorem~\ref{thm_presc_sep}, and Fact~\ref{fact_continuum}, this implies that we have $\omega_n\le\omega'_n$, for any large enough $n$.\\

Equivalently, for each subset $N\subset\N$, we can consider the associated sequence $(\omega_n)_{n\ge0}\in\left\{0,1\right\}^\N$ and we get a group that we call $\Delta_N$. From the preceeding, if there is a regular map from $\Delta_N$ to $\Delta_{N'}$, this implies that $N\setminus N'$ is finite, and each $\Delta_{N}$ is a finitely gererated amenable group of asymptotic dimension one. 

Following Hume~\cite{hume2017}, there exists a family $\mathcal{N}$ of $2^{\aleph_0}$ subsets of $\N$ with $M\setminus N, N\setminus M$ infinite for all distinct $M,N\in\mathcal{N}$. Then, the family of groups $(\Delta_{N})_{N\in\mathcal N}$ satisfies that there exists no regular map from $\Delta_{N}$ to $\Delta_{M}$, for all distinct $M$ and $N$.
\end{proof}
\paragraph{Embeddings in products of trees}
Dranishnikov showed in~\cite{dranishnikov-2003} that any bounded degree graph can be coarsely embedded in a finite product of trees. Until now, the issue of knowing whether these trees can be chosen of bounded degree or not remained open. Theorem~\ref{thm_presc_sep} is able to give a negative answer, see the statement below.
\begin{btheorem}
	There exist bounded degree graphs of asymptotic dimension one that do not coarsely embed in any finite product of bounded degree trees.
\end{btheorem}
\begin{proof}
We recall that the $L^1$-Poincar\'e profile is equivalent to the separation profile. A finite product of bounded degree trees has a separation profile bounded above by $ \frac{n}{\log(n)} $ (see~\cite[Theorem 3.5]{benjaminischrammtimar2012}). Taking any function $ \rho $ that is dominated by the identity function on $ \R_{\geq 1} $, for example $\log(x)$, the separation profile of the group given by Theorem~\ref{thm_presc_sep} dominates $ \frac{n}{\log(n)} $ along a subsequence. Since the separation profile is monotone under coarse embeddings (\cite[Lemma 1.3.]{benjaminischrammtimar2012}), this group cannot be embedded with a coarse embedding in any finite product of bounded degree trees.
\end{proof}
\paragraph{Embeddings in $L^p$ spaces}
Given a graph $\Gamma$, say on $n$ vertices, one can study how it can be embedded in $L^p$ spaces. For any injective map $F\colon V\Gamma\hookrightarrow  L^p$, we define the \textbf{distortion} of $F$ as:
\[\dist F = \sup_{a\neq b}\frac{d(a,b)}{\delta(F(a),F(b))}\sup_{a'\neq b'}\frac{\delta(F(a'),F(b'))}{d(a',b')}, \]
where $d$ and $\delta$ denote the distance in $\Gamma$ and in $L^p$, respectively. We then can define $c_p\vcentcolon=\inf\set{\dist(F)\mid F\colon V\Gamma\hookrightarrow L^p}$.

Bourgain showed in~\cite{MR815600} that $c_p$ is bounded by $O(\log n)$. It was proved that this is optimal for families of expander graphs~\cite{Mat-97-ExpandersLpp,MR1337355}. This was improved by Rao~\cite{MR1802217} to $O(\sqrt{\log n})$ in the case of planar graphs. Since any family of planar graphs is hyperfinite~\cite{liptontarjan2}, it is natural to ask if this bound is also valid for hyperfinite graphs. Recall that a sequence of bounded degree graphs $(G_n)$ is called \textbf{hyperfinite} if for any $\epsilon>0$ there exists $K>0$ such that for each $n\geq 1$, there exists a set $Z_n\subset VG_n$, with $\card{Z_n}\leq\epsilon\card{VG_n}$, such that $G_n\setminus Z_n$ consists of components of size at most $K$. This notion of hyperfiniteness was introduced by Elek in~\cite{MR2406929}. This question was posed to us by G\'abor Pete, also motivated by the fact that that planar graphs conjecturally embed in~$L_1$ with~$O(1)$ distortion~\cite{GNRS}. Theorem~\ref{thm_presc_sep} is able to give a negative answer (see below). To our knowledge, this statement is new.
\begin{btheorem}
For any $\epsilon\in(0,1)$, there exists a hyperfinite sequence of bounded degree graphs~$(\Gamma_n)_{n\geq0}$, such that for any~$p\in\left[1,\infty\right)$ there is a positive constant~$K'$ depending only on~$p$ such that for any~$n$,
\[c_p(\Gamma_n)\geq K' (\log\card{\Gamma_n})^{1-\epsilon}.\]
\end{btheorem}
This follows from the lemma below.
\begin{lemma}
For any non-decreasing function $\rho\colon\R_{\geq 1}\rightarrow\R_{\geq 1}$ such that $ \frac{x}{\rho(x)} $ is non-decreasing and $\lim_\infty \rho=\infty$, there exists a hyperfinite sequence of bounded degree graphs $(\Gamma_n)_{n\geq0}$, such that for any $p\in\left[1,\infty\right)$ there is a positive constant $K'$ depending only on $p$ such that for any $n$,
\[c_p(\Gamma_n)\geq K'\frac{\log\card{\Gamma_n}}{\rho(\log\card{\Gamma_n})}.\]
\end{lemma}
\begin{proof}
Let $\Delta$ be the group associated with $\min(x, \sqrt\rho)$, given by Theorem~\ref{thm_presc_sep_approx}. Then there exists a sequence $(\Gamma_n)_{n\geq0}$ of subgraphs of $\Delta$ such that for any $n\geq0$,
\[h_p(\Gamma_n)\geq\frac{4^{-p}\kappa_{1}}{\rho(\log\card{\Gamma_n})}.\]
Using~\cite[Theorem 1.1]{jolissaintvalette2014} together with~\cite[Proposition 3.3]{jolissaintvalette2014}, there exists a positive constant $K'(p)$ such that for any $n\geq0$,
\begin{align*}
c_p(\Gamma_n)&\geq K'(p)\log\card{\Gamma_n}h_p(\Gamma_n)\\
&\geq K(p)\frac{\log\card{\Gamma_n}}{\rho(\log\card{\Gamma_n})},\quad\text{with $K(p)=4^{-p}\kappa_1K'(p)$.}
\end{align*}
The sequence~$(\Gamma_n)_{n\geq0}$ is made of finite subgraphs of a Cayley graph of an amenable group. Then, from~\cite[Theorem~2]{elektimar11}, it is hyperfinite.\footnote{the fact that $\Delta$ has asymptotic dimension one also implies that the sequence~$(\Gamma_n)_{n\geq0}$ is hyperfinite (again from~\cite[Theorem~2]{elektimar11}).}
\end{proof}
\paragraph{Upper bounds on Poincar\'e profiles}

We say that a graph $G$ has a compression exponent $\alpha$ in $L^p$ if there exists a $1$-Lipschitz map $F\colon G\to L^p$ such that $\rho_F(t)\preceq t^\alpha$. Theorem~\ref{thm_borne_sup1} implies:

\begin{corollary}
Assume $G$ is a graph with bounded degree and compression exponent $\alpha$ in some $L_p$-space. Then there is a constant $K(p)$ so that
\[\Pi_{G,p}(n)\leq K\frac n{(\log n)^{\alpha}}.\]
\end{corollary}
Compression exponents have been widely studied, see for example~\cite{gournaylecoz} for a tabular summarizing known results.
\paragraph{Organization of the paper}
In \cpt{Chapter}{Section}~\ref{s_poincprof}, we give the definitions of Poincar\'e and separation profiles, and give comparison theorems, following~\cite{humemackaytessera}.
In \cpt{Chapter}{Section}~\ref{s_constructiondelta}, we give the construction of the groups $\Delta$, following~\cite{brieusselzheng2015}.
In \cpt{Chapter}{Section}~\ref{s_low_bd_sep}, we prove the lower bounds on the separation profile of the groups $\Delta$, and make a general study of Cartesian powers of graphs (\cpt{section}{subsection}~\ref{s:cheegercartesian}). In \cpt{Chapter}{Section}~\ref{s_up_bd}, we prove upper bounds on the Poincar\'e profiles using compression in $L^p$ spaces. Finally, in \cpt{Chapter}{Section}~\ref{s:comparison}, we prove Theorem~\ref{thm_better_bound_gener}, that generalizes Theorems~\ref{thm_presc_sep}, \ref{thm_presc_sep_approx} and~\ref{thm_better_bound}, by comparing the two bounds obtained in \cpt{Chapters}{Sections}~\ref{s_low_bd_sep} and~\ref{s_up_bd} in the case of the groups $\Delta$.

In Appendix~\ref{s_chgrcnstdstrtdgfphs}, we consider generalisations of the study of the separation of distorted graphs, with three methods: combinatorics, geometric, and analytic.
\paragraph{Acknowledgements}
The author would like to thank Romain Tessera who initiated this project and gave the idea of using Lipschitz embeddings to get upper bounds on Poincar\'e profiles, and Jérémie Brieussel who helped him understand more deeply the diagonal lamplighter groups. The author is also grateful to Tianyi Zheng for interesting discussions about these groups and to David Hume and Gabor Pete for discussions about applications of Theorem~\ref{thm_presc_sep}.
\section{Definitions}\label{s_poincprof}
In this \cpt{chapter}{section}, we give the basic definitions of Poincar\'e and separation profiles. We give comparison theorems, following~\cite[Sections 6 and 7]{humemackaytessera}.

The set of vertices of a graph $\Gamma$ will be denoted $V\Gamma$, while the set of edges will be written $E\Gamma$. Each edge is considered as a subset of $V\Gamma$ of cardinality $2$, which means that they are not oriented and that we do not allow self-loops.

A graph will always be considered as a set of vertices endowed with the shortest path metric. We ignore the ``points'' of the edges.
\subsection{Poincar\'e profiles}
\subsubsection*{Definition of $L^p$-Poincar\'e profiles}
We start with the definition of $ L^{p} $-Cheeger constants and Poincaré profiles.
\begin{definition}\label{d:lpchgrcnst}
	Let $ \Gamma $ be a finite graph. We define for any $ p\geq 1 $ the \textbf{$ L^p $-Cheeger constant} of $\Gamma$ as:
		\[ h_p(\Gamma) = \inf\set{\frac{\left\|\nabla f \right\|_p}{\left\|f - f_{\Gamma}\right\|_p} \colon f\in \Map(V\Gamma \rightarrow \R), \norm{f-f_\Gamma}\not\equiv 0},\]
		
		with $\card{\nabla f} (g) = \sup_{h,h'\in B(g,1)}\card{f(h) - f(h')}\text{ and }  f_{\Gamma}\vcentcolon=\card{V\Gamma}^{-1}\sum_{g\in V\Gamma}f(g).$

\noindent Let $ G $ be an (infinite) graph. Following~\cite{humemackaytessera}, we define the \textbf{$L^p$-Poincar\'e profile} of $ G $ as 
		\[  \Pi_{G,p}(n) = \sup \left\{\left |V\Gamma\right | h_p\left(\Gamma\right) \vcentcolon \Gamma \subset G, \left| V\Gamma \right|\leq n \right\}.\]
	\end{definition}
\subsubsection*{Interpretation of the $L^1$-Poincar\'e profile}
The $L^1$-Cheeger constant can be reinterpreted as the \textit{minimum isoperimetric ratio}, this is the purpose of this paragraph.
%
\begin{definition}\label{d:majcmbntrlchgrcnst}
	For any finite graph $\Gamma$, we define the \textbf{majored combinatorial Cheeger constant} of $\Gamma$ as
	\[\tilde{h}(\Gamma)=\inf\frac{|\tilde\partial A|}{\card{A}},\]
	where the infimum is taken on the subsets $A$ of $V\Gamma$ of size at most $\frac{\card{V\Gamma}}2$, and $\tilde\partial A$ is the boundary of $A$ defined by the set of vertices that are either in $V\Gamma\setminus A$ and at distance $1$ from $A$, or in $A$ and at distance $1$ from $V\Gamma\setminus A$.
\end{definition}
This majored combinatorial Cheeger constant is strongly related with the $L^1$-Cheeger constant, see proposition below.
\begin{proposition}\label{p:cheeger1}(\cite[Proposition 6.10]{humemackaytessera})
	Let $\Gamma$ be a finite graph. Then
	\[h_1(\Gamma)\leq \tilde{h}(\Gamma)\leq2h_1(\Gamma)\]
\end{proposition}
\begin{remark}
Our gradient is calculated ``at scale 1'', while \cite[Proposition 6.10]{humemackaytessera} concerns gradient at scales $a\geq 2$. However, in the context of graphs, it is easy to check that it is allowed to take $a=1$.
\end{remark}
\subsubsection*{Comparison of $L^1$ and $L^p$-Poincar\'e profile}
Hume, Mackay \& Tessera showed a lower bound on the $ L^p $-Cheeger constants depending on the $ L^1 $-Cheeger constant (\cite[Proposition 7.2]{humemackaytessera}). Working all the constants of their proof, we get the following statement.
\begin{proposition}\label{prop:monotonep}(from~\cite[Proposition 7.2]{humemackaytessera})
	Let $\Gamma$ be a finite graph with at least $ 3 $ vertices. Then for any $ p\in\left[1,\infty\right) $, we have:
	$$h_p(\Gamma)\geq \min\left(\frac1{12},\frac{4^{-p}}2\right) h_1(\Gamma).$$
		Let $ G $ be an infinite graph. Then for any $ p\in\left[1,\infty\right) $,
	\[ \Pi_{G,p}\geq \min\left(\frac1{12},\frac{4^{-p}}2\right) \Pi_{1,G}.\]
\end{proposition}
We can mention that, on the other hand, we have the following comparison theorem:
\begin{proposition}\cite[Proposition 6]{humemackaytessera}
If $\Gamma$ is a finite graph and $p\in[1,\infty)$, then
\[h^p(\Gamma)^p\leq 2^{p}h^1(\Gamma).\]
\end{proposition}
\subsection{Regular maps}\label{s:regmps}
Poincar\'e profiles have the nice property to be monotone under coarse embeddings and more generally under regular maps, see definition and theorem below.
\begin{definition}\label{d:regmps}
A map $F\colon VX\to VY$ between bounded degree graphs is said to be \textbf{regular} if there exists a constant $\kappa$ such that
\begin{itemize}
	\item $d(f(x),f(x'))\le\kappa d(x,x')$, for every $x,x'\in X$,
	\item and $\card{f^{-1}(\{y\})}\le\kappa$, for every $y\in Y$.
\end{itemize}
\end{definition}
Any coarse embedding is a regular map. The absolute value $\Z\to\N$ is an example of a regular map that is not a coarse embedding.
\begin{theorem}\label{thm:mntncty_regmps}
Let $X,Y$ be graphs with bounded degree. If there is a regular map $f\colon VX\to VY$, then for all $p\in[1,\infty]$, there exists $K$ depending only on $p$ such that 
\[\Pi_{X,p}(n)\leq K\Pi_{Y,p}(Kn),\quad\text{for any large enough $n$}.\]
\end{theorem}
Thus, for each $p\in[1,\infty]$, the growth type of the $L^p$-Poincar\'e profiles of the Cayley graphs of a finitely generated group $G$ do not depend on the chosen finite generating set.
\subsection{Separation profile}
	Poincar\'e profiles came up  as a generalization of the separation profile defined by Benjamini, Schramm \& Tim\'ar~\cite{benjaminischrammtimar2012}. We give here the definition of this profile, and his relation with Poincar\'e profiles.
	\begin{definition}\label{d:sprtnprfl}
For a finite graph $\Gamma'$, let $L(\Gamma')$ be the size of any largest component of $\Gamma'$. We first define the \textbf{$ \epsilon $-cut} of a finite graph $ \Gamma $ as \[ \cut^{\epsilon}\Gamma\vcentcolon=\min\set{|S| \colon S\subset V\Gamma\text{ and }\card{L(\Gamma-S)}\leq\epsilon\card{V\Gamma}}.\]
(we omit the ``$\epsilon$'' for $\epsilon=1/2$.)

For an infinite graph $G$, the \textbf{separation profile} is defined as
			$$ \sep_G(n) \vcentcolon= \sup\set{\cut^{1/2}\Gamma\colon\Gamma\subset G\text{ and }\card{\Gamma}\leq n}. $$
	\end{definition}
	It corresponds to the Poincar\'e profile with $p=1$, from the proposition below.
\begin{proposition}\label{prop:equivseppoinc}(from~\cite[Proposition 6.5]{humemackaytessera})
	Let $G$ be an (infinite) graph, and $ D $ be a bound on the degrees of the vertices of $ G $. Then for $n\geq 2$,
$$\frac18\sep_G(n)\leq\Pi_{G,1}(n)\leq4(D+1)\sep_G(n).$$
\end{proposition}
\begin{proof}
From~\cite[Proposition 2.2]{hume2017} and Lemma~\ref{p:cheeger1}, for any graph $\Gamma$ with at least $2$ vertices, we have
\[\cut\Gamma\geq\frac1{4(D+1)}h_1(\Gamma)\card{\Gamma},\]
and the right-hand side follows.

From~\cite[Proposition 2.4]{hume2017} and Lemma~\ref{p:cheeger1}, for any graph $\Gamma$ with at least $2$ vertices, there exists $\Gamma'\subset\Gamma$ satisfying
\[\card{\Gamma'}h_1(\Gamma')\geq\frac18\cut\Gamma,\]
and the left-hand side follows.
\end{proof}
Combining Propositions~\ref{prop:monotonep} and~\ref{prop:equivseppoinc}, we deduce:
\begin{theorem}\label{t:comppoincsep}
Let $ G $ be an infinite graph. Then for any $p\in\left[1,\infty\right)$
\[\Pi_{G,p}\geq\min\left(\frac1{96},\frac{4^{-p}}{24}\right)\sep_G.\]
\end{theorem}

\section{Construction of lamplighter diagonal products}\label{s_constructiondelta}
We write here the construction of lamplighter diagonal products, following~\cite{brieusselzheng2015}. We start with some definitions.
\begin{definition}
Let $\Gamma$ be a group. We denote by $1_\Gamma$ the identity element of $\Gamma$. For any function $ f\colon \Z\rightarrow\Gamma $, we define the support of $f$ by $\support(f)=\set{j \in \Z \mid f (j) \neq 1_\Gamma}$. We denote by $\Gamma^{\left(\Z\right)}$ the set of functions $\Z\rightarrow\Gamma $ with finite support.

There is a natural action of $\Z$ on $\Gamma^{\left(\Z\right)}$, by translation on the indices: for any $i\in\Z$ and $f\in\Gamma^{(\Z)}$, we define $i.f$ so that $(i.f)_x=f_{x-i}$ for any $x\in\Z$.

We define the \textbf{wreath product} of $\Gamma$ on $\Z$, denoted by $\Gamma \wr \Z$, as the semi-direct product $\Gamma^{(\Z)}\rtimes\Z$. An element of $ \Gamma\wr\Z $ is represented by a pair $ (f, i) $; we refer to $ f $ as the lamp configuration and to $ i $ as the position of the cursor. The product rule is:
\[ (f, i)(g, j) = (h, i + j),\quad\text{with $h_x=f_x g_{x-i}$ for every $x\in\Z$.}\]
\end{definition}
This group is also called the \textbf{lamplighter group} of~$\Gamma$ over~$\Z$. 
\begin{definition}
Let $\Gamma$ be a group, For any $ g\in\Gamma_s$ and $ i\in\Z $, we define the \textbf{$ g $-dirac function at $i$}, denoted by $g \delta_i $, as:
\begin{align*}
g\delta_i\colon\Z&\to\Gamma\\
n&\mapsto \left\{
\begin{array}{lcl}
 g&\text{if }n=i,\\
1_\Gamma&\text{otherwise.}&
\end{array}\right.
\end{align*}
\end{definition}
\begin{definition}
	Let $G$ be a group. Let $ \left(G_i\right)_{i\in I} $ be a family of groups and such that there exists, for any $i\in I$, a surjective homomorphism $\pi_i\colon G\twoheadrightarrow G_{i} $. We define the \textbf{diagonal product} of $ \left(G_i\right)_{i\in I} $ with respect to $ \left(\pi_{i}\right)_{i\in I} $ as the quotient group $ G/\cap_{i\in I}\ker(\pi_i) $.
\end{definition}
Let $ A $ and $ B $ be two (non trivial) finite groups. Let $(\Gamma_s)_{s\geq0}$ be a sequence of groups such that, for any $s\geq 0$, $ \Gamma_s $ possesses two subgroups $A_s$ and $B_s$ respectively isomorphic to $A$ and $B$, such that $A_s\cup B_s$ generates $\Gamma_s$.

For any $s\geq 0$, let $a_s\colon A\to A_s$ and $b_s\colon B\to B_s$ be two group isomorphisms, and $k_s$ be a non-negative integer.

Let $\textbf{G}$ be the free product of $ A $, $ B $ and $ \Z $, and let $\tau\in\textbf{G}$ be a generator of the copy of $\Z$. Let us fix $s\geq0$. We denote by $\Delta_s$ the wreath product $\Gamma_s\wr\Z$. There exists a unique surjective homomorphism $\pi_s\colon\mathbf{G}\to\Delta_s$ such that
\begin{itemize}
\item $\pi_s(a)=(a_s(a)\delta_{-k_s},0)$ for any $a\in A$\footnote{In~\cite{brieusselzheng2015}, $\pi_s(a)$ is defined as $(a_s(a)\delta_{0},0)$ instead of $(a_s(a)\delta_{-k_s},0)$. However, up to a factor $2$ on $k_s$ we obtain the same group.}.
\item $\pi_s(b)=(b_s(b)\delta_{k_s},0)$ for any $b\in B$,
\item and $\pi_s(\tau)=(1_{\Gamma_s},1).$
\end{itemize}
The symmetric set $\pi_s(A)\cup\pi_s(B)\cup\pi_s(\tau^{\pm1})$ generates the group $\Delta_s$. We can detail how each element of this generating set acts by right-translation. Let $(f,i)\in\Delta_s$.
\begin{itemize}
\item If $a\in A$, then $(f,i).\pi_s(a)=(g,i)$, with $g$ satisfying $g_{i-k_s}=f_{i-k_s}a_s(a)$ and $g_x=f_x$ if $x\neq i-k_s$. In words, we ``write'' $a$ at $i-k_s$.
\item If $b\in B$, then $(f,i).\pi_s(b)=(g,i)$, with $g$ such that $g_{i+k_s}=f_{i+k_s}b_s(b)$ and $g_x=f_x$ if $x\neq i+k_s$. In words, we ``write'' $b$ at $i+k_s$.
\item $(f,i).\pi_s(\tau^{\pm1})=(f,i\pm1)$.
\end{itemize}
\begin{definition}\label{d_lamp_diag_prod}
We define the associated \textbf{lamplighter diagonal product $\Delta$} as the diagonal product of the sequence $ \left(\Delta_s\right)_{s\geq 0} $ with respect to $ \left(\pi_{s}\right)_{s\geq 0} $, \textit{i.e.} $\Delta$ is the quotient group \[\Delta=\textbf{G}/{\cap_{s\geq 0}\ker(\pi_s}).\]
\end{definition}
\begin{assumption}
	Let~$(\Gamma_s,a_s,b_s,k_s)_{s\geq0}$ and~$(\pi_s)_{s\geq0}$ be as above. We we always assume that the following conditions are satisfied:
	\begin{itemize}
	\item the sequence $(k_s)_{s\geq0}$ satisfies $k_0=0$, and $ k_{s+1} > 2 k_s $ for every $s\geq0$.
	\item for every $s\geq0$, the group $ A_s \times B_s $ is a quotient of $ \Gamma_s $, \textit{i.e.} $ \Gamma_s / \left[A_s,B_s\right]^{\Gamma_s} $ is isomorphic to $ A_s\times B_s$.
	\end{itemize}
\end{assumption}
The first assumption is an independence property between the quotients $\left(\Delta_s\right)_{s\ge0}$ of $\Delta$. The second assumption is more sutle and restrictive. It ensures the existence of projection maps $ \Gamma_s \to \Gamma_s / \left[A_s,B_s\right]^{\Gamma_s} \simeq A\times B $ that plays a role in proving local finitess properties, see Paragraph~2.2.2. of~\cite{brieusselzheng2015} for details.

From the definition of diagonal products, an element of $ \Delta $ is totally determined by its projections on the quotients $ \Delta_s $. Moreover, given an element of $ \Delta $, the position of the cursor in each of these projections is constant. Therefore we will denote the elements of $ \Delta $ by $ \left(\left(f_s \right)_{s\geq 0},i \right) $, where $i\in\Z$ and $ f_s\colon\Z\to\Gamma_s$ is a finite support map, for each $s\geq0$.

	Let $\pi$ the canonical projection map from $\textbf{G}$ to $\Delta$. Due to its quotient structure, the group $ \Delta $ has the following universal property:
\begin{proposition}
For any group homomorphism $f\colon G \rightarrow X$ such that $\cap_{s\geq 0} \ker \pi_s \subset \ker f$, there exists a unique group homomorphism $\tilde{f}\colon \Delta \rightarrow G$ such that $f = \tilde{f}\circ \pi$.
\end{proposition}

\begin{example}\label{e_lafforgue}
	An example of a family of groups satisfying the conditions above is the Lafforgue super expanders~\cite{lafforgue2008}. For any prime number $q$, let $A=\Z_q^2$, $B = \Z_3$, $\Gamma_0=A\times B$, and, for every $s\ge1$, $\Gamma_s $ be the diagonal product of $ \SL_3(\mathbf{F}_q\left[ X \right]/(X^s-1))$ and $ A \times B $, with respect to the following surjective homomorphisms:
\[\pi_{1}\colon A*B\twoheadrightarrow A\times B,\]
and
\[\pi_{2}\colon A*B\twoheadrightarrow \SL_3(\mathbf{F}_q\left[ X \right]/(X^s-1)), \]
	where $\pi_{2}$ is defined with the following identifications:
	\[\Z_q^2\simeq\left<
	\begin{pmatrix}
	1 & 1 & 0 \\
	0 & 1 & 0 \\
	0 & 0 & 1
	\end{pmatrix},
	\begin{pmatrix}
	1 & X & 0 \\
	0 & 1 & 0 \\
	0 & 0 & 1
	\end{pmatrix}\right>,\text{ and }
	\Z_3\simeq\left <
	\begin{pmatrix}
	0 & 0 & 1 \\
	1 & 0 & 0 \\
	0 & 1 & 0
	\end{pmatrix}\right >.
	\]
Then, $(\Gamma_s)_{s\ge1}$ satisfies the above properties, with $A=\Z_q^2$ and $B=\Z_3$.\\

This example is important because the sequence $(\Gamma_s)_{s\ge1}$ is an expander. This will be used in applications. For simplicity, we denote by $(\Gamma_s)_{s\ge1}$ the sequence $(\Cay(\Gamma_s,A_s\cup B_s))_{s\ge1}$, which is a sequence of regular graphs. We have the following theorem,
\begin{theorem}\label{thm:lffrgxpndrs}\cite{lafforgue2008}
There exist $D,\epsilon>0$ such that for every $s\ge1$,
\begin{itemize}
\item $\tilde h(\Gamma_s)>\epsilon$,
\item $\deg \Gamma_s\le D$,
\item $(\card{\Gamma_s})_{s\ge1}$ is unbounded.
\end{itemize}
\end{theorem}
\end{example}

\section{A lower bound on Poincar\'e profiles}\label{s_low_bd_sep}
The goal of this \cpt{chapter}{section} is to give a lower bound on the Poincar\'e profiles of diagonal lamplighter products. We fix a diagonal product of lamplighter groups $\Delta$, keeping the same notations as above. We show the following theorem:
\begin{theorem}
	\label{thm_low-bnd-sep}%
	Let $ \Delta $ be the lamplighter diagonal product of $(\Gamma_s,a_s,b_s,k_s)_{s\geq 0}$. Then for any $ s\geq 0 $ and $ r\leq k_s/2 $,
\[\Pi_{\Delta,p}((2k_s+2r+1)\card{\Gamma_s}^{2r+1})\geq 4^{-p}\frac{h(\Gamma_s)^2}{1536(\deg\Gamma_s)^2}\frac{\card{\Gamma_s}^{2r+1}}{2r+1}
.\]
\end{theorem}
This theorem is the technical core of the lower bounds obtained in Theorems~\ref{thm_presc_sep}, \ref{thm_presc_sep_approx} and~\ref{thm_better_bound}, that will be proved in \cpt{Chapter}{Section}~\ref{s:comparison}. To show it, we will exhibit subgraphs, that we call \textbf{distorted lamp groups}, and study their separation. We will make a comparison with Cartesian powers of finite graphs, that will play the role of model graphs. The lower bound will finally be extended to Poincar\'e profiles using Theorem~\ref{t:comppoincsep}. We start with a general study of Cartesian powers of a given finite graph.
\subsection{Cheeger constants of Cartesian powers of a given graph}\label{s:cheegercartesian}
	Here, we will consider sequences of graphs of unbounded maximal degree. We will use another definition of Cheeger constants, that is more relevant in this context, see definition and proposition below.
	\begin{definition}\label{d:cbntrlchgrcnst}
	For any finite graph $\Gamma$, we define the \textbf{combinatorial Cheeger constant} of $\Gamma$ as
	\[h(\Gamma)=\inf\frac{\card{\partial A}}{\card{A}},\]
	where the infimum is taken on the non-empty subsets $A$ of $V\Gamma$ of size at most $\frac{\card{V\Gamma}}2$, and $\partial A$ is the boundary of $A$ defined as the set of vertices of $V\Gamma\setminus A$ and at distance $1$ from $A$.
	\end{definition}
Mind the difference with the \textit{majored} combinatorial Cheeger constant $\tilde{h}(\Gamma)$ of Definition~\ref{d:majcmbntrlchgrcnst}, where the boundary includes more vertices. This definition is motivated by the following proposition:
\begin{proposition}\cite[Proposition 2.2]{hume2017}\label{p:cmprsn_chgrct}
For any graph $\Gamma$ with at least $2$ vertices,
\[\cut(\Gamma)\geq \frac14h(\Gamma)\card{\Gamma}.\]
\end{proposition}
This statement should be compared with Proposition~\ref{prop:equivseppoinc}, where the maximal degree of the graph appears in the inequality. Proposition~\ref{p:cmprsn_chgrct} is more relevant here, as we work in an unbounded degree context. We have the following comparison between these two combinatorial Cheeger constants:
\begin{proposition}\label{p:cmpcmbntrlchgrcnstns}
Let $\Gamma$ be a finite graph of maximal degree $D$. Then,
\[h(\Gamma)\leq\tilde{h}(\Gamma)\leq (D+1)h(\Gamma)\]
\end{proposition}
We will also use the notion of spectral gap.
\begin{definition}
If $\Gamma$ is a finite graph, we can define the \textbf{Laplacian} $\Delta_\Gamma$ as the operator of $\ell^2(V\Gamma)$ satisfying:
\[\Delta_\Gamma f(i)=\sum_{j\sim i}f(i) - f(j),\]
for every $f\in\ell^2(V\Gamma)$ and $i\in V\Gamma$. We denote by $\lambda_2(\Gamma)$ the second smallest eigenvalue of $\Delta_\Gamma$, called the \textbf{spectral gap} of $\Gamma$.
\end{definition}
Spectral gaps and Cheeger constants are related by the Cheeger inequalies.
\begin{theorem}[the Cheeger inequalities]\label{t:chgrinqlts}
Let $\Gamma$ be a finite regular graph of degree $D$. Then
\[\frac{h(\Gamma)^2}{2D}\leq\lambda_2(\Gamma)\leq2Dh(\Gamma).\]
\end{theorem}
See~\cite[Lemma 2.1, Theorem 2.2]{chung1997}, and~\cite[Lemma 2.4]{alon1986} for detail.
\begin{definition}
	Let $G$ and $H$ be two graphs. We define the \textbf{Cartesian product} of $G$ and $H$, denoted by $G\times H$, as the graph with vertex set $VG\times VH$ satisfiying that $(g,h)$ and $(g',h')$ are linked with an edge if and only if: $\set{g,g'}\in EG $ and $h=h'$, or $g=g'$ and $\set{h,h'}\in EH$.
\end{definition}
	The following proposition gives lower and upper bounds on Cheeger constants of Cartesian powers of a given graph.
\begin{proposition}\label{p:prop_prod}
	Let $ G $ be a finite connected regular graph. Let $ k $ be a positive integer and $ G^{k} = \underbrace{G\times \dots \times G }_{k \text{ terms} }$ the Cartesian product of $k$ copies of $ G $. Then we have
	\[ \frac{a}{k} \leq h(G^k) \leq  \frac{b}{\sqrt{k}},\]
	with $ a = \left(\frac{h(G)}{2\deg G}\right)^2 $ and $ b =(2\sqrt{2}+2)\sqrt{\deg(G)h(G)}.$
\end{proposition}
From Proposition~\ref{p:cmprsn_chgrct}, we obtain the following lower bound for the separation of Cartesian powers of a given graph:
\begin{corollary}\label{c:prop_prod}
Let $ G $ be a finite connected regular graph with at least $ 2 $ vertices. Let $ k $ be a positive integer. Then,
	\[ \cut(G^k) \geq \frac{h(G)^2}{16(\deg G)^2}\frac{\card{G}^{k}}{k}.  \]
\end{corollary}
\begin{remark}
	The $ k $ in the denominator will have an impact in \cpt{Chapter}{Section}~\ref{s:comparison} where we compare the lower and upper bounds obtained on the Poincar\'e profiles of the groups $ \Delta $. Without this term, the upper and lower bounds of Theorem~\ref{thm_presc_sep_approx} would match each other. However, the upper bound in Proposition~\ref{p:prop_prod}, and the equivalence between Cheeger constants and cuts from~\cite{hume2017}, show that such a loss is probably unavoidable.
\end{remark}
\begin{proof}[Proof of Proposition~\ref{p:prop_prod}]
We will use the following equality, from the statement 3.4 of Fiedler~\cite{fiedler1973}:
	\[\stepcounter{equation}\tag{\theequation}\label{lambda1}\lambda_2(G^k)=\lambda_2(G).\]

We start with the lower bound. The degree of the graph $G^k$ is $k\deg G$. From the Cheeger inequalities (Theorem~\ref{t:chgrinqlts}), we have
\[\stepcounter{equation}\tag{\theequation}\label{cheeger2}
h(G^k)\geq\frac{\lambda_2(G^k)}{2k\deg G}
\text{ and }
\lambda_2(G)>\frac{h(G)^2}{2\deg G}.
\]
	Combining \eqref{lambda1} and \eqref{cheeger2}, we get $ h(G^k) \geq \frac1k\left(\frac{h(G)}{2\deg G}\right)^2$.

Let us prove the upper bound. In~\cite{bobkovbhoudretetali2000}, Bobkov, Houdr\'e and Tetali introduced another spectral quantity called $ \lambda_{\infty} $ that is linked with the vertex isoperimetry. It is defined by 
\[
\lambda_\infty(\Gamma)=2\inf_{f\colon V\Gamma\to\R}\frac{\frac1n\sum_{i\in V\Gamma} \sup_{j\sim i}(f(i)-f(j))^2}{\frac1{n^2}\sum_{i,j\in V\Gamma}(f(i)-f(j))^2},
\]
where $n$ is the size of the finite graph $\Gamma$ (see~\cite[section 2]{bobkovbhoudretetali2000}). From~\cite[Theorem 1]{bobkovbhoudretetali2000} and a basic convexity argument, we have 
	\[ h(G^k)\leq (2+\sqrt{2})\sqrt{\lambda_{\infty}(G^k)}.\]
	Moreover, we have $ \lambda_{\infty}(G^k) = \frac{\lambda_{\infty}(G)}{k} $ (\cite[Concluding Remarks]{bobkovbhoudretetali2000}), $ \lambda_{\infty}(G) \leq \lambda_2(G) $ by definition, and $\lambda_2(G)\leq 2\deg(G)h(G)$ from Theorem~\ref{t:chgrinqlts}. Then, we derive
	\[h(G^k)\leq (2\sqrt{2}+2)\frac{\sqrt{\deg(G)h(G)}}{\sqrt{k}}.\]
\end{proof}
\begin{example}
	We do not know whether the lower bound is sharp or not, but the upper bound is sharp in the case where $ G $ is the path $ \left[-n,n\right] $. Indeed, Wang \& Wang showed in~\cite{wangwang_discrete} that, up to constants, the following sets realize the infimum in the definition of the Cheeger constant of $ \left[-n,n\right]^k $:
	\[ A_k = \set{(x_1,\dots,x_k)\in \left[-n,n\right]^k, \sum_{i=1}^k x_i < 0} \]
	Indeed, $ A_k $ contains roughly half of the points of $ \left[-n,n\right]^k $, and its (vertex)-boundary is:
	\[ \partial A_k = \set{(x_1,\dots,x_k)\in \left[-n,n\right]^k, \sum_{i=1}^k x_i = 0} \]
	
	If we consider that $(x_i)_{i\geq 1}$ is a sequence of independent uniformly distributed random variables in $\left[-n,n\right] $, their partial sum $ y_k = \sum_{i=1}^{k} x_i $ can be reinterpreted as a random walk in $ \Z $. It is a well known fact that the probability of having $ y_k = 0 $ is, up to constants, equivalent to $ \frac{1}{\sqrt{k}} $. This gives then an isoperimetric ratio $\frac{\card{\partial A_k}}{\card{A_k}}$ of the form $ \frac{1}{\sqrt{k}} $.
\end{example}
\paragraph{Edge-Cheeger constants}
We give here the analogous of Proposition~\ref{p:prop_prod} in the context of edge-Cheeger constants. This paragraph will not be used in the proofs of our theorems. We detail this here for completeness, because this context is more usual and has more connections with analysis.
\begin{definition}
	We define the {\bf edge-Cheeger constant} of a graph $ \Gamma $ as 
\[h_e(\Gamma) \vcentcolon= \inf\frac{|E(A,V\Gamma\setminus A)|}{\card A},\] where the infimum is taken on non-empty subsets $ A $ of $ V\Gamma $ of size at most $\frac{V\Gamma}2$, and $ E(A,V\Gamma\setminus A) $ denotes the set of edges between $ A $ and its complementary in $ V\Gamma $.
\end{definition}
The analogous of Proposition~\ref{p:prop_prod} in this context is:
\begin{proposition}\label{prop_prod_edge}
	Let $ G $ be a connected regular graph. Let $ k $ be a positive integer. Then
	\[ a' \leq h_e(G^k) \leq  b'{\sqrt{k}}, \]	
	with $ a' = \frac14 \frac{h_e(G)^2}{\deg G}$ and $ b' = 2\sqrt2\sqrt{h(G)\deg G}. $
\end{proposition}
\begin{proof}
	The proof uses the same ingredients as the proof of Proposition~\ref{p:prop_prod}:
	\begin{itemize}
		\item The Cheeger inequalities for edge-Cheeger constants (see~\cite[Lemma 2.1, Theorem 2.2]{chung1997}) give
		\[ \frac{h_e^2(G)}{2\deg G}\leq \lambda_2(G)\leq 2h_e(G),\]
and
		\[ \frac{h_e^2(G^k)}{2k\deg G}\leq \lambda_2(G^k)\leq 2h_e(G^k),\]
		\item and~\cite{fiedler1973} gives $ \lambda_2(G^{k}) = \lambda_2(G).$\qedhere 
	\end{itemize}
\end{proof}
The lower bound in Proposition~\ref{prop_prod_edge} is sharp. We can take again the example where $ G $ is the path $ \left[-n,n\right] $. From~\cite{bollobasleader_edge}, the half space $ G^{k-1}\times\left[-n,0\right] $ realizes (up to constants) the infimum in the definition of the (edge-)Cheeger constant of $ \left[-n,n\right]^k $. Since its edge-boundary consists in $ (2n+1)^{k-1} $ edges, the resulting Cheeger constant is, up to constants, equivalent to $ 1/n $, which is independent of $ k $. 

This paragraph shows a difference of behaviour, depending on the notion of isoperimetry that we consider. See~\cite{barbererde_isoperimetry} for more details on isoperimetric problems in the grid.
\subsection{Distorted lamp groups and their separation}
We fix a lamplighter diagonal product $\Delta$ as in Definition~\ref{d_lamp_diag_prod}. In this subsection, we exhibit subgraphs of $ \Delta $, and study their separation. To do so, we compare these subgraphs with Cartesian powers of the lamp groups, that will play the role of model graphs.
\subsubsection{Distorted lamp groups}\label{s_dstrtdlmpgrps}
\begin{definition}\label{dist_lamp_goup}
	Let $\Gamma_s$ be a group generated by two subgroups $A_s$ and $B_s$.
	We define $\Gamma_s^{k_s,r}$ as the graph with vertex set $\left( \Gamma_s \right)^{\left[-r,r\right]}\times \left[ -(r+k_s),r+k_s \right]$, and the following edges:
	
	\begin{itemize}
			\item $[(x_{-r},\dots,\underset{\scriptscriptstyle(j)}{x_j},\dots,x_r),j-k_s] \sim [(x_{-r},\dots,\underset{{\scriptscriptstyle (j)}}{x_j b},\dots,x_r),j-k_s]$\quad (called ``$ B$-edges''),
		\item $\left[\left(x_{-r},\dots,x_r\right),i\right] \sim \left[\left(x_{-r},\dots,x_r\right),i+1\right]$\quad (called ``$ \Z $-edge''),
		\item $[(x_{-r},\dots,\underset{\scriptscriptstyle(j)}{x_j},\dots,x_r),j+k_s] \sim [(x_{-r},\dots,\underset{{\scriptscriptstyle (j)}}{x_j a},\dots,x_r),j+k_s]$\quad (called ``$A$-edges''),
	\end{itemize}
	\noindent for any $ i\in \left[-(r+k_s), r + k_s -1\right] $, $ j \in \left[-r,r\right] $, $ a\in A_s $ and $ b\in B_s $. The notation ``$g\sim h$" means that $\set{g,h}$ is an edge of the graph $\Gamma_s^{k_s,r}$.
	
\end{definition}
To figure out more clearly the shape of the graphs $ \Gamma_s^{k_s,r} $, see Figure~\ref{f:lineofx0}. Intuitively, we think of this graph as a distorted product of lamp groups: a product of copies of the group $ \Gamma_s $ where we have extended the edges by a factor $ 2k_s +1 $. More precisely, a way of representing the graph $ \Gamma_s^{k_s,r} $ is to partition it by subsets of the form $ \set{\left(x,i\right),\, i\in\left[-k_s-r,  k_s + r\right]} $. We call such a subset a \textit{line}, see Figure~\ref{f:lineofx0}. Then, we can distinguish three parts in such a \textit{line}:
	\begin{itemize}
		\item For $ i\in\llbracket -k_s-r,-k_s+r\rrbracket $, the \textit{$B$-tail}, where vertices have $ \Z $-edges and $ B $-edges.
		\item For $ i\in\llbracket -k_s+r-1,k_s-r-1\rrbracket $, the \textit{body}, where vertices only have $ \Z $-edges.
		\item For $ i\in\llbracket k_s-r,k_s+r\rrbracket $, the \textit{$A$-tail}, where vertices have $ \Z $-edges and $ A $-edges.
	\end{itemize}

	Travelling through an $ A $-edge or a $ B $-edge changes one coordinate of $ x $, and keeps the same value for $ i $, and travelling through a $\Z$-edge keeps the same value for $ x $ and adds or subtracts $ 1 $ from $ i $ (see \S\ref{s_constructiondelta} for details).
	
The case $r=0$ is particular, because $\Gamma_s^{k_s,0}$ is an homothetic copy of $\Gamma_s$. This is the following proposition.
\begin{proposition}\label{p:hmthtclmpgrps}
Let $\Gamma_s^{k_s,0}$ be as in Definition~\ref{dist_lamp_goup} with $r=0$. We can define
\[\fonction{\iota}{\Gamma_s}{\Gamma_s^{k_s,0}}{x}{(x,0)}\]
 Then, for any $x,y\in\Gamma_s$, we have
\[d(\iota(x),\iota(y))=2k_sd(x,y).\]
\end{proposition}
This observation will be exploited in Appendix~\ref{s_lipschitz} to prove more general results concerning bilipschitz embeddings of graphs.

To show that this graph embeds in $ \Delta $, we start with a lemma. We remind the reader that $ a_s $ (respectively $ b_s $) denotes a group isomorphism from $ A $ to $ A_s $ (respectively from $ B $ to $ B_s $).
		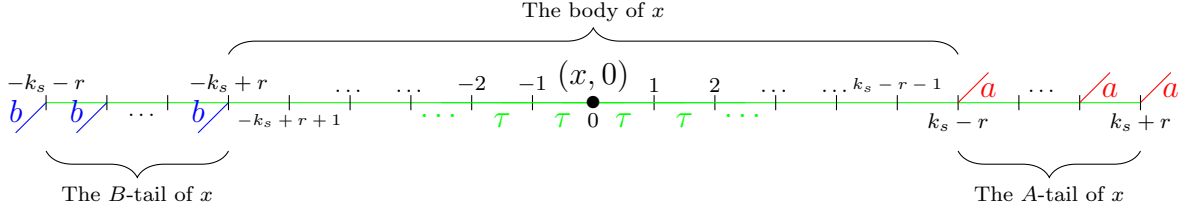
\begin{figure}
			\begin{center}
				\caption{the \textit{line} in $ \Gamma_{s}^{k_s,r} $ of $ x $ : $ \set{\left(x,i\right),\, i\in\left[-k_s-r,  k_s + r\right]} $.}
				\label{f:lineofx0}
				\begin{tikzpicture}
				\newcommand\echelle{0.8}
				\newcommand\x{6}
				\newcommand\y{\x}
					\begin{scope}
						\newcommand\maxi{9}
						\draw[-,color=green] (-\maxi*\echelle,0) -- (\maxi*\echelle,0);
						\foreach \x in {1,...,\maxi} \draw[-] (\echelle*\x,-0.1) -- (\echelle*\x,0.1) (-\echelle*\x,-0.1) -- (-\echelle*\x,0.1);
						\renewcommand\y{6}
						\draw[-,color=blue] (-\y*\echelle,0) -- (-\y*\echelle-0.5*\echelle,0-0.5*\echelle) node[above] {$ b $};
						\renewcommand\y{8}
						\draw[-,color=blue] (-\y*\echelle,0) -- (-\y*\echelle-0.5*\echelle,0-0.5*\echelle) node[above] {$ b $};
						\renewcommand\y{9}
						\draw[-,color=blue] (-\y*\echelle,0) -- (-\y*\echelle-0.5*\echelle,0-0.5*\echelle) node[above] {$ b $};
						\foreach \eps in {-1,1}	\foreach \x in {1,0}
						\draw[-,color=green] (\eps*\x*\echelle,0) -- (\eps*\x*\echelle+\eps*\echelle,0);
						\foreach \eps in {-1,1}	\foreach \x in {1,0}
						\draw[-,color=green] (\eps*\x*\echelle,0) -- (\eps*\x*\echelle+\eps*\echelle*0.5,0) node[below] {$ \tau $};
						\foreach \eps in {-1,1}	\foreach \x in {2}
						\draw[-,color=green] (\eps*\x*\echelle,0) -- (\eps*\x*\echelle+\eps*\echelle*0.5,0) node[below] {$ \dots $};
						\renewcommand\y{6}
						\draw[-,color=red] (\y*\echelle,0) -- (\y*\echelle+0.5*\echelle,0+0.5*\echelle) node[below] {$a$};
						\renewcommand\y{8}
						\draw[-,color=red] (\y*\echelle,0) -- (\y*\echelle+0.5*\echelle,0+0.5*\echelle) node[below] {$a$};
						\renewcommand\y{9}
						\draw[-,color=red] (\y*\echelle,0) -- (\y*\echelle+0.5*\echelle,0+0.5*\echelle) node[below] {$a$};
						\draw (0*\echelle,0) node[below] {\scriptsize$0$};
						\draw (0*\echelle,0) node {$\bullet$} node[above] {${(x,0)}$};
						\foreach \x in {-2,-1,1,2}
							\draw (\x*\echelle,0) node[above] {\scriptsize $\x$};
						\foreach \eps in {-1,1}
							\draw (\eps*3*\echelle,0) node[above] {\scriptsize $\dots$};
						\foreach \eps in {-1,1}
							\draw (\eps*4*\echelle,0) node[above] {\scriptsize $\dots$};
						\draw (\x*\echelle-\echelle,0) node[above] {\tiny $k_s-r-1$};
						\draw (\x*\echelle,0) node[below] {\scriptsize $k_s-r$};
						\draw (\x*\echelle+1.4*\echelle,0) node[above] {\scriptsize $\dots$};
						\draw (\x*\echelle+3*\echelle,0) node[below] {\scriptsize  $k_s+r$};
						\draw (-\x*\echelle+\echelle,0) node[below] {\tiny $-k_s+r+1$};
						\draw (-\x*\echelle,0) node[above] {\scriptsize $-k_s+r$};
						\draw (-\x*\echelle-1.4*\echelle,0) node[below] {\scriptsize $\dots$};
						\draw (-\x*\echelle-3*\echelle,0) node[above] {\scriptsize $-k_s-r$};
						\draw[decorate,decoration={brace,amplitude=10pt},xshift=0pt,yshift=0pt] 
						(-\x*\echelle,-0.8*\echelle)--(-9*\echelle,-0.8*\echelle) node [black,midway,yshift=-20*\echelle] 
						{\scriptsize The $ B $-tail of $ {x} $}; 
						\draw[decorate,decoration={brace,amplitude=10pt},xshift=0pt,yshift=0pt]
						(-\x*\echelle,0.8*\echelle)--(\x*\echelle,0.8*\echelle) node [black,midway,yshift=20*\echelle] {\scriptsize The body of $ {x} $}; 
						\draw[decorate,decoration={brace,amplitude=10pt},xshift=0pt,yshift=0pt] (9*\echelle,-0.8*\echelle)--(\x*\echelle,-0.8*\echelle) node [black,midway,yshift=-20*\echelle] 
						{\scriptsize The $ A $-tail of $ {x} $}; 
					\end{scope}
				\end{tikzpicture}
			\end{center}
		\end{figure}
\begin{lemma}
	Let $ x $ be an element of $ \Gamma_s $. Then there exists a couple $ \left(x^{A_s},x^{B_s}\right)\in A_s\times B_s $ such that for any decomposition of $ x = \prod_{i=0}^{n} a_i b_i $, where $\left(a_i\right)_{i\in \left[0,n\right]}$ and $\left(b_i\right)_{i\in \left[0,n\right]}$ are some sequences of elements respectively of $ A_s $ and $ B_s $, we have $\prod_{i=0}^{n} a_i = x^{A_s}$ and $\prod_{i=0}^{n} b_i=x^{B_s}$.
\end{lemma}

\begin{proof}
	According to the assumption that the groups $ \Gamma_s / \left[A_s,B_s\right]^{\Gamma_s} $ and $A_s\times B_s$ are isomorphic, we have a well defined group homomorphism from $ \Gamma_s / \left[A_s,B_s\right]^{\Gamma_s} $ to $ A_s\times B_s $. Composing by the quotient map $ \Gamma_s \twoheadrightarrow \Gamma_s /\left[A_s,B_s\right]^{\Gamma_s} $, we get a well defined group homomorphism from $ \Gamma_s $ to $ A_s\times B_s $. The announced result follows.
\end{proof}

\begin{proposition}\label{plongement_expanseur_dilate}
	For any $ r\leq k_s/2 $, the graph $\Gamma_s^{k_s,r}$ is isomorphic to a subgraph of $ \Delta $.
\end{proposition}
For simplicity, we will still denote by $\Gamma_s^{k_s,r}$ the corresponding subgraph of $ \Delta $.
\begin{proof}
	We remind that the elements of $\Delta$ are denoted $ \left(\left(f_{s'} \right)_{{s'}\geq 0},i \right) $, where $ i $ is an integer, and for every $ {s'} $, $ f_{s'} $ is a map of finite support from $ \Z $ to $ \Gamma_{s'} $. 
	
	For any $ x\in \Gamma_s $ and $ s'\geq 0 $, we write $ x^{A_{s'}} = a_{s'}\circ a_{s}^{-1}(x^{A_s})  $ and $ x^{B_{s'}} = b_{s'}\circ b_{s}^{-1}(x^{B_s})  $. Let $ r $ be such that $ r\leq k_s/2 $. We define the following map:
\begin{align*}
\phi \colon \left( \Gamma_s \right)^{\left[-r,r\right]}\times \left[ -(k_s+r),r+k_s \right] &\rightarrow \Delta\\
	\left[\left(x_{-r},\dots,x_r\right),i\right] &\mapsto \left(\left(f_{s'}\right)_{{s'}\geq 0},i\right),\\
	&\text{with } f_{s'} = \sum_{j\in {\left[-r,r\right]}} x_{j}^{A_{s'}} \delta_{j + k_s - k_{s'}} + \sum_{j\in {\left[-r,r\right]}} x_{j}^{B_{s'}} \delta_{j - k_s + k_{s'}}\text{ if } s'\neq s,\\
	&\text{and } f_s = \sum_{j\in {\left[-r,r\right]}} x_j \delta_j.
\end{align*}

When we define $ f_{s'} $ for $ s'\neq s $, we think of the two sum as ``writing'' some elements of $ A_{s'} $ and of $ B_{s'} $. The sum is valid if they are written at different places, \textit{i.e.} if the supports of the two sums are disjoint, which is not clear \textit{a priori}. However, under the assumption that $ r\leq k_s/2 $:
	
\begin{itemize}
	\item If $ s' < s $: the elements of $ B_{s'} $ are written in the interval $ \left[-r-(k_s - k_{s'}), r -(k_s - k_{s'}) \right] $, and the elements of $ A_{s'} $ are written in the interval $ \left[-r + (k_s - k_{s'}), r + (k_s - k_{s'}) \right] $. Since $ k_s > 2 k_{s'} $ by hypothesis, which implies $ k_{s}/2 < k_s - k_{s'} $, these two intervals are disjoint.
	\item If $ s' > s $: the elements of $ A_{s'} $ are written in the interval $ \left[-r -(k_{s'} - k_{s}), r -(k_{s'} - k_{s}) \right] $, and the elements of $ B_{s'} $ are written in the interval $ \left[-r + (k_{s'} - k_{s}), r + (k_{s'} - k_{s}) \right] $. Since $ k_{s'} > 2 k_{s} $ by hypothesis, which implies $ k_{s} < k_{s'} - k_{s}$, these two intervals are disjoint.
	
\end{itemize}

Thus $\phi$ is well defined and is moreover injective. Let $(v_1,v_2) $ be an edge of $\Gamma_s^{k_s,r}$. Using the terminology of Definition~\ref{dist_lamp_goup}, three cases can occur:
\begin{itemize}
	\item if $ (v_1,v_2) $ is a $ \Z $-edge, then $ (\phi(v_1),\phi(v_2)) $ is clearly an edge of $ \Delta $.
	\item if $ (v_1,v_2) $ is a $ A $-edge, then $v_1$ and $v_2$ are respectively of the form: $$[(x_{-r},\dots,\underset{\scriptscriptstyle(j)}{x_j},\dots,x_r),j+k_s] \text{, and }[(x_{-r},\dots,\underset{{\scriptscriptstyle (j)}}{x_j a},\dots,x_r),j+k_s].$$ This implies, in $\Delta_s$, we have $ \pi_s(\phi(v_1)) = \pi_s(\phi(v_2))\times (a_s(a)\delta_{-k_s},0)  $. Additionally, for any $ s'\neq s $, $ ({x_j a})^{A_{s'}} = ({x_j}^{A_{s'}})\times a_{s'}(a) $ and then we have the same equality in $ \Delta_{s'}$: $\pi_{s'}(\phi(v_1)) = \pi_{s'}(\phi(v_2))\times (a_{s'}(a)\delta_{-k_{s'}},0)$.
	Then, $\phi(v_1)=\phi(v_2)a$, which means that $ (\phi(v_1),\phi(v_2)) $ is an edge of $ \Delta $.
	\item if $ (v_1,v_2) $ is a $ B $-edge, the same reasoning as for $ A $-edges is valid.
\end{itemize}

 Therefore $ \phi $ is a graph embedding from $\Gamma_s^{k_s,r}$ to $ \Delta $.
\end{proof}
\subsubsection{Comparison with Cartesian powers}
For any $ r\geq 0$, we denote  ${\Gamma_s}^{\left[-r,r\right]}$ the (cartesian) product of $ 2r+1 $ copies of $ \Gamma_s $, indexed by $[-r,r]$. The following proposition compares the separation of ${\Gamma_s}^{\left[-r,r\right]}$ with that of the graph $\Gamma_s^{k_s,r}$ introduced above.
\begin{proposition}\label{p_cut_expanseurs_dilates}
	For any $ r\geq 0 $,
\[\cut(\Gamma_s^{k_s,r})\geq \cut\left({\Gamma_s}^{\left[-r,r\right]}\right).\]
\end{proposition}

\begin{proof}
	Let $ C^{k_s} $ be a cutset of $\Gamma_s^{k_s,r}$. Let \[ C = \left\{ x\in{\Gamma_s}^{\left[-r,r\right]}\mid \exists i\in \left[-(r+k_s),r + k_s \right]\ (x,i)\in C^{k_s} \right\}.\] We have $ \left|C\right | \leq \left |C^{k_s}\right | $. Let us show that $ C $ is a cutset of $ \Gamma_s^{\left[-r,r\right]} $.
	Let $ A $ be a connected subset of $ {\Gamma_s}^{\left[-r,r\right]}\setminus C $. Let $ A^{k_s} = \left\{ (x,i)\mid x\in A \text{ and } i\in \left[-(r+k_s),r + k_s \right] \right\}  $.	We have $ \left|A^{k_s}\right| = \left(2r + 2k_s + 1 \right)\times \left|A\right| $. Moreover, $ A^{k_s} $ does not meet $ C^{k_s} $ and induces a connected graph: any path in $ \Gamma_s^{r+1}\setminus C $ can be followed in $ \Gamma_s^{k_s,r}\setminus C^{k_s} $ since we are allowed to move the integer $ i $ in the whole interval $\left[-(r+k_s),r + k_s \right]$. Since $ C^{k_s} $ is a cutset of $\Gamma_s^{k_s,r}$, $ \left|A^{k_s} \right| \leq \frac{\left|\Gamma_s^{k_s,r}\right|}{2}=\frac{2r + 2k_s + 1}{2} \left| {\Gamma_s}^{\left[-r,r\right]} \right|$. Since $  \left|A^{k_s}\right| = \left(2r + 2k_s + 1 \right)\times \left|A\right|  $, we can deduce that $ A \leq \frac{\left|{\Gamma_s}^{\left[-r,r\right]}\right|}{2} $. This means that $ C $ is a cutset of $ \Gamma_s^{r+1} $.
	Therefore, $ \cut\left({\Gamma_s}^{\left[-r,r\right]}\right) \leq \cut\left(\Gamma_s^{k_s,r}\right) $.
	\end{proof}
In Appendix~\ref{s_chgrcnstdstrtdgfphs}, we study more general statements in the same spirit: in section~\ref{s_coar_part}, we show a generalization of this proof in the context of coarsenings of graphs, and, in sections~\ref{s_lipschitz} and~\ref{s_analyticmethod}, two alternative proofs in the case $r=0$.\\

We can prove Theorem~\ref{thm_low-bnd-sep}. 
\begin{proof}[Proof of Theorem~\ref{thm_low-bnd-sep}]
Let $s\geq0$ and $r\leq k_s/2$. Then, from Proposition~\ref{plongement_expanseur_dilate}, the graph $\Gamma_s^{k_s,r}$  is isomorphic to a subgraph of $\Delta$. We have
\begin{align*}
\cut(\Gamma_s^{k_s,r})&\geq \cut\left({\Gamma_s}^{\left[-r,r\right]}\right),\quad\text{from Proposition~\ref{p_cut_expanseurs_dilates},}\\
&\geq\frac{h(\Gamma_s)^2}{16(\deg\Gamma_s)^2}\frac{\card{\Gamma_s}^{2r+1}}{2r+1},\quad\text{from Corollary~\ref{c:prop_prod}.}
\end{align*}
The graph $\Gamma_s^{k_s,r}$ has $(2k_s+2r+1)\card{\Gamma_s}^{2r+1}$ vertices.
Then, we have
\[\sep_\Delta((2k_s+2r+1)\card{\Gamma_s}^{2r+1})\geq\frac{h(\Gamma_s)^2}{16(\deg\Gamma_s)^2}\frac{\card{\Gamma_s}^{2r+1}}{2r+1}.\]
Finally, from Theorem~\ref{t:comppoincsep},
\[\Pi_{\Delta,p}((2k_s+2r+1)\card{\Gamma_s}^{2r+1})\geq 4^{-p}\frac{h(\Gamma_s)^2}{1536(\deg\Gamma_s)^2}\frac{\card{\Gamma_s}^{2r+1}}{2r+1}\qedhere \]
\end{proof}
\section{An upper bound on the Poincar\'e profiles}\label{s_up_bd}
\subsection{Compression in \texorpdfstring{$L^p$}{Lp} spaces and Poincar\'e profiles}
We show here an upper bound on $ L^p $-Poincaré profiles of graphs, using embeddings into $ L^p $~spaces. Before stating our theorem, we define the compression function of such an embedding:
\begin{definition}
	Let $ f\colon G \rightarrow L^p $ be a $ 1- $Lipschitz map from a graph into an $ L^p $ space. We define the \textbf{compression function} of $ f $, denoted $ \rho_f $, as:
	\[ \rho_f(t) = \inf\set{\norm{f(g)-f(h)}\mid d_G(g,h) \geq t}.\]
\end{definition}
We state our upper bound theorem:
\begin{theorem}\label{thm_borne_sup}
Let $ G $ be a graph of bounded degree. Then there exist two constants $ c_1,c_2 > 0 $, depending only on the maximum degree in $G$, such that if $ f\colon VG \rightarrow L^p $ is a $ 1 $-Lipschitz map, then
		\[  \Pi_{G,p}(N) \leq c_1 \frac{N}{\rho_f(c_2\log N)} \stepcounter{equation}\tag{\theequation}\label{croiss_exp},\]
for all $p\in \left [ 1,\infty \right) $ and $N\geq0$.

More precisely, if there exists a function $ \sigma $ such that for any vertex $ x $ of $ G $, the sphere centred at $ x $ of radius $ n $ contains at most $ \sigma(n) $ vertices, then for any $ N $ we have:
\[\stepcounter{equation}\tag{\theequation}\label{croiss_generale}
\Pi_{G,p}(N) \leq 2^{\frac{2p-1}p} \sigma(1)^{1/p}\left( \frac{ N^{p+1}}{\sum_{n=0}^{K} \sigma(n) \rho_{f}(n)^{p}} \right)^{1/p},\]
	where $ K $ is the biggest integer such that $\sum_{n=0}^{K} \sigma(n) \leq N $ (depends on $ N $).
\end{theorem}
\begin{remark}
	As mentionned in the introduction (see Theorem~\ref{thm_borne_sup1}), the inequality~\eqref{croiss_exp} is known to be sharp. In this more precise statement, we can comment on inequality~\eqref{croiss_generale} which improves~\eqref{croiss_exp} when $G$ doesn't have exponential growth. Indeed, one may notice that the inequality \eqref{croiss_generale} is asymptotically optimal for the inclusion map $ \Z^{d} \hookrightarrow \left(\R^{d} , \ell^{1}\right) $. In this case the compression function is $\rho(t) \simeq t$ and we can take $ \sigma(n) = cn^{d-1} $. From Theorem~\ref{thm_borne_sup}, we can deduce that $ \Pi_{\Z^{d},1}(N) \preceq n^{\frac{d-1}{d}} $, which is optimal, using Proposition~\ref{p:prop_prod}, or~\cite[Theorem 7]{humemackaytessera}.
	
	In the case of the Heisenberg group, the inequality \eqref{croiss_generale} is not asymptotically optimal if $ p\geq 2 $. Indeed, Austin, Naor and Tessera showed in~\cite{austinnaortessera2013} that any $1$-Lipschitz embedding of the Heisenberg group in a superreflexive Banach space has a compression function at most equivalent to $t\mapsto\frac{t}{\log^c t} $ for some positive constant $ c $. The inequality \eqref{croiss_generale} gives, in this optimal case (with $ \sigma(n) = c' n^{3} $ and assuming that $ c< 1/p $), $ \Pi_{\mathbb{H}^{4},p}(N) \preceq \log(N)^{\frac{1}{p}-c} N^{\frac{3}{4}} $, while we have $ \Pi_{\mathbb{H}^{4},p}(N) \asymp N^{\frac{3}{4}} $, again from~\cite[Theorem 7]{humemackaytessera}.
	 
	We will see some cases where \eqref{croiss_exp} is optimal in \cpt{Chapter}{Section}~\ref{s:comparison}.
\end{remark}
For the proofs, we will use another notion of gradient; we define the associated Poincar\'e profile:
	\begin{definition}\label{d:modifiedpoincare}
	Let $p\in\left[1,\infty\right)$.
	\begin{itemize}
		\item Let $ \Gamma $ be a finite graph. We define the \textbf{modified $ L^p $-cheeger constant} of $ \Gamma $ as:
		\[ \tilde{h}_p(\Gamma) = \inf \left\{\frac{\left\|\tilde{\nabla} f \right\|_p}{\left\|f - f_{\Gamma}\right\|_p} \colon f\in \Map(V\Gamma \rightarrow \R), \left\|f \right\|_p \not\equiv f_{\Gamma}  \right\},\]
with $\card{\nabla f} (g) = \left(\sum_{h\sim g} \left|f(g) - f(h)\right|^{p}    \right)^{1/p}$ and $f_{\Gamma}=\left|V\Gamma\right|^{-1} \sum_{g\in\Gamma} f(g)$.
		\item Let $ G $ be an (infinite) graph. Following~\cite{humemackaytessera}, we define the \textbf{modified $L^p$-Poincar\'e profile} of $ G $ as  
		\[  \tilde{\Pi}_{G,p}(n) = \sup \left\{\left |V\Gamma\right | \tilde{h}_p\left(\Gamma\right) \vcentcolon \Gamma \subset G, \left| V\Gamma \right|\leq n \right\}.\]
	\end{itemize}	
\end{definition}
\begin{remark}\label{r:modifiednoloss}
This definitions are equivalent to our previous ones (see Definition~\ref{d:lpchgrcnst}) in the following sense:
\begin{itemize}
\item If $\Gamma$ is a finite graph, and $ D $ is a bound on the degrees of the vertices of $ \Gamma $, then for any $p\in\left[1,\infty\right)$,
\[D^{-1/p}\tilde{h}_p(\Gamma)\leq h_p(\Gamma)\leq2^\frac{p-1}p\tilde{h}_p(\Gamma).\]
\item If $ G $ is an infinite graph of bounded degree, and $ D $ is a bound on the degrees of the vertices of $ G $, then, for any $p\in\left[1,\infty\right)$,
\[D^{-1/p}\tilde{\Pi}_{G,p}\leq\Pi_{G,p}\leq2^\frac{p-1}p\tilde{\Pi}_{G,p}.\]
\end{itemize}
	Then, the proof of Theorem~\ref{thm_borne_sup} can be done without loss of generality on the \textit{modified} Poincar\'e profiles.
\end{remark}

We give a property on modified $ L^p $-Cheeger constants.
\begin{proposition}\label{p:cheeger-lp-valued}
	If $ p > 1 $, we do not change the value of $ h_{p}\left(\Gamma\right ) $ considering functions taking their values in an $ L^{p} $ space instead of $ \R $,  \textit{i.e.}:	
	
If we define
\[\tilde h_p(\Gamma,L^p) = \inf \left\{\frac{\norm{\tilde\nabla f}}{\norm{f - f_{\Gamma}}} \vcentcolon f\in \Map(V\Gamma \rightarrow  \mathbf{L^{p}}), \left\|f \right\|_p\not\equiv f_{\Gamma} \right\},\]
with
\begin{itemize}
\item $\card{\tilde\nabla f} (g) = \left(\sum_{h\sim g}\norm{f(g) - f(h)}^{p}\right)^{1/p}$,
\item $f_{\Gamma}=\left|V\Gamma\right|^{-1} \sum_{g\in\Gamma} f(g)$,
\item and $\norm{f - f_{\Gamma}}=\left(\sum_{g\in VG}\norm{f(g)-f_\Gamma}^p\right)^{1/p}$,
\end{itemize}
then, we have
\[\tilde h_p(\Gamma,L^p) = \tilde h_p(\Gamma).\]
\end{proposition}
\begin{proof}
	The inequality $\tilde h_p(\Gamma,L^p)\leq\tilde h_p(\Gamma)$ is obvious. We prove the other inequality. Let us write $ L^p = L^p\left (X,\mu\right ) $, with $ \left(X,\mu\right) $ a measured space. 
	We denote by $ \mathcal{L}^p $ the set of functions from $ X $ to $ \R $ such that their $ p $ power is integrable (without quotienting by the almost everywhere equality equivalence relation). 
	Let $ f\colon V\Gamma\to\mathcal{L}^p $
	be a non zero map. Without loss of generality, we can assume that $f_{\Gamma} = 0$. For every $ x \in X$, we set
\[ \fonction{f_x}{V\Gamma}{\R}{g}{f(g)(x)}.\]	
Since $f_{\Gamma} = 0$, we have ${(f_x)}_\Gamma=0$ for every $x\in X$.
	Let $c\ge0$ be such that for every $ x \in X $ we have $\norm{\tilde\nabla f_x}\geq c\norm{f_x}$. Then we have for every vertex $ g $ of $ \Gamma $:
	\allowdisplaybreaks
	\begin{align*}	
		\left(\tilde{\nabla} f (g) \right)^p &= \sum_{h\sim g} \left \| f(g) - f(h) \right \|_p^p\\
		&= \sum_{h\sim g} \int_{X} \left| f_x(g) - f_x(h) \right|^p \mathrm{d}\mu(x)\\
		&= \int_{X} \sum_{h\sim g} \left| f_x(g) - f_x(h) \right|^p \mathrm{d}\mu(x)\\
		&= \int_{X} \left(\tilde{\nabla} f_x(g)\right)^p \mathrm{d}\mu(x).
	\end{align*}
	Therefore,
	\allowdisplaybreaks
	\begin{align*}
		\left\|\tilde{\nabla} f \right\|_p^p 
		&=  \sum_{g\in V\Gamma}\int_{X} \left(\tilde{\nabla} f_x(g)\right)^p \mathrm{d}\mu(x)\\
		&= \int_{X} \sum_{g\in V\Gamma}\left(\tilde{\nabla} f_x(g)\right)^p \mathrm{d}\mu(x)\\
		&= \int_{X} \left \| \tilde{\nabla} f_x(g) \right \|_p^p \mathrm{d}\mu(x)\\
		&\geq c^p \int_{X} \left \| f_x \right \|_{p}^{p}\mathrm{d}\mu(x)\\
		&= c^p \int_{X} \sum_{g\in V\Gamma} \left| f_x(g) \right|^{p}\mathrm{d}\mu(x)\\
		&= c^p \sum_{g\in V\Gamma} \left \| f(g) \right \|_{p}^{p} \\
		&= c^p \left \| f \right \|_p^p.
	\end{align*}
	
	Then we deduce that $\left\|\tilde{\nabla} f \right\|_p \geq c \left \| f \right \|_p$.

	Let now $c\ge0$ satisfying $\norm{\tilde\nabla f} < c \norm{f}$. Then, from above, there exists $x\in X$ such that $\norm{\tilde\nabla f_x}<c\norm{f_x}$. This implies in particular $\norm{f_{x}}\neq0$. Then we have $\tilde h_p(\Gamma)\leq\frac{\norm{\tilde\nabla f_x}}{\norm{f_x}}<c$. Taking the infimum in $c$, we obtain $\tilde h_p(\Gamma)\leq\frac{\norm{\tilde\nabla f}}{\norm{f}}$. Taking the infimum in $f$, we obtain $\tilde h_p(\Gamma)\leq\tilde h_p(\Gamma,L^p)$.
\end{proof}
Before proving Theorem~\ref{thm_borne_sup}, we prove two lemmas.
	\begin{lemma}\label{var}
	Let $\Gamma$ be a finite graph, let $p\in[1,\infty)$. We define the \textbf{$ p $-variance} of a function $f\colon\Gamma\rightarrow L^{p}$ as:
	\[\Var_p(f) = \left( \dfrac{1}{\left |V\Gamma\right |^{2}} \sum_{g\in V\Gamma} \sum_{h\in V\Gamma} \left \|f(g)-f(h)\right \|_{p}^{p} \right)^{1/p}.\]
	
	Then we have:
\[\frac{1}{\left |V\Gamma\right |^{1/p} } \left\|f-f_\Gamma\right \|_p	
\leq \Var_p(f)
\leq \frac{2}{\left| V\Gamma \right |^{1/p}} \left \|f-f_\Gamma\right \|_p.\]
	\end{lemma}
\begin{proof}
	\allowdisplaybreaks
	\begin{align*}
	\frac{1}{\card{V\Gamma}}\norm{f - f_{\Gamma}}^p
	&= \frac{1}{\card{V\Gamma}}\sum_{g\in V\Gamma} \norm{f(g) - f_{\Gamma}}^p\\
	&= \frac{1}{\card{V\Gamma}^{p+1}}\sum_{g\in V\Gamma} \norm{\sum_{h\in \Gamma} f(g) - f(h)}^p\\
	&\leq \frac{1}{\card{V\Gamma}^{p+1}}\sum_{g\in V\Gamma} \left( \sum_{h\in \Gamma} \norm{f(g) - f(h)}\right)^{p}\\
	&\leq \frac{\card{V\Gamma}^{p-1}}{\card{V\Gamma}^{p+1}}\sum_{g\in V\Gamma}\sum_{h\in \Gamma} \norm{f(g) - f(h)}^{p}\quad\text{since $ \left(\sum_{i=1}^nx_i\right)^p\leq n^{p-1} \left(\sum_{i=1}^nx_i^p\right)$}\\
	&= \frac{1}{\left |V\Gamma\right |^{2}} \sum_{g\in V\Gamma} \sum_{h\in V\Gamma} \left \|f(g)-f(h)\right \|_{p}^{p}\\
	&=\left(\Var_p(f)\right)^{p}\\
	&\leq \frac{1}{\left |V\Gamma\right |^{2}} \sum_{g\in V\Gamma} \sum_{h\in V\Gamma} \left(\left \|f(g)-f_{\Gamma}\right \|_{p} + \left \|f(h)-f_{\Gamma}\right \|_{p}\right)^{p}\quad\text{(triangle inequality)}\\
	&\leq \frac{2^{p-1}}{\left |V\Gamma\right |^{2}} \sum_{g\in V\Gamma} \sum_{h\in V\Gamma} \left \|f(g)-f_{\Gamma}\right \|_{p}^{p} + \left \|f(h)-f_{\Gamma}\right \|_{p}^{p}
	\\
	&= \frac{2^{p}}{\left |V\Gamma\right |} \sum_{k\in V\Gamma} \left \|f(k)-f_{\Gamma}\right \|_{p}^{p}\\
	&= \frac{2^{p}}{\left |V\Gamma\right |} \left\| f - f_{\Gamma} \right\|_{p}^{p}\qedhere
	\end{align*}
\end{proof}

Therefore we could have written a variance time $ \left|V\Gamma\right|^{1/p} $ instead of a norm in the definition of the Cheeger constant of $ \Gamma $. This would give an equivalent notion, since we are only interested in asymptotic behaviours. The second lemma is the following.

\begin{lemma}\label{fact}
Let $ h,s \colon \mathbb{N} \rightarrow \mathbb{N} $ be such that for any $ n\geq 0 $, $h(n) \leq s(n)$. We assume that the sum $ N \vcentcolon = \sum_{n=0}^{k} h(n)  $ is finite. Then for any non-decreasing function $\rho\colon \mathbb{N} \rightarrow \R	$, we have:
\[ \sum_{n = 0}^{+\infty} h(n) \rho(n) \geq \sum_{n = 0}^{k} s(n) \rho(n),\quad\text{for any $k$ such that $  \sum_{n=0}^{k} s(n) \leq N$ }. \]
\end{lemma}
%
%
%
\begin{proof}
The proof is very elementary. The function $h(n)$ being at most equal to $s(n)$, we will modify inductively it by a series of elementary actions such that we conserve the sum of $ h(n) $ equal to $ N $, and such that there is an integer $k$ such that $h(n)$ is equal to $s(n)$ in the interval $[0,k]$. At each step, this integer $k$ will increase by $1$, until we have $d(n)=0$ for every $n\ge k+1$.
The algorithm is the following: (see Figure~\ref{f_algo} for an illustration)
\begin{algorithm}[H]
	\While{True}{
		\eIf{$\forall i\geq 0\ h(i) = s(i) $}{\Return{$ h $}}{
		let $ i_0 $ be the smallest integer such that $ h(i_0) < s(i_0) $.
		}
		\eIf{$\forall i > i_0\ h(i) = 0 $}{
			\Return{$ h $}
		}
		{
			\eIf{$\sum_{i = i_0}^{+\infty} h(i) < s(i_0)$}
			{$h(i_0) \longleftarrow \sum_{i = i_0}^{+\infty} h(i) $\\
			for any $ i>i_0 , h(i)\longleftarrow 0$\\
			\Return{h}
			}{let $ j_0 $ be the smallest integer such that $\sum_{i = i_0}^{j_0} h(i) \geq s({i_0})$
			
			$\delta \longleftarrow  \sum_{i = i_0}^{j_0} h(i) - {s(i_0)} $
			
			$h(i_0) \longleftarrow s(i_0)$,\\
			for any $ i_0 < i < j_0 $,  $h(i) \longleftarrow 0 $, 
			
			$ h(j_0) \longleftarrow \delta $,			
			}
			
		}
	}
\end{algorithm}
\begin{figure}
	\begin{center}
		\caption{Illustration of Lemma~\ref{fact}}
		\label{f_algo}
		\begin{tikzpicture}
			\begin{scope}
				\draw[->] (0,0) -- (6,0);\draw (6,0) node[right] {$n$};\draw [->] (0,0) -- (0,6);\draw (0,6);
				\draw [domain=0:4.5,color=red]   plot(\x,1+\x) node[above] {$ s(n) $};
				\draw[line width=2mm,color=blue] plot[ycomb] coordinates {(0,1) (1,2) (2,1) (3,3) (4,2) (5,3)} node[above] {$ h(n) $}  ;
				\draw[<->,>=latex] (2,1) to (2,3) node[above]{this is missing...};
			\draw (2,-0.3) node {$ i_0 $};
			\end{scope}
			\begin{scope}[shift={((7.5,0))}]
				\draw[->,>=latex] (-1,3) to (1,3);
			\end{scope}
			\begin{scope}[shift={(9,0)}]
				\draw[->] (0,0) -- (6,0);\draw (6,0) node[right] {$n$};\draw [->] (0,0) -- (0,6);\draw (0,6);
				\draw [domain=0:4.5,color=red]   plot(\x,1+\x) node[above] {$ s(n) $};
				\draw[line width=2mm,color=blue!50] plot[ycomb] coordinates {(2,3)}  ;
				\draw[line width=2mm,color=blue] plot[ycomb] coordinates {(0,1) (1,2) (2,1) (3,1) (4,2) (5,3)} node[above] {$ h(n) $}  ;
				\draw (2,-0.3) node {$ i_0 $};
				%
				\draw[->,>=latex] (3,2) to[bend right] node[above=30] {we fill the gap !}  (2,2);
				\draw [-][dashed] (2.9,1) to  (2.9,3);
				\draw [-][dashed] (2.9,3) to  (3.1,3);
				\draw [-][dashed] (3.1,3) to  (3.1,1);
				\draw [-][dashed] (3.1,1) to  (2.9,1);	
			\end{scope}	
		\end{tikzpicture}
	\end{center}
\end{figure}
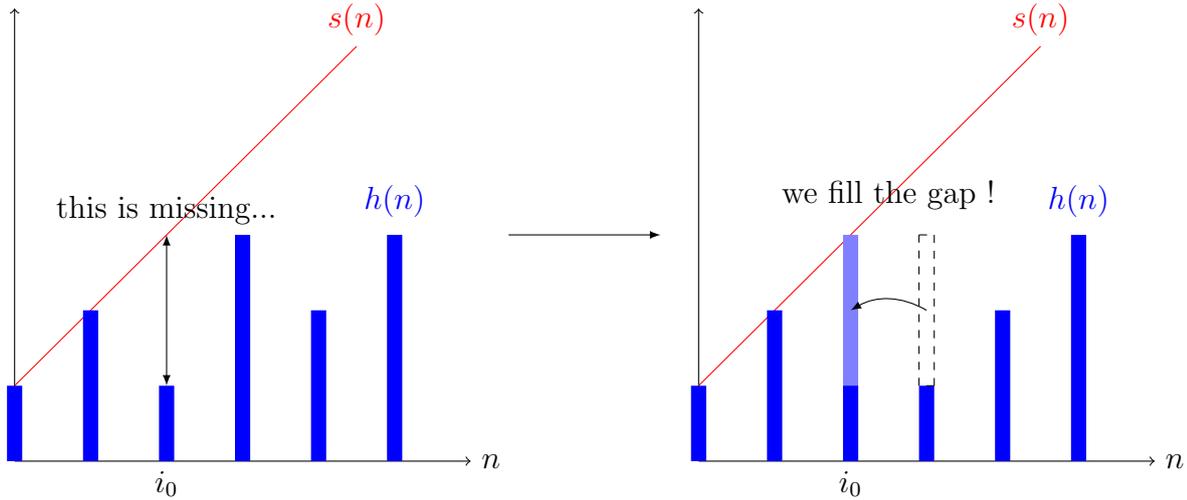


Since $ \rho  $ is non-decreasing, at each step of the process the quantity $ \sum_{n = 0}^{+\infty} h(n) \rho(n)  $ won't increase.

At the end on the process, the function $ h $ satisfies the following properties:
\begin{itemize}
	\item there exists an integer $ i_0 $ such that $ h(i) = s(i)$ for any $ i < i_0 $, and $ h(i) = 0$ for any $ i > i_0 $
	\item $\sum_{n=0}^{+\infty} h(n) = N$
\end{itemize}
	This proves that the inequality \[ \sum_{n = 0}^{+\infty} h(n) \rho(n) \geq \sum_{n = 0}^{k}s(n) \rho(n) \] is true for any $ k $ such that $ \sum_{n=0}^{k}s({n})\leq N $, which is what we wanted to prove.
%
%
\end{proof}

We can start the proof of Theorem~\ref{thm_borne_sup}.
\begin{proof}[Proof of Theorem~\ref{thm_borne_sup}]
	Without loss of generality, we can use the \textit{modified} Poincar\'e profile definition (Definition~\ref{d:modifiedpoincare}), see Remark~\ref{r:modifiednoloss} for details.
	We start with the second inequality. By definition, $\sigma(1)$ is a bound on the degrees on the vertices of $ G$.  	Let $n$ be a positive integer and $\Gamma$ be a connected subgraph of $G$ with at most $ n $ vertices.
	 Then the restriction $ f_{\vert V\Gamma}\colon \Gamma \rightarrow L^{p} $ is also $1$-Lipschitz for the induced metric on $ \Gamma $. For simplicity, we will still denote $f_{\vert V\Gamma}$ by $f$.
	Then we have:
	\begin{align}\label{grad}
		\norm{\tilde\nabla f}\leq \sigma(1)^{1/p}\card{V\Gamma}^{1/p}
	\end{align}
	We will now give an upper bound on the norm of $f_{\vert V\Gamma}$. We have the following inequalities:
\begin{align*}
	\Var_{p}(f_{\vert\Gamma})^{p}&=\frac{1}{\left |V\Gamma\right |^{2}} \sum_{g\in V\Gamma} \sum_{g'\in V\Gamma} \left \|f(g)-f(g')\right \|_{p}^{p}\\
	&\geq \dfrac{1}{|V\Gamma|^2} \sum_{g,g'\in V\Gamma}\left(\rho_{f}(d(g,g')\right)^{p} \\
	&\geq \dfrac{1}{|V\Gamma|^2} \sum_{g\in V\Gamma}\sum_{n\geq 0} \#{\set{g'\in V\Gamma \mid d(g,g')=n}} \rho_{f}(n)^{p}
\end{align*}

We fix $g\in V\Gamma$. Using Lemma~\ref{fact}, with $ h(n) = \#{\set{g'\in V\Gamma \mid d_G(g',g)=n}} $, $ s(n) = \sigma(n)$ and $ \rho = \rho_f^{p} $, we have $ \sum_{n=0}^{+\infty} h(n) = \card{V\Gamma} $ and we can set $ K $ the biggest integer such that $\sum_{n=0}^{K} \sigma(n) \leq \card{V\Gamma} $. We obtain, for every $g\in V\Gamma$,
\[\sum_{n\geq 0} \#{\set{g'\in V\Gamma \mid d(g,g')=n}} \rho_{f}(n)^{p}\geq\sum_{n= 0}^{K} \sigma(n) \rho_{f}(n)^{p}\]
We get
\begin{align*}
\Var_{p}(f)^{p}
&\geq \frac{1}{|V\Gamma|^2} \sum_{g\in V\Gamma}\sum_{n= 0}^{K} \sigma(n) \rho_{f}(n)^{p}\\\stepcounter{equation}\tag{\theequation}\label{variance}
&= \frac{1}{|V\Gamma|} \sum_{n= 0}^{K} \sigma(n) \rho_{f}(n)^{p}.
\end{align*}
Combining \eqref{grad}, Lemma~\ref{var}, and \eqref{variance}, we get:
\begin{align*}
\frac{\norm{\tilde\nabla f}}{\norm{f-f_\Gamma}}&\leq 2 \frac{\norm{\tilde\nabla f}}{ \card{V\Gamma}^{1/p}\Var_{p}(f)}\\
	&\leq2\frac{\sigma(1)^{1/p}\card{V\Gamma}^{1/p}}{\left(\sum_{n= 0}^{K} \sigma(n) \rho_{f}(n)^{p}\right)^{1/p}}.
\end{align*}
This implies
\begin{align*}
\card{V\Gamma}h_p(\Gamma)&\leq 2^{\frac{p-1}p}\card{V\Gamma}\tilde{h_p}(\Gamma),\quad\text{from Remark~\ref{r:modifiednoloss}}\\ 
&\leq 2^{\frac{2p-1}p} \sigma(1)^{1/p}\left( \frac{ \card{V\Gamma}^{p+1}}{\sum_{n=0}^{K} \sigma(n) \rho_{f}(n)^{p}} \right)^{1/p}.
\end{align*}
Since this is true for every subgraph $\Gamma\subset G$, we obtain, for every $N\geq 0$,
\[\label{croiss_exp_pf}\stepcounter{equation}\tag{\theequation}
\Pi_{G,p}(N)\leq 2^{\frac{2p-1}p} \sigma(1)^{1/p}\left( \frac{ N^{p+1}}{\sum_{n=0}^{K} \sigma(n) \rho_{f}(n)^{p}} \right)^{1/p},
\]
where $ K $ the biggest integer such that $\sum_{n=0}^{K} \sigma(n)\leq N$, which is the inequality~\eqref{croiss_exp}.

Let us prove the second inequality \eqref{croiss_exp}. Let $D$ be a bound on the degrees of the vertices of $G$. Inequality \eqref{croiss_exp} is obtained by applying inequality~\eqref{croiss_exp_pf} with $ \sigma(n) = D^{n} $, which is possible by definition of $ D $. Then we have $K\geq  \frac{\log((D-1)N+1)}{\log D} -2 \geq \frac{\log N}{\log D} -2$, and $D^K\geq ND^{-2}$. We can deduce, keeping only the last term of the sum in~\eqref{croiss_exp_pf},
\begin{align*}
\Pi_{G,p}(N)&\leq2^{\frac{2p-1}p}D^{1/p}\left( \frac{ N^{p+1}}{\sum_{n=0}^{K}D^n\rho_{f}(n)^{p}} \right)^{1/p}\\
&\leq2^{\frac{2p-1}p}D^{1/p}\left( \frac{ N^{p+1}}{D^K\rho_{f}(K)^{p}} \right)^{1/p}\\
&=2^{\frac{2p-1}p}D^{1/p}\frac{N^{\frac{p+1}{p}}}{D^{K/p}\rho_{f}(K)}\\
&\leq2^{\frac{2p-1}p}D^{3/p}\frac N{\rho_{f}\left(\frac{\log N}{2\log D}\right)},\quad\text{if $N\geq D^4$,}
\end{align*}
When $N<D^4$, we have $\rho_{f}\left(\frac{\log N}{2\log D}\right)\leq\frac{\log N}{2\log D}+1\leq3$ and $\Pi_{G,p}(N)\leq 6N\leq 6D^4$, from~\cite[Proposition 7.1]{humemackaytessera}.

Then, we deduce the inequality \eqref{croiss_exp}. One may notice that, in this situation, conserving only the last term of the sum can't lead to a dramatic loss, since $ \sum_{n=0}^{K} D^{n}\asymp D^{K}  $, and $ \rho_{f} $ is non-decreasing. This ends the proof of Theorem~\ref{thm_borne_sup}.
\end{proof}
\subsection{Application to lamplighter diagonal products}
In this \cpt{section}{subsection}, we exhibit embeddings of lamplighter diagonal products and deduce an upper bound on their Poincaré  profile, using Theorem~\ref{thm_borne_sup}. In~\cite{brieusselzheng2015}, Brieussel and Zheng exhibit ``global'' embeddings into $ L^p $ spaces, meaning that they almost realize the compression upper bound at every scale. To do so, they use a process designed by Tessera in~\cite{Tes11}: they sum up infinitely many cocycles, such that at each cocycle realizes the compression upper bound at a particular scale. Finally, the embedding obtained covers every scale. Unfortunately, this process costs a logarithmic factor in the compression function obtained. In our context, it happens that the conclusion of Theorem~\ref{thm_borne_sup} only considers one particular value of the embedding $ f $. Therefore we can take each one of these cocycles individually, and we will avoid this logarithmic factor. We will show the following theorem:
\begin{theorem}\label{thm_up_bd_diag}
	Let $ \Delta $ be the lamplighter diagonal product of $(\Gamma_s,a_s,b_s,k_s)_{s\geq 0}$. For any $s\geq0$, we set $ l_s=\diam(\Gamma_s) $. We assume that there exists $ m_0\geq 2 $ such that for any $s\geq0$, we have $k_{s+1}\geq m_0k_s $ and $ l_{s+1}\geq m_0l_s$.
	
	Let $ \varrho_{\Delta} $ be defined as follows:
	\begin{align*}
		\varrho_\Delta \colon \R_{\geq 1}&\rightarrow\ \R_{\geq 1}
		\\x&\mapsto\ \begin{cases}
		x/l_s & \text{if $ x\in\left[k_s l_s, k_{s+1} l_{s}\right) $ }\\
		k_{s+1} & \text{if $ x\in \left[k_{s+1} l_s, k_{s+1} l_{s+1}\right) $}
		\end{cases}
	\end{align*}
	Then there exists some positive constants $ c_1,c_2 $  depending only on $ m_0 $ and on the degree of $ \Delta $ such that for any $ p\in\left[1,\infty\right)$ and any positive integer $ N $ we have:
\[\Pi_{\Delta,p}(N) \leq c_1 \frac{N}{\varrho_{\Delta}(c_2\log N)}.\]
\end{theorem}
We will simply adapt to our context the content of Section 6.2.3 of~\cite{brieusselzheng2015} ``Basic test functions and $1$-cocycles on $\Delta$''.
We start with some definitions:
\begin{definition}
Let $\Delta$ be a lamplighter diagonal product.
\begin{itemize}
\item We define the $ \Z $ projection as:
	\begin{align*}
	 p_{\Z}\colon\Delta&\rightarrow\Z\\
	 \left(\left(f_s \right)_{s\geq 0},i \right) &\mapsto i
	\end{align*}
For any subset $S\subset\Delta$, we define $\range(S)=\diam\set{p_{\Z}(z),\,z\in S}$. For any $ z\in\Delta $, we define its \textbf{range} as 
\[\range(z)= \min\left\{\range\left(\gamma_{1,z}\right)\mid \gamma_{1,z} \text{ is a path from $1$ to $z$} \right\}.\]
Roughly speaking, it is the minimal diameter of the intervals of $ \Z $ visited by the cursor when following a path linking $ 1 $ and $ z $.
\item We define for any $ r\geq 2 $ a subset $ U_r $ of $ \Delta $ as
	\[U_r=\set{z\in\Delta\mid\range(z)\leq r}.\]
\item For any $g\in\Delta$, and $\varphi\colon\Delta\to X$, $ \tau_g\varphi$ denotes the $g$-right translate of $\varphi$:
\[\tau_g \varphi(h) = \varphi \left(h g^{-1}\right),\quad\text{for any $h\in\Delta$.
}\]\item We finally define  \[ \varphi_{r}\left(\left(f_s\right),i\right) = \max \left\{0,1-\frac{\card{i}}{r}\right\} \mathbf{1}_{U_{r}}\left(\left(f_s\right ),i\right),\]
	and, for every $j\ge1$,
	\begin{align*}
		\Phi_{j}\colon\Delta &\rightarrow  \ell^{2}\left(\Delta\right)
		\\Z&\mapsto \frac{\varphi_{2^{j}} - \tau_z\varphi_{2^j}}{\twonorm{\nabla\varphi_{2^j}}},
	\end{align*}
\end{itemize}
\end{definition}
As shown by the following lemma, the family of $1$-cocycles $ \left(\Phi_{j}\right)_{j\ge1} $ captures the size of $ \range(z) $.
\begin{lemma}\label{l:ing2psscjplsn}
Let $j\ge1$. For any $z\in\Delta$ satisfying $\range(z)>2^{j+1}$, we have
\[\twonorm{\Phi_j(z)}\ge\frac{2^{j}}{3}.\]
\end{lemma}
\begin{proof}
Let $j\geq1$ and $z\in\Delta$ be such that $\range(z) > 2^{j+1}$.
		
By definition, of $\varphi_r$, any element $w$ of $\support(\varphi_{2^j})$ satisfies $\range(w)\leq2^j$. Let now $w$ be an element of $\support(\tau_z\varphi_{2^j})$. It satisfies $\range(wz^{-1})\leq2^j$. Then, there is a path $\gamma_{w,z}$ from $w$ to $z$ such that $\range(\gamma_{w,z})\leq2^j$. Hence, if $\gamma_{1,w}$ is a path from $1$ to $w$, then $\gamma_{1,z}=\gamma_{1,w}\cup\gamma_{w,z}$ is a path from $1$ to $z$. By assumption, we can deduce that we have $\range(\gamma_{1,z})>2^{j+1}$. This implies $\range(\gamma_{1,w})>2^j$, and since this is true for any path from $1$ to $w$, we obtain $\range(w)>2^j$. Then,
		\[ \support(\varphi_{2^j}) \cap \support(\tau_z \varphi_{2^j}) = \emptyset.\]
Therefore,
\[ \twonorm{\varphi_{2^{j}} - \tau_z\varphi_{2^j}}^2 = 2 \twonorm{\varphi_{2^{j}}}^2.\]
Let us write $r=2^j$. We set $U_r^0=\set{g\in U_r\mid p_\Z(g)=0}$. Then, any element of $U_r$ can be written $g\tau^{i}$, with $g\in U_r^0$ and $i\in[-r,r]$. Then,
\begin{align*}
\twonorm{\varphi_{r}}^2&=\sum_{g\in U_r^0}\sum_{i\in[-r,r]} \left(1-\frac{\abs{i}}r\right)^2\\
&\ge\card{U_r^0}\frac r6.
\end{align*}
Let $g\in\Delta$. For any $a\in A$, and $b\in B$, we have $\range(g)=\range(ga)=\range(gb)$, which implies $\varphi_r(g)=\varphi_r(ga)=\varphi_r(gb)$. Then,
\begin{align*}
\twonorm{\nabla\varphi_r}^2&=\sum_{g\in\Delta}\abs{\varphi_r(g)-\varphi_r(g\tau)}^2\\
&=\sum_{g\in U_r}\abs{\varphi_r(g)-\varphi_r(g\tau)}^2\\
&\leq\frac{\card{U_r}}{r^2}\\
&\leq3\frac{\card{U_r^0}}{r}.
\end{align*}
Therefore we have, for any $z\in\Delta$ satisfying $\range(z)>2^{j+1}$,
\[\twonorm{\Phi_j(z)}^2= \frac{\twonorm{\varphi_{2^{j}} - \tau_z\varphi_{2^j}}^2}{\twonorm{\nabla\varphi_{2^{j}}}^2}\geq\frac{r^2}{9}=\frac{2^{2j}}{9}.\qedhere\]
\end{proof}

	\begin{proof}[Proof of Theorem~\ref{thm_up_bd_diag}]
		For any $j\ge0$, $\Phi_j$ satisfies the following identity:
\[\stepcounter{equation}\tag{\theequation}\label{e:cocycle}
\Phi_{j}(gh)=\Phi_{j}(g)+\tau_g\Phi_{j}(h),\]
for any $g,h\in\Delta$ (this is a cocycle identity). Moreover $ \twonorm{\Phi_{j}(z)} = 0 $ if $z$ is a generator in $ A\cup B $ and $\twonorm{\Phi_{j}(z)} \leq 1 $ if $z$ is a generator in $ \Z$. Therefore $ \Phi_{j} $ is $ 1- $Lipschitz.

As noticed in the proof of Lemma 6.9 of~\cite{brieusselzheng2015}, we have, for any $z\in\Delta$, 
\[\stepcounter{equation}\tag{\theequation}\label{lm69BZ}
\range(z)\in \left[k_s,k_{s+1}\right)\implies \card{z}_\Delta \leq \frac{9000 (\range(z)+1) l_s}{1 - 1/m_0}.\]

Let $s\geq1$. Let $r\in[k_s,k_{s+1})$, and let $j$ such that $2^{j+1}<r\leq2^{j+2}$. We set $t=\frac{9000(1+2/m_0)}{1-1/m_0}rl_s$. We will show that we have
\[\stepcounter{equation}\tag{\theequation}\label{e:bdcmprssn}\rho_{\Phi_j}(t)\geq\frac r{12} .\]
Let then $z\in\Delta$ be such that $\abs{z}_\Delta\geq t$. This implies in particular $\card{z}_\Delta\geq\frac{9000 (r+1) l_s}{1 - 1/m_0}$. If $\range(z)<r$, then $\card{z}_\Delta>\frac{9000 (\range(z)+1) l_s}{1 - 1/m_0}$. This implies, from~\eqref{lm69BZ}, that we have $\range(z)\geq k_{s+1}$, which is a contradiction. Then, we have $\range(z)\geq r>2^{j+1}$. From Lemma~\ref{l:ing2psscjplsn}, we deduce $\twonorm{\Phi_j(z)}\geq\frac{2^j}3\geq\frac r{12}.$ This implies, from the cocycle identity~\eqref{e:cocycle}, that for any $z_1,z_2\in\Delta$ such that $\card{z_1z_2^{-1}}>t$, we have $\twonorm{\Phi_j(z_1)-\Phi_j(z_2)}=\twonorm{\Phi_j(z_1z_2^{-1})}\geq\frac r{12}$, which proves~\eqref{e:bdcmprssn}.

Since $\ell^{2}$ embeds isometrically in $L^p$ for all $p\geq 1 $ (see Lemma 2.3 of~\cite{NPspeed}), we obtain that for every $p\in[1,\infty)$, $s\geq1$ and $r\in[k_s,k_{s+1})$, there exists a $1$-Lipschitz map $\Phi_r^p\colon\Delta\to L^p$ such that, if we write $\rho_r^p$ the compression function of $\Phi_r^p$,
\[\stepcounter{equation}\tag{\theequation}\label{e:bdcmprssn2}\rho_r^p\left(Crl_s\right)\geq\frac r{12},\quad\text{with $C=\frac{9000(1+2/m_0)}{1-1/m_0}$.}\]
From Theorem~\ref{thm_borne_sup}, there exists two constants $c_1$ and $c_2$ depending only on the degree of $\Delta$ such that for each $p\in[1,\infty)$, $s\geq1$ and $r\in[k_s,k_{s+1})$, we have for every $n\geq 0$,
\[\stepcounter{equation}\tag{\theequation}\label{e:upprbndthm}
\Pi_{\Delta,p}(n)\leq\frac{c_1}{\rho_r^p(c_2\log n)}.\]
Let $n\ge0$. There exists $s\ge0$ such that $c_2\log n\in[Ck_sl_s,Ck_{s+1}l_{s+1}].$ Without loss of generality, we can assume that $s\ge1$. Two cases can occur:
\begin{enumerate}
\item If $c_2\log n\in[Ck_sl_s,Ck_{s+1}l_s]$, then, if we set $r=\frac{c_2\log n}{Cl_s}$, and $x=\frac{c_2\log n}C$, we have
\begin{align*}
\Pi_{\Delta,p}(n)&\le\frac{c_1n}{{\rho^p_r}(c_2\log n)}\quad\text{from~\eqref{e:upprbndthm}}\\
&=\frac{c_1n}{{\rho^p_r}(Crl_s)}\\
&\le\frac{12c_1n}{r}\quad\text{from~\eqref{e:bdcmprssn2}}\\
&=\frac{12c_1n}{\varrho_\Delta(\frac{c_2}C\log n)}.
\end{align*}
\item If $c_2\log n\in[Ck_{s+1}l_s,Ck_{s+1}l_{s+1}]$, then $c_2\log N\ge C\frac{k_{s+1}}2l_s\ge Ck_sl_s$. Then, we have
\begin{align*}
\Pi_{\Delta,p}(n)&\le\frac{c_1n}{{\rho^{p}_{\frac12({k_{s+1})}}}(c_2\log n)}\quad\text{from~\eqref{e:upprbndthm}}\\
&\le\frac{c_1n}{{\rho^{p}_{\frac12({k_{s+1})}}}(C\frac{k_{s+1}}2l_s)}\\
&\le\frac{24c_1n}{k_{s+1}}\quad\text{from~\eqref{e:bdcmprssn2}}\\
&=\frac{24c_1n}{\varrho_\Delta(k_{s+1}l_{s+1})}\\
&\leq\frac{24c_1n}{\varrho_\Delta(\frac{c_2}C\log n)}.
\end{align*}
\end{enumerate}
This ends the proof of Theorem~\ref{thm_up_bd_diag}.
\end{proof}
\section{Comparison of the bounds}\label{s:comparison}
We compare the bounds obtained in \cpt{Chapter}{Sections}~\ref{s_low_bd_sep} and~\ref{s_up_bd} to prove Theorems~\ref{thm_presc_sep},~\ref{thm_presc_sep_approx} and~\ref{thm_better_bound}. We start with some definitions.
\begin{definition}
	Let $ \rho \colon \R_{\geq 1}\rightarrow\ \R_{\geq 1}$ be an non-decreasing function. For any $  \alpha\in \left[0,1\right]$ and $\beta>0$, we say that $ \rho $ satisfies the condition~\eqref{d_s_alph_beta} if it is injective and moreover there exists $ C > 0 $ such that
	\[
	\tag{$ S_{\alpha,\beta} $}\label{d_s_alph_beta} \rho^{-1}\left(\frac{x^{1/\beta}}{C}\right) \leq \frac{\rho^{-1}(x)}{x^{1-\alpha}},\quad\text{for any large enough $ x $.}\]

	Let $ \rho \colon \R_{\geq 1}\rightarrow\ \R_{\geq 1}$ be an non-decreasing function. We say that $ \rho $ is \textbf{strongly sublinear} if it is injective and moreover there exists $ C > 0 $ such that
	\[
	\tag{SSL}\label{d_str_linea} \rho^{-1}\left(\frac{x}{C}\right) \leq \frac{\rho^{-1}(x)}{x},\quad\text{for any large enough $ x $.}\]
\end{definition}
\begin{remark}
We can make two simple remarks. First, it is obvious that condition \eqref{d_str_linea} is the same as \eqref{d_s_alph_beta} with $ \alpha= 0 $ and $ \beta = 1 $. It has its own name because it will play a particular role in the proofs.

Second, it is clear that every function satisfies the condition~\eqref{d_s_alph_beta} with $\alpha=1$ and $\beta=1$, with $C=1$.
\end{remark}
Let us detail these two conditions.
\paragraph{Condition~\eqref{d_s_alph_beta}}
For every $ a\in\left(0,1\right)$, $x\mapsto x^{a}$ satisfies condition \eqref{d_s_alph_beta} with $ \alpha = 0 $, $\beta=\frac1{1-a}$ and $ C = 1 $. We have the following proposition:
\begin{proposition}\label{p_exists_sab}
	Let $ \rho \colon \R_{\geq 1}\rightarrow\ \R_{\geq 1}$ be an increasing function such that there exists some $ a\in\left(0,1\right) $ such that $ \frac{\rho^{-1}}{x^{1/a}} $ is non-decreasing. Then $ \rho $ satisfies \eqref{d_s_alph_beta}  with $ \alpha = 0 $ and $\beta=\frac1{1-a}$, with $ C = 1 $.
\end{proposition}
\begin{proof}
For any $x\ge1$, we have $x\ge x^{1/\beta}$, which implies $\frac{\rho^{-1}(x)}{\rho^{-1}(x^{1/\beta})}\ge \frac{x^{1/a}}{x^{1/\beta a}}=x$.
\end{proof}

\paragraph{Condition~\eqref{d_str_linea}} The intuition behind condition \eqref{d_str_linea} is the following: a change of scale for $ \rho^{-1} $ is able to compensate the division by the identity function. We think of $ \rho^{-1} $ as ``big'', and therefore think of $ \rho $ as ``small''. For example:
\begin{itemize}
	\item if $ \rho $ is of the form $x\mapsto x^{\alpha} $, with $ \alpha\in\left(0,1\right) $, condition \eqref{d_str_linea} is \textit{not} satisfied, since $ \rho^{-1} $ is a power function.
	\item if $ \rho $ is of the form $x\mapsto(\log x)^{\alpha} $, with $ \alpha>0 $, condition \eqref{d_str_linea} is satisfied, since $ \rho^{-1} $ is a power function composed with the exponential.
\end{itemize}
The following proposition gives more examples of functions satisfying \eqref{d_str_linea}. Roughly speaking, it states that any function $ \rho $ lower than $ \log(n)^{1/\alpha} $ satisfies \eqref{d_str_linea}.
\begin{proposition}\label{p_exists_ssl}
	Let $ \rho \colon \R_{\geq 1}\rightarrow\ \R_{\geq 1}$ be an increasing function such that there exists some $ \alpha > 0 $ such that $ \frac{\rho^{-1}}{\exp(x^{\alpha})} $ is non-decreasing. Then $ \rho $ satisfies \eqref{d_str_linea} for any $ C>1 $.
\end{proposition}
\begin{proof}
Let $C>1$. Then, for any $x\ge1$, we have $x\ge x/C$, which implies $\frac{\rho^{-1}(x)}{\rho^{-1}(x/C)}\ge\frac{\exp(x^\alpha)}{\exp(x^\alpha/C^\alpha)}=\exp\left(\frac{x^\alpha}{C^\alpha}(C^\alpha-1)\right)$. We conclude by noticing that this last term is more than $x$, if $x$ is large enough.
\end{proof}
We can state our main theorem.
\begin{theorem}\label{thm_better_bound_gener}
		There exist a universal constant $\kappa_1$ such that the following is true. Let $\rho\colon\R_{\geq 1}\rightarrow\R_{\geq 1}$ be a non-decreasing function  such that $ \frac{x}{\rho(x)} $ is non-decreasing and $ \lim_{\infty}\rho = \infty $. We assume that $\rho$ satisfies~\eqref{d_s_alph_beta} for $\alpha\in[0,1]$ and $\beta>0$.
		
		Then, there exists a positive constant $\kappa_2$, that only depends on $\beta$, and a finitely generated elementary amenable group $ \Delta $ of exponential growth and of asymptotic dimension one such that for any $p\in\left[1,\infty\right)$,
	\begin{align*}
	\Pi_{\Delta,p}(n) &\leq\kappa_{1}\frac{n}{\rho(\log n)}\quad\text{for any $ n $,}\\
	\text{and}\quad\Pi_{\Delta,p}(n)&\geq 4^{-p}\kappa_2\frac{n}{\big(\rho(\log n)\big)^{\beta(1+\alpha)}}\quad\text{for infinitely many $ n $'s.}
	\end{align*}
	Moreover, when $\beta\leq2$, $\kappa_2$ can be chosen independent of $\beta$.
\end{theorem}
\begin{remark}\label{r:partcases}
\begin{itemize}
\item Theorem~\ref{thm_presc_sep} is a particular case of Theorem~\ref{thm_better_bound_gener}, with $\alpha=0$ and $\beta=1$. Indeed, with these values for $\alpha$ and $\beta$, condition~\eqref{d_s_alph_beta} is the same as condition~\eqref{d_str_linea}, and this condition is implied by the assumptions made on~$\rho$, from Proposition~\ref{p_exists_ssl}. This gives an exponent $\beta(1+\alpha)=1$ on the lower bound.

\item Theorem~\ref{thm_presc_sep_approx} is a particular case of Theorem~\ref{thm_better_bound_gener}, with $\alpha=1$ and $\beta=1$. Indeed any function satisfies condition~\eqref{d_s_alph_beta} with these values for $\alpha$ and $\beta$. This gives an exponent $\beta(1+\alpha)=2$ on the lower bound.

\item Theorem~\ref{thm_better_bound} is a particular case of Theorem~\ref{thm_better_bound_gener} with $\alpha=0$, $\beta=\frac1{1-a}$. Indeeed, with these values for $\alpha$ and $\beta$, condition~\eqref{d_s_alph_beta} is implied by the assumptions made on $\rho$, from Proposition~\ref{p_exists_sab}. In the statement Theorem~\ref{thm_better_bound}, we make the assumption that $a\in(0,1/2)$ because if $a\geq\frac12$, if $\alpha=0$ and $\beta=\frac1{1-a}$, then we have $\beta(1+\alpha)\ge2$. In that case, Theorem~\ref{thm_better_bound_gener} do not improve the lower bound of Theorem~\ref{thm_presc_sep_approx}. When $a\in(0,1/2)$, then $\beta=\frac1{1-a}\le2$ and $\kappa_2$ can be chosen universal. This gives an exponent $\beta(1+\alpha)=\frac1{1-a}$ on the lower bound.
\end{itemize}
\end{remark}

We can prove Theorem~\ref{thm_better_bound_gener}.
\begin{proof}[Proof of Theorem~\ref{thm_better_bound_gener}]
	We set $(\Gamma'_{m_s})_{s\geq 0}$ to be the aforementioned sequence of Lafforgue super expanders (see Example~\ref{e_lafforgue}), say with $q=2$, indexed such that, for every $s\ge0$, $\card{\Gamma'_{m_s}}=m_s$.

Let $ \rho $ be a function satisfying the assumptions of Theorem~\ref{thm_better_bound_gener}. We can model the process of~\cite[Proposition B.2.]{brieusselzheng2015} and get two increasing sequences of integers $ k_s $ and $  n_s $ such that 
\begin{enumerate}[label=(\roman*)]
 \item The sequence $(n_s)_{s\ge0}$ is a subsequence of $(m_s)_{s\ge0}$. Then, we can set $l_{s}=\diam \Gamma'_{n_s}$.
 \item We have $k_0=0$, $k_1\ge3$, $k_{s+1}\geq 3k_s $ and $ l_{s+1}\geq 3l_s$ for every $s\ge0$.
\item\label{equiv_rho} There is a universal constant $c$ such that if we define $\tilde{\rho}$ by:
	\[ \tilde{\rho}(x) = \begin{cases}
	x/l_s & \text{if $ x\in\left[k_s l_s, k_{s+1} l_{s}\right) $ }\\
	k_{s+1} & \text{if $ x\in \left[k_{s+1} l_s, k_{s+1} l_{s+1}\right), $}
	\end{cases} \]	
	then we have \[c^{-1}\rho(x)\le\tilde{\rho}(x)\le c\rho(x),\quad\text{for any $x\ge1$.}\]
 \end{enumerate}
Moreover, since the function $x\mapsto\frac{x}{\rho(x)}$ is non-decreasing, we have, for any $a,x\ge1$,
\[\label{e:sublinear}\stepcounter{equation}\tag{\theequation}\rho(ax)\leq a\rho(x).\]
	For any $ s $, we set $ \Gamma_s \vcentcolon= \Gamma'_{n_s}$ . Let now $ \Delta $ be the lamplighter diagonal product associated with $(\Gamma_s,a_s,b_s,k_s)_{s\geq 0}$, using the notations of Definition~\ref{d_lamp_diag_prod}.	
To get the upper bound of Theorem~\ref{thm_better_bound_gener}, we can apply Theorem~\ref{thm_up_bd_diag} to $\Delta$. Then, by construction, $ \varrho_{\Delta} = \tilde{\rho} $, and therefore $ c^{-1}\rho\le\varrho_{\Delta}\le c\rho$. Then, there are universal constants $c_1$ and $c_2$ such that, for any $n\ge0$,
\begin{align*}
\Pi_{\Delta,p}(n)&\leq c_1\frac{n}{\varrho_{\Delta}(c_2\log n)}\\
&\leq c_1c_2^{-1}\frac{n}{\varrho_{\Delta}(\log n)}\quad\text{from~\eqref{e:sublinear},}
\stepcounter{equation}\tag{\theequation}\label{upprbndprf}
\end{align*}
which gives the upper bound of Theorem~\ref{thm_better_bound_gener}.\\

The lower bound requires more calculation. We will use the following facts:
\begin{enumerate}[label=(\roman*)]\setcounter{enumi}{3}
\item\label{e:diamexp} There is a constant $c_3$ such that $\diam \Gamma_s\leq c_3\log \card{\Gamma_s}$, for every $s\ge0$ (see~\cite[Example 2.3.]{brieusselzheng2015}).
\item\label{e:lwbndls} From~\ref{equiv_rho}, we have $ c^{-1} k_s \leq \rho(k_s l_s) \leq c k_s $, for any $ s $. In particular, since $\rho$ is non-decreasing, this implies $ l_s \geq \frac{\rho^{-1}(c^{-1} k_s)}{k_s} $. 
\item\label{e:lffrgxpndrs} The sequence $(\Gamma_s)_{\ge0}$ is an expander: from Theorem~\ref{thm:lffrgxpndrs} and Proposition~\ref{p:cmpcmbntrlchgrcnstns}, there is $D,\epsilon>0$ such that we have $\deg\Gamma_s\le D$ and $h(\Gamma_s)\geq\epsilon$, for every $s\ge0$.
\end{enumerate}
We fix $p\in[1,\infty)$. We assume that $ \rho $ satisfies~\eqref{d_s_alph_beta} with $\alpha\in[0,1]$, and $\beta>0$. Let $s\ge1$. We apply Theorem~\ref{thm_low-bnd-sep} with $r=\lfloor k_s^{\alpha}/2\rfloor$.  We get
\[\stepcounter{equation}\tag{\theequation}\label{e:lwrbndsab} 
\Pi_{\Delta,p}\left(N_s\right) \geq 4^{-p}\frac{h(\Gamma_s)^{2}}{1536(\deg \Gamma_s)^2} \frac{N_s}{(2k_s + 2\lfloor k_s^{\alpha}/2\rfloor+1)(2\lfloor k_s^{\alpha}/2\rfloor+1)},\]
with $N_s=\card{\Gamma_s}^{2\lfloor k_s^{\alpha}/2\rfloor+1} \times(2k_s+2\lfloor k_s^{\alpha}/2\rfloor+1)\ge\card{\Gamma_s}^{k^\alpha/2} $. Then,
\begin{align*}
\log N_s &\geq \frac{k_s^{\alpha}}2\log\card{\Gamma_s}\\
&\ge (2c_3)^{-1}k_s^\alpha l_s\quad\text{from~\ref{e:diamexp}}\\
&\geq (2c_3)^{-1} \frac{\rho^{-1}(c^{-1} k_s)}{k_s^{1-\alpha}}\quad\text{from~\ref{e:lwbndls}}\\
&= {(2c^{1-\alpha}c_3)}^{-1} \frac{\rho^{-1}(c^{-1} k_s)}{({c^{-1}k_s})^{1-\alpha}}\\
&\geq {(2c^{1-\alpha}c_3)}^{-1}\rho^{-1}\left(\frac {c^{-1/\beta} k_s^{1/\beta}}{ C}\right)\quad\text{from \eqref{d_s_alph_beta}, if $ s $ is large enough.}
		\end{align*}
	Then, since $\rho$ is non-decreasing, we obtain $k_s\le C^\beta c\Big(\rho(2c^{1-\alpha}c_3\log N_s)\Big)^\beta$. Moreover, we have $(2k_s + 2\lfloor k_s^{\alpha}/2\rfloor+1)(2\lfloor k_s^{\alpha}/2\rfloor+1)\leq8k_s^{1+\alpha}$. Therefore, combining with~\ref{e:lffrgxpndrs} and~\eqref{e:lwrbndsab}, we obtain, for every large enough $s$:
\begin{align*}
\Pi_{\Delta,p}\left(N_s\right) &\geq 4^{-p}\frac{\epsilon^{2}}{12288D^2C^{\beta(1+\alpha)}c^{1+\alpha}}\frac{N_s}{\Big(\rho(2c^{1-\alpha}c_3\log N_s)\Big)^{\beta(1+\alpha)}}\\
&\geq 4^{-p}\kappa_2(\alpha,\beta)\frac{N_s}{\Big(\rho(\log N_s)\Big)^{\beta(1+\alpha)}},
\end{align*}
with
\[\kappa_2(\alpha,\beta)=\dfrac{\epsilon^{2}}{12288D^2C^{\beta(1+\alpha)}c^{(1+\alpha)(1+\beta(1-\alpha))}(2c_3)^{\beta(1+\alpha)}}\quad\text{(here, we use~\eqref{e:sublinear})}.\]
Since $\alpha\in(0,1)$, we can deduce
\[\kappa_2(\alpha,\beta)\ge\frac{\epsilon^{2}}{12288D^2C^{2\beta}c^{2(1+\beta)}(2c_3)^{2\beta}},\]
which proves that $\kappa_2$ can be chosen independent of $\alpha$. If moreover $\beta\le2$,
\[
\kappa_2(\alpha,\beta)\geq\frac{\epsilon^{2}}{49152D^2C^4c^6c_3^4},
\]
which proves that, in that case, $\kappa_3$ can be chosen independent of $\beta$. This ends the proof of Theorem~\ref{thm_better_bound_gener}.
\end{proof}\nocite{loukasvandergheynst_spectrallyapproximating}

\begin{remark}\label{remkfactcontinuum}
Fact~\ref{fact_continuum} (from the proof of Theorem~\ref{thm:cntinuum}) uses an important feature of this proof: we have explicit values for the integers $N_s$ where the lower bounds on Poincar\'e profiles are known to be valid.
More precisely, Theorem~\ref{thm:cntinuum} relies on Theorem~\ref{thm_better_bound_gener} with $\alpha=0$ and $\beta=1$.
In that case, we have $N_s=\card{\Gamma_s}\times(2k_s+1)$. The contruction of~\cite[Proposition B.2.]{brieusselzheng2015} shows that, in the case of functions satisfying condition~\eqref{d_str_linea}, we can take $k_s=3^s$. Then, it is clear from the condition~\eqref{equiv_rho} that the sequence $N_s$ will be sparser when~$\rho$ grows slower. This is roughly what is stating Fact~\ref{fact_continuum}.
\end{remark}

\begin{remark}
The lower bounds are obtained by exhibiting families of subgraphs of the group $\Delta$. These subgraphs are isomorphic to graphs of the family $\Gamma_s^{k_s,r}$, which consist of Cartesian products of $2r+1$ copies of the lamp groups $\Gamma_s$, ``distorted'' by a scale factor $k_s$, see Definition~\ref{dist_lamp_goup}. From Proposition~\ref{plongement_expanseur_dilate}, these graphs are isomorphic to subgraphs of $\Delta$ when $r$ is at most $k_s/2$. The choice of $r$ is made so that we obtain the highest lower bound. In the proof of Theorem~\ref{thm_better_bound_gener}, we take $r$ to be equal to $\lfloor k_s^{\alpha}/2\rfloor$, where $\alpha$ is such that $\rho$ satisfies condition~\eqref{d_s_alph_beta}. Then, for such a $\rho$, we obtain the lower bound of Theorem~\ref{thm_presc_sep_approx} considering $1+2\lfloor k_s^{\alpha}/2\rfloor$ copies of the lamp groups. To apply Theorem~\ref{thm_better_bound_gener} to a given function $\rho$, one needs to find a couple $(\alpha,\beta)$ that minimizes the exponent of the lower bound $\beta(1+\alpha)$. Let us detail this fact in our applications.

In Theorem~\ref{thm_presc_sep_approx}, we consider general functions $\rho $. This case corresponds to Theorem~\ref{thm_better_bound_gener} with $\alpha=1$ and $\beta=1$, see Remark~\ref{r:partcases}. Then $r\simeq k_s/2$. That means that the lower bound is obtained considering the maximal number of copies of the lamp groups. This gives a lower bound of the form $\frac n{(\rho\log n)^2}$, that doesn't match with~\eqref{upprbndprf}.

In Theorem~\ref{thm_presc_sep}, we consider functions $\rho $ growing slower than $\log$, namely condition~\eqref{d_str_linea}. This case corresponds to Theorem~\ref{thm_better_bound_gener} with $\alpha=0$ and $\beta=1$, see Remark~\ref{r:partcases}. Then $r=0$ and $2r+1=1$. That means that the lower bound is obtained considering single copies of the lamp groups, namely the graphs~$\Gamma_s^{k_s,0}$, which are homothetic copies of $\Gamma_s$, see Proposition~\ref{p:hmthtclmpgrps}. This gives a lower bound of the form $\frac n{\rho(\log n)}$, which is optimal, from~\eqref{upprbndprf}.

Nevertheless, when $\rho$ grows faster than $\log(x)$ we loose this matching. Indeed, if we consider $a\in(0,1)$, then $x^a$ satisfies condition~\eqref{d_s_alph_beta} with $\alpha=0$ and $\beta=\frac1{1-a}$. The lower obtained with Theorem~\ref{thm_better_bound_gener} is of the form $\frac n{(\rho\log n)^{1/(1-a)}}$. As above, since $\alpha=0$, it is obtained considering single copy of the lamp groups. We see that this lower bound gets worse when $\alpha$ increases, and that the exponent $\frac1{1-a}$ goes beyond $2$ when $a$ is more than $1/2$. Hence, despite Theorem~\ref{thm_better_bound} also applies for $a>1/2$, it is better to use the general Theorem~\ref{thm_presc_sep_approx}.

The case of power functions is very instructive. Let $a\in(0,1)$ and $\rho\colon x\mapsto x^a$, and let $\Delta$ be the associated group (as in the proof of Theorem~\ref{thm_better_bound_gener}). Then, as explained before, we can take for any $\alpha\in[0,1]$ a family of subgraphs of the form $\Gamma_s^{k_s,r}$, with $r\simeq k_s^\alpha$. Then, after a short calculation, we obtain a lower bound on the form $\frac n{(\log n)^\gamma}$, with $\gamma=\frac{1+\alpha}{1-a(1-\alpha)}$.
\begin{itemize}
\item If $a>1/2$, $\gamma$ is minimized with $\alpha=0$. In this case $\gamma=\frac1{1-a}$. We recover Theorem~\ref{thm_better_bound}.
\item If $a<1/2$, $\gamma$ is minimized with $\alpha=1$. In this case $\gamma=2$. We recover Theorem~\ref{thm_presc_sep_approx}.
\item If $a=1/2$, $\gamma=2$ for any $\alpha\in[0,1]$. In this case, any subgraph of the form $\Gamma_s^{k_s,r}$, with $r\ge k_s/2$ gives a lower bound of the form $\frac n{(\rho\log n)^2}$.
\end{itemize}
\end{remark}

\appendix
\section{Separation of distorted graphs.}\label{s_chgrcnstdstrtdgfphs}
In this appendix, we address the following question:
\begin{center}
\textbf{If a graph is distorted, how much can his separation decrease?}
\end{center}
Indeed, the same question could be asked for Cheeger constants. The equivalence of Proposition~\ref{prop:equivseppoinc} shows that these questions are closely related.

The toy example we have in mind is the following: let $\Lambda$ be a finite graph. Let $\kappa$ be an integer. Let $\Gamma$ be the graph obtained adding $ \kappa $ vertices along each edge of $\Lambda$. How can be compared the separation properties of $\Gamma$ with those of $\Lambda$?\\

We give three methods of answering this question.
The first is called \textit{combinatorial}. It is based on the notion of coarsening of graphs, and is very close to the proof of Proposition~\ref{p_cut_expanseurs_dilates}. The second is called \textit{geometric} because it is based on a metric assumption. The third is called \textit{analytic} because it concerns $L^p$-Cheeger constants of metric measure spaces, where graphs are considered as simplicial complexes. These three methods apply in the aforementioned toy example, see Corollaries~\ref{c:distorted1}, \ref{c:distorted2} and~\ref{c:distorted3}. They can also provide alternative proofs of Proposition~\ref{p_cut_expanseurs_dilates}, see Corollaries~\ref{c:altproof1}, \ref{c:altproof2} and~\ref{c:altproof3}.

	\subsection{Combinatorial method: coarsenings}\label{s_coar_part}
	In this \cpt{section}{subsection}, we study the separation of coarsenings of graphs. See \cite{loukasvandergheynst_spectrallyapproximating} for a more precise study of this notion, in the context of spectral graph theory.
	
		For any graph $ \Gamma $ and any subset $A\subset V\Gamma$, we will still denote by $A$ the graph of vertex set $A$ obtained by taking every edge of $\Gamma$ of the form $\set{a,a'}$, with $a,a'\in A$.
		
		For any graph $ \Gamma $ and any subset $C\subset V\Gamma$, we denote $\Gamma\setminus C$ the graph obtained removing $C$, and the edges having an endpoint in $C$, to the graph $\Gamma$.
			\begin{figure}[!ht]
	\centering
	\includegraphics[width=0.35\linewidth]{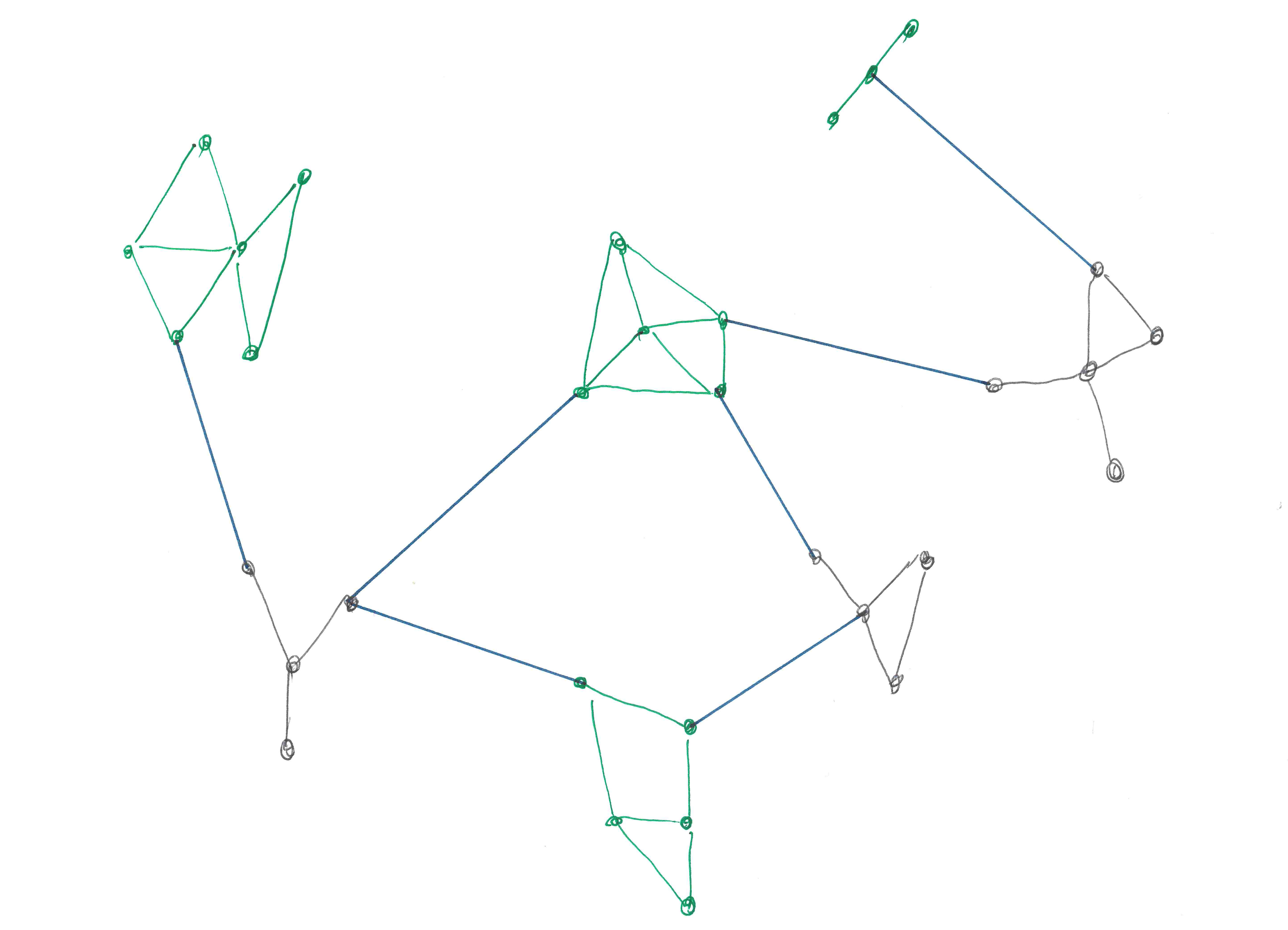}
	\includegraphics[width=0.35\linewidth]{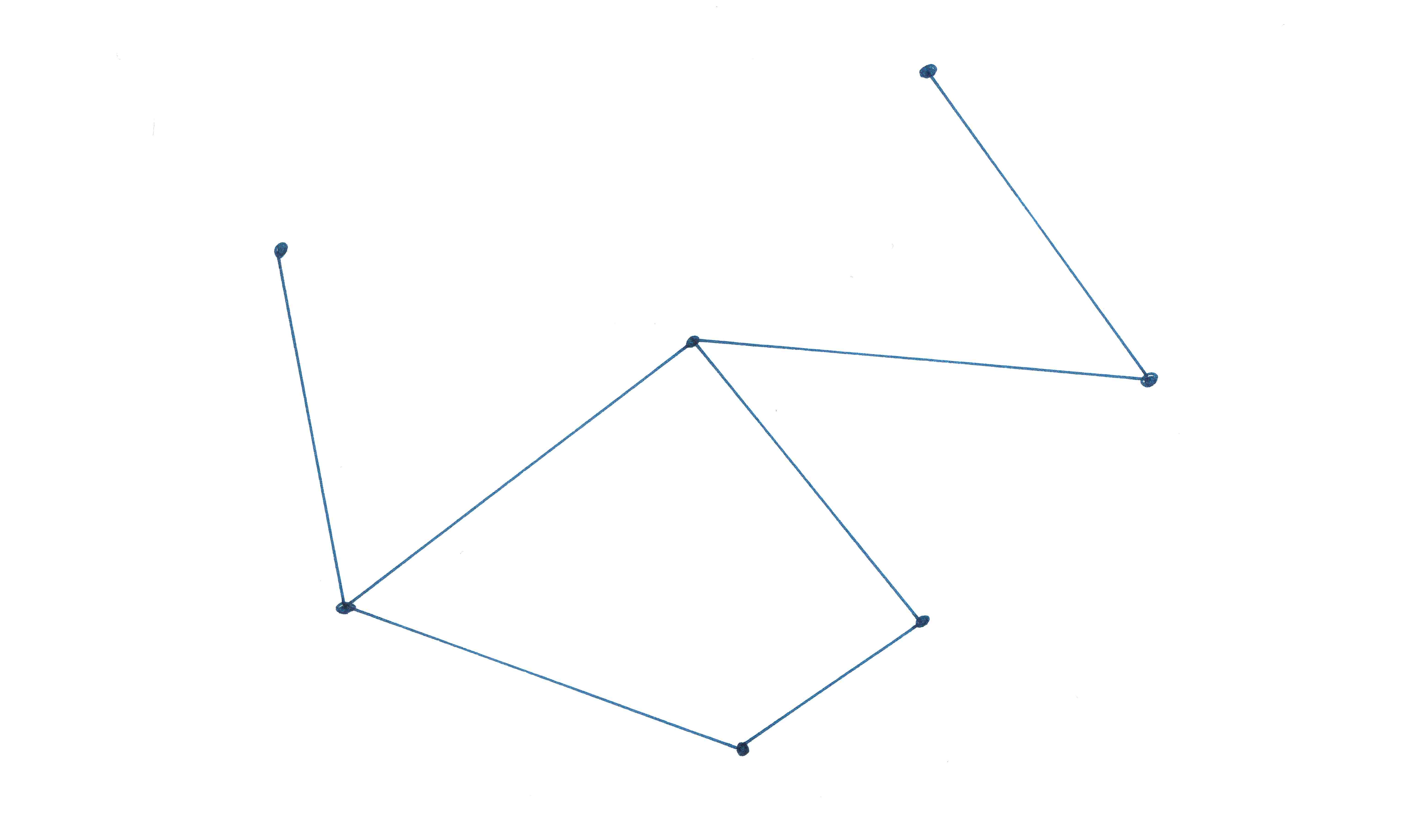}
	\caption{An example of a regular coarsening $ \Gamma $ (left) and $ \Gamma_\mathcal{A} $ (right)}
	\label{fig_ex_coar}
	\end{figure}
\begin{definition}
Let $\Gamma$ be a finite graph, let $s\in(0,1)$. We will say that a subset $C\subset V\Gamma$ is an \textbf{$s$-cut set} if every connected component of $\Gamma\setminus C$ contain at most $s\card{V\Gamma}$ vertices.

We recall moreover that the \textbf{$s$-cut} of a finite graph $\Gamma $ is the minimum size of an $s$-cut set of $\Gamma$, and that the \textbf{$s$-separation profile} of an infinite graphs maps, maps every positive integer $n$ to the supremum of the $s$-cuts of the subgraphs of $G$ having  at most $n$ vertices (see Definition~\ref{d:sprtnprfl} for details).
\end{definition}
	\begin{definition}
		Let $ \Gamma$ be a finite graph. A partition $ \left(A_i\right)_{i\in I} $ of $ V\Gamma $ is said to be \textbf{connected} if the graph $A_i$ is connected, for every $i\in I$.
		
		Given a connected partition $\mathcal{A}=\left(A_i\right)_{i\in I} $ of $ V\Gamma $, we define the \textbf{coarsened graph}, denoted by $ \Gamma_\mathcal{A} $, as the graph of vertex set $\set{A_i, i\in I}$, such that two distinct vertices $A_i$ and $A_j$ are linked by an edge if and only if there exists $(x,y)\in A_i\times A_j$ such that $\set{x,y}$ is an edge of $\Gamma$.
		
	For any subset $A\subset V\Gamma$, we define its \textbf{boundary}, denoted by $\partial A$, as the set of $ x\in A$ such that there exists $ y\in V\Gamma\setminus A$ satisfying $y\sim x$.
	
	Given a connected partition $ \mathcal{A}=\left(A_i\right)_{i\in I} $ of $ V\Gamma $, the cardinality of $\partial A_i$ will be called the \textbf{anchoring} of the set $ A_i $, denoted by $ \anch(A_i) $.
	\end{definition}
	See Figure \ref{fig_ex_coar} for an example of a regular coarsening.
	\begin{theorem}\label{thm_coar}
		Let $ \Gamma $ be a finite graph and $\Gamma_\mathcal{A}$ be coarsening associated with a partition $\mathcal{A}=\left(A_i\right)_{i\in I}$. Then		
		\[\sep_\Gamma(\card{V\Gamma})\geq\frac{\min(\card{A_i})}{8\max(\card{A_i})}\cut^{1/2}(\Gamma_\mathcal{A})\]
		On the other hand, if for any $ i\in I $ we have $ \card{A_i} \leq \frac{\card{V\Gamma}}{2} $, then
		\[\cut^{1/2}(\Gamma)\leq8\frac{\max(\card{A_i})}{\min(\card{A_i})}\max(\anch A_i)\sep_{\Gamma_\mathcal{A}}(\card{V\Gamma_\mathcal{A}}).\]
	\end{theorem}
	\begin{remark}
		If an $ A_i $ contains more than $ \frac{\card{V\Gamma}}{2} $ vertices, then $ \Gamma $ can be cut extracting $ A_i $ (removing at most $ \anch(A_i) $ vertices), and cutting it (removing at most $ \cut^{1/2}(A_i) $ vertices). This proves that, in this case, we have: \[ \cut^{1/2}(\Gamma) \leq \anch(A_i) + \max\left(\cut^{1/2}(A_i) \right) \]
	\end{remark}
Theorem~\ref{thm_coar} has the two following corollaries. The first graph concerns the toy example of the introduction of Appendix~\ref{s_chgrcnstdstrtdgfphs}, the second is a variant of Proposition~\ref{p_cut_expanseurs_dilates}.
\begin{corollary}\label{c:distorted1}
Let $ \Lambda $ be a finite graph with no isolated vertex. Let $\kappa\geq2$ be an integer. Let $ \Gamma $ be the graph obtained adding $ \kappa $ vertices along each edge of $ \Lambda $. Let $D$ be a bound on the degrees of the vertices of $\Lambda$. Then, $\Gamma$ has a subgraph $\Gamma'$ such that
\[
\cut \Gamma'\geq\frac1{24D}\cut\Lambda,
\]
and
\[\stepcounter{equation}\tag{\theequation}\label{dist2}
\cut\Gamma\leq 24D^2\sep_\Lambda(\card{V\Lambda}).
\]
\end{corollary}
\begin{proof}
$\Lambda$ can be recovered from $\Gamma$ by doing a coarsening, making a partition $(A_i)_{i\in I}$ of $\Gamma$ using balls of radius $\kappa/2$ centred at the vertices of $\Lambda$ (when $\kappa$ is odd, the middle can be associated with any of the ends of his edge). Then, we have for every $i\in I$, $\frac\kappa2+1\leq\card{A_i}\leq D\frac\kappa2+1$ when $\kappa$ is even, and $\frac{\kappa-1}2+1\leq\card{A_i}\leq D\frac{\kappa+1}2+1$ when $\kappa$ is odd. Both imply $\frac{\max\card{A_i}}{\min\card{A_i}}\leq3D$. Moreover, the anchoring of the $A_i$'s is bounded by $D$. This implies inequality~\eqref{dist2} and
\[\sep_\Gamma(\card{V\Gamma})\geq\frac1{24D}\cut^{1/2}(\Lambda),
\]
which implies that $\Gamma$ has a subgraph $\Gamma'$ such that \[\cut \Gamma'\geq\frac1{24D}\cut\Lambda.\qedhere\]
\end{proof}
\begin{corollary}\label{c:altproof1}
Let $\Gamma_s^{k_s,0}$ be as in Definition~\ref{dist_lamp_goup}, with $r=0$. Then, $\Gamma_s^{k_s,0}$ has a subgraph $\Gamma$ such that
\[\cut(\Gamma)\geq \frac18\cut(\Gamma_s).\]
\end{corollary}
\begin{proof}
This straightforward, considering the partition in \textit{lines} explained in \S\ref{s_dstrtdlmpgrps}.
\end{proof}
This statement should be compared with $Proposition~\ref{p_cut_expanseurs_dilates}$, which states, for $r=0$, $\cut(\Gamma_s^{k_s,0})\geq \cut\left({\Gamma_s}\right)$.

To prove Theorem~\ref{thm_coar}, we will use the following lemma:

		\begin{lemma}\label{l_comp_cuts}
			Let $ G $ be a finite graph, let $ s\leq 1/2 $. Then
			\[ \cut^{s}\left(G\right) \leq \frac4s\sep_G(\card{VG}). \]
		\end{lemma}
		\begin{proof}
			We will show at first that for any positive integer $ k $ we have
			\[\stepcounter{equation}\tag{\theequation}\label{comp_cuts} \cut^{\frac{1}{2^k}}\left(G\right) \leq 2^{k+1}\sep_{G}(\card{VG}).
			 \]
			This is obtained by induction on $ k $. If $k=1$, this is immediate. Let $k$ be a positive integer. By assumption, there exists a $\frac1{2^k}$-cut set of $G$ of size at most $2^{k+1}\sep_{G}(\card{VG})$. Let us call $C$ such a set. In particular, $C$ is non-empty. Then, taking unions of connected components of $VG\setminus C$, on can find a partition of $G\setminus C$ into $l$ subgraphs $G_1,\dots,G_l$ such that $G_i$ contains at most $\frac1{2^k}\card{VG}$ vertices. Up to making unions of subgraphs of $G_i$'s of size less than $\frac1{2^{k+1}}\card{VG}$, and to change the numbering, we can assume without loss of generality that for every $i\leq l-1$, $G_i$ contains at least $\frac1{2^{k+1}}\card{VG}$ vertices. Then, we have
\[\card{VG}>\card{VG}-\card{C}\geq\sum_{i=1}^{l-1}\card{G_i}\geq\frac{l-1}{2^{k+1}}\card{VG},\]
which implies $l\leq2^{k+1}$. Then, each $G_i$ can be $1/2$-cut removing a set $C_i$ containing at most $\sep_G(\card{VG})$ vertices. Then, the set $C'=C\cup C_1\cup\dots C_l$ is a $\frac1{2^{k+1}}$-cut set of $G$. We have
\begin{align*}
\card{C'}&\leq\card{C}+\sum_{i=1}^l\card{C_i}\\
&\leq2^{k+1}\sep_G(\card{VG})+l\sep_G(\card{VG})\\
&\leq\left(2^{k+1}+2^{k+1}\right)\sep_G(\card{VG})\\
&=2^{k+2}\sep_G(\card{VG}),
\end{align*}
which ends the proof of \eqref{comp_cuts}.

			Let now $ s\leq 1/2 $. Let $ k $ be the smallest integer such that $ \frac{1}{2^k} \leq s $.  Then we have $ 2^{k+1} \leq 4/s $. Therefore,
\begin{align*}
\cut^s(G)&\leq\cut^{\frac1{2^k}}(G)\\
&\leq2^{k+1}\sep_G(\card{VG})\quad\text{from~\eqref{comp_cuts}}\\
&\leq \frac4s\sep_G(\card{VG})\qedhere
\end{align*}
		\end{proof}

	\begin{proof}[Proof of Theorem \ref{thm_coar}]
For every vertex $x$ of $\Gamma$, we denote by $\bar{x}$ the unique $A_i$ that contains $x$. Then, $\bar{x}$ is a vertex of $\Gamma_\mathcal{A}$.
	
		We start with the first inequality. Let $ s\in\left(0,1\right) $. Let $ C $ be a $ s $-cut set of $ \Gamma$. Let $ C' $ be the set of vertices $ \bar{c}\in V\Gamma_\mathcal{A} $ such that there exists some $ x \in \bar{c} $ such that $ x\in C $. We have $ \card{C'} \leq \card{C}$.
		
		Let $ F' \subset V\Gamma_\mathcal{A} \setminus C'$  be such that the graph $F'$ is connected. Then we can denote by $ F $ the set of vertices $ x \in V\Gamma $ such that $ \bar{x}\in F' $. $ F $ does not meet $ C $, and moreover $ \tilde{F} $ is connected: any path in $ F' $ can be followed identically, adding some steps to cross the $ A_i $'s, which are connected by assumption.
		
		Since $ C $ is a $ s $-cut set of $ \Gamma $, we have:
\[\card{F} \leq s\card{V\Gamma}.\]
We have moreover $\card{V\Gamma} \leq \max(\card{A_i}) \times \card{V\Gamma_\mathcal{A}}$ and $\card{F'}\times \min(\card{A_i})\leq\card{F}$. Therefore we can deduce
\[ \card{F'} \leq \frac{\max(\card{A_i})}{\min(\card{A_i})} s \times \card{V\Gamma_\mathcal{A}},\]
which means that $C'$ is a $\left(\frac{\max(\card{A_i})}{\min(\card{A_i})} s\right)$-cut set of $ \Gamma_\mathcal{A} $. Then, we have shown that for any $ s\in\left(0,1\right) $, we have
\[\cut^{\frac{\max(\card{A_i})}{\min(\card{A_i})}s}(\Gamma_\mathcal{A})\leq \cut^{s}(\Gamma).\]
		In particular, for $ s = \frac12 \frac{\min(\card{A_i})}{\max(\card{A_i})} $, this gives
\begin{align*}
\cut^{1/2}(\Gamma_\mathcal{A})&\leq \cut^{s}(\Gamma)\\
&\leq\frac4s\sep_\Gamma(\card{V\Gamma})\quad\text{from Lemma~\ref{l_comp_cuts}}.\\
&=8\frac{\max(\card{A_i})}{\min(\card{A_i})}\sep_\Gamma(\card{V\Gamma}).
\end{align*}
		
		We prove now the second inequality. Then we assume that for any $ i $, $ A_i $ contains at most $ \frac{\card{V\Gamma}}{2} $ vertices. Let $ s\in\left(0,1\right) $. Let $ C' $ be a $ s $-cut set of $ \Gamma_\mathcal{A} $ of size $\cut^{s}(\Gamma_\mathcal{A})$. Let $ C $ be the set of vertices $ x $ such that $ \bar{x}\in C' $ and $ x\in\partial\bar{x} $. Then $ C $ contains at most $ \card{C'}\max(\anch({A_i}))$ vertices, and any connected subgraph of $\Gamma\setminus C $ is an union of at most $s\card{V\Gamma_\mathcal{A}}$ graphs $A_i$. Each of these contains at most $\max\card{A_i}$ vertices, and $\Gamma_\mathcal{A}$ contains at most $\frac{\card{V\Gamma}}{\min\card{A_i}}$ vertices. Then, each connected subgraph of $\Gamma\setminus C $ contains at most $s\frac{\max\card{A_i}}{\min\card{A_i}}\card{V\Gamma}$ vertices. Finally,
\[ \cut^{s\times\frac{\max\card{A_i}}{\min\card{A_i}}}(\Gamma) \leq \max(\anch A_i)\times\cut^{s}(\Gamma_\mathcal{A}).\]
In particular, for $s = \frac{1}{2}\frac{\min(\card{A_i})}{\max(\card{A_i})}$,
\begin{align*}
\cut^{1/2}(\Gamma)&\leq \max (\anch A_i)\times\cut^{s}(\Gamma_\mathcal{A})\\
&\leq\frac4s\max(\anch A_i)\sep_{\Gamma_\mathcal{A}}(\card{V\Gamma_\mathcal{A}}\\
&=8\frac{\max(\card{A_i})}{\min(\card{A_i})}\max(\anch A_i)\sep_{\Gamma_\mathcal{A}}(\card{V\Gamma_\mathcal{A}}.\qedhere
\end{align*}
	\end{proof}

\subsection{Geometric method: bi-Lipschitz embeddings}\label{s_lipschitz}
	In this \cpt{section}{subsection}, we adress the question in the case where the so-called \textit{distorsion }satisfies some metric assumptions. More precisely, we assume that the initial graph embeds with a Lipschitz map, with some additional assumptions.
	\begin{theorem}\label{t:bilip}
		Let $ \Gamma $ and $ X $ be two graphs, with $ \Gamma $ finite containing at least 4 vertices. Let $ D\geq 2 $ be a bound on the degrees of the vertices of $ \Gamma $. Let $ \kappa >0 $, $ \alpha \in \left(0,1\right] $ and $ c > 0 $ be  such that there exists a map $ f \colon V\Gamma \rightarrow VX $ such that
\begin{enumerate}[label=(\roman*)]
	\item $d(f(x),f(y))\leq\kappa $, for every edge $\set{x,y}$ of $\Gamma$.
	\item \label{bilip_loc} for any subset $F\subset V\Gamma$ satisfying $\card{F}\geq \frac{V\Gamma}2$, we have
$$\frac1{\card{EF}}\sum_{\set{x,y}\in EF} d(f(x),f(y)) \geq \alpha \kappa ,$$
where $EF$ is the set of edges of $\Gamma$ of the form $\set{x,y}$ with $x,y\in F$.
	\item \label{diam} for any ball $ B $ of $X$ of radius $ \kappa $, we have $ \card{f^{-1}\left(B\right)}\leq c $.
\end{enumerate} 
Then
\[ \Sep_X\left (\kappa\frac{D}{2} \card{V\Gamma}\right ) \geq \frac\alpha{4c^3D} \cut^{1/2}\left (\Gamma\right ).\]
\end{theorem}
\begin{remark}
	The assumptions of the theorem above are satisfied when $\Gamma$ embeds in $X$ with a bilipschitz map of constants $ \alpha\kappa $ and $ \kappa $, taking $ c $ to be the maximal size of a ball of radius $\frac{1}{\alpha}$ in $\Gamma$. This is the setting we have in mind. The assumptions on $f$ are a little more general, allowing some local perturbations, such that $f$ is still bilipschitz on average (assumption~\ref{bilip_loc}), and satisfies a loose notion of injectivity (assumption~\ref{diam}).
\end{remark}
\begin{corollary}\label{c:distorted2}
Let $ \Lambda $ be a finite graph, and $D$ be a bound on the degrees of the vertices of $\Lambda$. Let $\kappa\geq2$ be an integer. Let $ \Gamma $ be the graph obtained adding $ \kappa $ vertices along each edge of $ \Lambda $. Then, $\Gamma$ has a subgraph $\Gamma'$ such that
\[\cut(\Gamma')\geq(4D)^{-1}\cut^{1/2}(\Lambda),\]
\end{corollary}
\begin{proof}
The canonical map $V\Lambda\hookrightarrow V\Gamma$ is clearly $\kappa+1$-bilipschitz, then we can apply Theorem~\ref{t:bilip} with $\alpha=1$ and $c=0$.
\end{proof}
\begin{corollary}\label{c:altproof2}
Let $\Gamma_s^{k_s,0}$ be as in Definition~\ref{dist_lamp_goup}, with $r=0$. Then, $\Gamma_s^{k_s,0}$ has a subgraph $\Gamma'$ such that
\[\cut(\Gamma')\geq \frac14(\card{A}+\card{B})^{-1}\cut(\Gamma_s).\]
\end{corollary}
\begin{proof}
The canonical map $V\Gamma_s\hookrightarrow\Gamma_s^{k_s,0},x\mapsto(x,0)$ is $2k_s$-bilipschitz, then we can apply Theorem~\ref{t:bilip} with $\alpha=1$ and $c=0$. Moreover, the degree of $\Gamma_s$ is equal to $\card{A}+\card{B}$.
\end{proof}
\begin{proof}[Proof of Theorem~\ref{t:bilip}]
Given a graph $\Lambda$, we will identify every subset of $V\Lambda$ with a subgraph of $\Lambda$, kepping every edge of $\Lambda$ of the form $\set{x,y}$, with $x,y\in V\Lambda$.

We will define a subgraph $ \Gamma' $ of $ X $, that will be considered as an avatar of $ \Gamma $. For any edge $\set{x,y}$ of $\Gamma$, the vertices $f(x)$ and $f(y)$ are at distance at most $\kappa$, then we can choose a sequence of less than $\kappa-1$ vertices that link them along a geodesic. We will denote the set of these vertices by ``$ \geod(f(x),f(y)) $''. We then define $ \Gamma' $ as the graph \[\Gamma'=f(V\Gamma) \cup\bigcup_{\set{x,y}\in E\Gamma} \geod(f(x),f(y)).\]

We can define a projection map \[\fonction{\pi_\Gamma}{V\Gamma'}{\mathcal P(V\Gamma)}{x}{\set{y\in V\Gamma\mid d(x,f(y))=d(x,f(V\Gamma)}.}\]
For every $x\in V\Gamma'$, we have
\[\stepcounter{equation}\tag{\theequation}\label{e:pigamma}
\pi_\Gamma(x)\subset \set{y\in V\Gamma\mid d(f(x),y)\leq\kappa}.\]
The graph $ \Gamma $ has at most $ \frac{1}{2} D \card{V\Gamma} $ edges. Therefore, \[\stepcounter{equation}\tag{\theequation}\label{e:vammaprmvgmm} \card{V\Gamma'} \leq  \card{V\Gamma} + (\kappa - 1)\frac{D}{2} \card{V\Gamma} \leq \kappa\frac{D}{2}\card{V\Gamma}.\] 
		Let $ s = \frac{\alpha}{D  c^2}\in\left(0,1\right) $. Let $ C'$ be a $ s $-cut set of $ \Gamma' $.
		We set $ C = \set{x\in V\Gamma \mid d(f(x),C')\leq \kappa }	 $.
		We have \[ \stepcounter{equation}\tag{\theequation}\label{inclu} f^{-1}\left(C'\right)\subset C \quad\text{and}\quad\pi_{\Gamma}(C')\subset C, \]
where the second inclusion comes from~\eqref{e:pigamma}. Moreover, by assumption \ref{diam}, to each vertex of $ C' $ corresponds at most $ c $ vertices in $ C $. Therefore
		\[ \card{C} \leq  c \card{C'}. \]
		We will show that $ C $ is a $1/2$-cut of the graph $ \Gamma $. Let $F$ be a connected subgraph of $\Gamma\setminus C$. We need to show that $ F $ contains at most half of the vertices of $ \Gamma $. Let us assume by contradiction that we have $ \card{F}>\card{V\Gamma}/2 $.  Let $F'$ be the following subset of $ V\Gamma' $: \[ F' = f(F)\cup\bigcup_{(x,y)\in EF} \geod(f(x),f(y)). \]
		Since $ {F} $ is connected, $ {F'} $ is connected as well. Let us see that $ F' $ do not intersect $ C' $. First, from the left inclusion of \eqref{inclu}, $ f(F) $ do not intersect $ C' $. Second, if $\set{v_1,v_2}$ is an edge of ${F}$, and $ v'$ is a vertex of $\geod(f(v_1),f(v_2)) $, then we have $ d(v',f(v_1))\leq\kappa $. Therefore, from the definition of $ C $, and since $ v_1 $ is not in $ C $, $ v' $ is not in $ C' $.

Then, $F'$ is a connected subgraph of $\Gamma$ and do not intersect $C'$. From the fact that $ C' $ is an~$ s $-cut set of $ \Gamma' $, we can deduce
			\[\stepcounter{equation}\tag{\theequation}\label{e:efprimegammaprim}  \card{F'}\leq s \card{V\Gamma'}. \]

		To each edge of the graph ${F}$ corresponds some vertices in $ F' $: the images by $ f $ of the source and the target of the edge, and the vertices that link these two points along the geodesic ``$\geod$'' we have chosen. We can call this set of vertices a ``path''. From assumption \ref{bilip_loc} this gives in total at least $ \card{E{F}}\alpha\kappa $  vertices, counted with multiplicity. 

A single vertex of $ F' $ can lie in several of these paths. Precisely, if a vertex $ x $ appears in $ k $ paths, then we can call $v_1,\dots,v_l$ the endpoints of these paths. Then, we have $k\leq C_{l}^{2} = \frac{l(l-1)}{2} $. Moreover, for any $ i $, the distance from $ x $ to $ f(v_i) $ is at most $ \kappa $. Then, from assumption \ref{diam} we have $ l\leq c$. So $ k \leq \frac{c^2}{2}$. Finally, we can deduce
		\[\stepcounter{equation}\tag{\theequation}\label{e:efprimef}  \card{F'} \geq \frac{2\alpha\kappa}{c^2}\card{E{F}}.\]
		Then, since $ {F} $ is connected, we have $ \card{F} \leq \card{EF} + 1 $ and then, combining with the previous inequalities:
\begin{align*}
\card{F} &\leq \frac{c^2}{2\alpha\kappa}\card{F'}+1\quad\text{from~\eqref{e:efprimef}}\\
&\leq \frac{sc^2}{2\alpha\kappa}\card{V\Gamma'}+ 1\quad\text{from~\eqref{e:efprimegammaprim}}\\
&\leq \frac{sc^2D}{4\alpha}\card{V\Gamma}+1 \quad\text{from~\eqref{e:vammaprmvgmm}}\\
&= \frac{1}{4}\card{V\Gamma} + 1.
\end{align*}
If $ \Gamma $ has at least $ 4 $ vertices, we deduce $\card{F} \leq \frac{\card{V\Gamma}}{2}$, which is a contradiction. Then, the graph $\Gamma$ has a $\frac12$-cut set of size at most $c\cut^{s}(\Gamma').$ We deduce
\begin{align*}
\cut^{1/2}(\Gamma) &\leq c\cut^{s}\left(\Gamma'\right)\\
&\leq c\frac4s\sep_{\Gamma'}(\card{V\Gamma'})\quad\text{from Lemma \ref{l_comp_cuts}}\\
&\leq c\frac4s\sep_{\Gamma'}\left(\kappa \frac D2\card{V\Gamma}\right)\quad\text{from Lemma \ref{e:vammaprmvgmm}}\\
&\leq c\frac4s\sep_{X}\left(\kappa \frac D2\card{V\Gamma}\right)\\
&= \frac{4c^3D}\alpha\sep_{X}\left(\kappa \frac D2\card{V\Gamma}\right).\qedhere
\end{align*} 
\end{proof}
\subsection{Analytic method: \texorpdfstring{$L^p$}{Lp}-Cheeger constants}\label{s_analyticmethod}
	In this \cpt{section}{subsection}, we adress the question from an analytic point of view. We will consider that both initial and distorted graphs describe the same metric space, but at different scales.
	\subsubsection*{Statement and consequences}
We start with some definitions.
		\begin{definition}\label{d:sepatated}
		Let $ \Gamma = (V\Gamma,E\Gamma) $ be a graph, and $ b \geq 2 $. Let $ Y $ be a subset of $ X $.
		\begin{itemize}
			\item We say that $ Y $ is \textbf{$ b $-separated} if for every pair $ y,y' $ of distinct points of $ Y $, we have $ d(y,y')\geq b $.
			\item We say that $ Y $ is \textbf{maximal $ b $-separated} if moreover it is maximal with this property: any subset $ Z $ of $ X $ that is $ b $-separated and contains $ Y $, is equal to $ Y $.
		\end{itemize}		
	\end{definition}
	\begin{definition}\label{d:rescaling}
	Let $ \Gamma = (V\Gamma,E\Gamma) $ be a graph, and $ b > 0 $. Let $ S $ be a maximal $ b $-separated subset of $ V\Gamma $. Then we can endow $ S $ with a \emph{graph structure}, declaring that $v$ and $v'$ in $S$ are neighbours if and only if $ d_{\Gamma}(v,v')<2b $.
	
	Any graph obtained with this process will be called a \textbf{$ b $-rescaling} of $ \Gamma $.
	\end{definition}

	\begin{theorem}\label{thm:cheeggraphs}
		Let $ \Gamma $ be a finite graph of maximal degree $ D $, let $ b $ be a positive integer and $k$ be such that every ball of radius $8b$ in $\Gamma$ have at most $kb$ vertices. Let $ \Lambda $ be a $ b $-rescaling of $ \Gamma $. Then there exists a positive constant $C$ that only depend on $ D $ and $ k $ such that for any $ p \in \left[1,\infty\right) $,
		\[ h_p(\Gamma)\geq \frac{C}{b} \cdot h_p(\Lambda),\]
	\end{theorem}
	Recall that $h_{1}^p(\Gamma)$ denotes the Cheeger constant of the graph $\Gamma$ (see Definition~\ref{d:lpchgrcnst}). The theorem is only intersting when $k$ is independent on $b$. This is the case in the following corollaries, which give examples of maps.
	\begin{corollary}\label{c:distorted3}
		Let $ \Lambda $ be a finite graph. Let $ \kappa $ be a positive integer. Let $ \Gamma $ be the graph obtained adding $ \kappa $ vertices along each edge of $ \Lambda $. Then there exists a positive constant $ C $ depending only on the maximal degree of $ \Lambda $ such that for any $ p \in \left[1,\infty\right) $,
		\[ h_p(\Gamma)\geq \frac{C}{\kappa} \cdot h_p(\Lambda).\]
	\end{corollary}
There is a $\kappa^{-1}$ factor on the right-hand side, which differs from Corollaries~\ref{c:distorted1} and~\ref{c:distorted2}. However, the equivalence between $\cut(\Gamma)$ and $\card\Gamma h(\Gamma)$ shown by Hume~\cite{hume2017} (used in the proof of Theorem~\ref{prop:equivseppoinc}) shows that this result is not weaker.
	\begin{proof}[Proof of Corollary \ref{c:distorted3}]
Let us consider $ V\Lambda $ as a subset of $ V\Gamma $. For any distinct pair of vertices $ \lambda,\lambda' $ in $ V\Lambda $, we have $ d_{\Gamma}(\lambda,\lambda')\geq \kappa $. Then $ V\Lambda $ is a $ \kappa $-separated subset of $V\Gamma $. Moreover, any vertex of $ \Gamma $ in $ \Gamma\setminus\Lambda $ is at distance less than $ \kappa $ from a vertex of $ \Lambda $. Therefore $ V\Lambda $ is maximal $ \kappa $-separated in $ \Gamma $. Is is clear that the corresponding $ b $-rescaling is equal to the graph $ \Lambda $. Finally, in $ \Gamma $, the balls of radius $ 8\kappa $ contain at less than $ D^{9}\kappa $ vertices, therefore the result follows from Theorem \ref{thm:cheeggraphs}.
	\end{proof}
\begin{corollary}\label{c:altproof3}
Let $\Gamma_s^{k_s,0}$ be as in Definition~\ref{dist_lamp_goup}, with $r=0$. Let $D$ be the degree of the graph $\Gamma_s$. Then, there exists a positive constant $C'$ that only depend on $D$ such that we have for any $p\in[1,\infty)$
\[h_p(\Gamma_s^{k_s,0})\geq\frac{C'}{k_s}h_p(\Gamma_s).\]
\end{corollary}
\begin{proof}
We recall that the vertex set of $\Gamma_s^{k_s,0}$ is $\Gamma_s\times[-k-s,k_s]$. The subset of elements of the form $(x,0)$, with $x\in\Gamma_s$, is $2k_s$-separated. The $2k_s$-rescaling associated with this subset is isomorphic to $\Gamma_s$. Moreover, the balls of radius $16k_s$ in $\Gamma_s^{k_s,0}$ contain at most $2k_sD^9$ vertices. The inequality follows from Theorem~\ref{thm:cheeggraphs}.
\end{proof}
\subsubsection*{Proof of Theorem \ref{thm:cheeggraphs}}
We give the proof of Theorem \ref{thm:cheeggraphs}. For any $ r $ and $ y $, we will denote by $ B(y,r) $ the closed ball centred at $ y $ of radius $ r $. When $(Z,\nu)$ is a positive finite measure space, we denote the averaged integral by $ \dashint_Z f d\nu \vcentcolon= \frac{1}{\nu(Z)} \int_Z f d\nu. $ 
	After \cite{humemackaytessera}, we introduce a notion of metric measure spaces.
\begin{definition}\label{def:metricmeasurespace:1geodesic}
	A \textbf{standard metric measure space} is a metric measure space $(X,d,\mu)$ with the following properties:
\begin{enumerate}[label=(\roman*)]
\item $(X,d)$ is a complete and separable metric space.\label{i:mes}
\item $\mu$ is a non-trivial, locally finite, Borel measure.
\item $X$ has \textbf{bounded packing on large scales}: there exists $r_0\geq 0$ such that for all $r\geq r_0$, there exists $K_r>0$ 	such that
\[\forall x\in X, \; \mu(B(x,2r))\leq K_r\mu(B(x,r)).\]
We then say that $X$ has \textbf{bounded packing on scales $\geq r_0$}.\label{i:bddpkngs}
\item $X$ is $k$\textbf{-geodesic} for some $k>0$: for every pair of points $x,y\in X$ there is a sequence $x=x_0,\ldots,x_n=y$ such that $d(x_{i-1},x_i)\leq k$ for all $i$ and $d(x,y)=\sum_{i=1}^n d(x_{i-1},x_i)$.	 
\end{enumerate}
\end{definition}
Up to rescaling the metric we will always assume that $X$ is $1$-geodesic and has bounded packing on scales $\geq1$.
\begin{definition}
	We will say that a subset of a standard metric measure space is \textbf{$1$-thick} if it is a union of closed balls of radius $1$. Axioms~\ref{i:mes} and~\ref{i:bddpkngs} imply in particular that a non-empty $1$-thick subset has positive measure. Such a subset $Z\subset X$ will be equipped with the \textbf{induced measure} and the \textbf{induced and $1$-distance}:
	\[d(z,z')=\inf\set{\sum_{i=1}^n d(z_{i-1},z_i)},\]
where the infimum is taken over all sequences $z=z_0,\ldots,z_n=z'$, such that each $z_i$ is an element of $Z$, and $d(z_i,z_{i+1})\leq 1$ for every $i$. (this distance takes values in $[0,\infty]$.)
\end{definition}	
\begin{remark} In the case of a bounded degree graph, $d$ is the shortest path metric and $\mu$ is the (vertex) counting measure. $1$-thick subspaces are $1$-thick subgraphs equipped with the vertex counting measure and their own shortest path metric.
\end{remark}
	The following definition is a generalization of Definition~\ref{d:lpchgrcnst}, for standard metric measure spaces, and different scales.
	\begin{definition}
	Let $(X,d,\nu)$ be a measured metric space and let $a>0$. Given a measurable function $f:X\to \R$, we define its \textbf{upper gradient at scale $a$} to be $$|\nabla_{a} f|(x)=\sup_{y,y'\in B(x,a)}|f(y)-f(y')|.$$ 

		Let $(Z,d,\nu)$ be a metric measure space with finite measure and fix a scale $a>0$.
		We define the $L^p$\textbf{-Poincar\'e constant at scale} $a$ of $Z$ to be
		$$h_{a,p}(Z)=\inf_{f} \frac{\|\nabla_a f\|_p}{\|f\|_p},$$
		where the infimum is taken over all $f\in L^p(Z,\nu)$ such that $f_Z:=\frac{1}{\nu(Z)}\int_Z fd\nu=0$ and $f \not\equiv 0$. We adopt the convention that $h_{a,p}(Z)=0$ whenever $\nu(Z)=0$.	
	\end{definition}
This generalizes Definition~\ref{d:lpchgrcnst} in the following sense: if we endow a graph with shortest path distance and the (vertex) counting measure, we get the same definition. We now introduce a notion of discretization for metric measure spaces.
	\begin{definition}
		Let $ (Z,d,\nu )$ be a metric measured space and $ b > 0 $. A partition $ \mathcal{A} = (A_y)_{y\in Y} $ of $ Z $ is called a \textbf{partition of scale $ b $} if for any $ A\in \mathcal{A} $, there exists $ z\in Z $ such that
		\[ B(z,b)\subset A\subset B(z,2b). \]
		Any point $ z $ satisfying these inclusions is called a \textbf{$ b $-centre} of $ A $.
		We will always assume that such a partition $ \mathcal{A} $ is indexed by a set of $ b $-centres. This implies in particular that $ Y $, which is a priori an abstract set, is a subset of $ Z $. 
	\end{definition}
	\begin{definition}\label{d:discretztion}
	Let $ (Z,d,\nu)$ be a metric measured space and $ b > 0 $. Let $ \mathcal{A} = (A_y)_{y\in Y} $ be a measurable partition of scale $ b $, such that for any $ y\in Y $, $ y $ is a $ b $-centre of $ A_y $.
	
	Then we can endow $ Y $ with the subset distance, and the unique measure $ \nu_Y $ satisfying $ \nu_Y(\{y\}) = \nu(A_y) $. 
		
	Let $\pi: Z\to Y$ be defined by ``$\pi(z)$ is the only $y\in Y$ such that $z\in A_y$''. Note that $\pi$ is surjective, and a right-inverse of the inclusion $j: Y\to Z$. Moreover, $\pi^{-1}(\{y\})=A_y$ for every $y\in Y$.
	
	Any space $ (Y,d_{\vert Y},\nu_Y) $ obtained with this process will be called a \textbf{discretization of $ Z $ parameter $ b $}.
\end{definition}
	\begin{remark}
		\begin{enumerate}
			\item Given a maximal $ b $-separated subset $ Y $ of $ Z $ (see Definition \ref{d:sepatated}), there always exists a partition of scale $ b $ indexed by $ Y $. Then we can consider $Y$ as a metric measure space, up to choosing an appropriate partition. Indeed, since $ \cup_{y\in Y}B(y,2b) $ covers $ Z $, one can find a measurable partition of scale $ b $ such that each element is $ b $-centred at a point of $ Y $.
			\item As we mentioned above, any graph can also be considered as a metric measure space, where the distance takes only integer values. The notion of \textit{$ b $-rescaling} (Definition \ref{d:rescaling}) should not be confused with the \textit{discretization of parameter $ b $} presented here. Indeed, given a positive integer $ b $ and a maximal $ b $-separated subset of a given graph, one can construct a $ b $-rescaling (see details below in the proof of Theorem \ref{thm:cheeggraphs}), or, choosing an appropriate partition of scale $ b $, a \textit{discretization of parameter $ b $}. These two metric measure spaces are different, but look alike when the initial graph has enough regularity; one may notice that the distances differ by a factor between $ b $ and $ 2b $.

		\end{enumerate}
		
	\end{remark}
	\begin{proposition}(see \cite[Lemma 5.8]{humemackaytessera})\label{p_discretization}
		Let $(Z,d,\nu)$ be a metric measure space of finite total measure. Assume there is no $z\in Z$ with $\nu(\set{z})>\frac23\nu(Z)$. Let $ Y $ be a discretization of $ Z $ of parameter $ b \geq 1$. Then for all $p\in[1,\infty)$ and all $ a \geq 2b $,
		$$h_{a,p}(Y) \leq 12 h_{2a,p}(Z),\quad\textup{and}\quad h_{a,p}(Z)\leq h_{3a,p}(Y).$$
	\end{proposition}
	We will use the following lemma:
	\begin{lemma}\label{prop:linubd}(see \cite[Proposition 7.1]{humemackaytessera})
	Let $Z$ be as in Proposition~\ref{p_discretization}. Then for all $p\in[1,\infty)$ and all $a\geq 1$, we have $h_{a,p}(Z)\leq 6$.
	\end{lemma}
	\begin{proof}[Proof of Lemma \ref{prop:linubd}]
From our assumptions (Definition~\ref{def:metricmeasurespace:1geodesic}), $\nu$ is measure isomorphic to a real interval and an at-most-countable collection of atoms. Then there exists a subset $Y\subset Z$ satisfying $\frac13\nu(Z)\leq \nu(Y) \leq \frac23\nu(Z)$.  Let $f$ be the characteristic function of $Y$. 
		
		Then $\norm{f-f_Z}^p \geq \frac{\nu(Z)}{3 \cdot 2^{p}}$ and $\norm{\nabla_a f}^p \leq \nu(Z)$, thus $h_{a,p}(Z) \leq 2 \cdot 3^{\frac1p} \leq 6$.
	\end{proof}
	\begin{proof}[Proof of Proposition \ref{p_discretization}]
	This is the same proof as in \cite{humemackaytessera}, where we detail the constants involved.

		Let $ \mathcal{A} = (A_y)_{y\in Y} $ be a partition of scale $ b $ associated with $ Y $. Let $f\in L^{\infty}(Z)$ be such that $\dashint_Z fd\nu=0$. We define $\phi\in \ell^{\infty}(Y)$ by 
		$\phi(y)=\dashint_{A_y}fd\nu$. Clearly $\dashint_Y \phi d\nu_Y=0$ and $\|\phi\circ \pi\|_{Z,p} = \|\phi\|_{Y,p}$. 
		Write $f(z)=\phi(\pi(z))+\dashint_{A_{\pi(z)}}(f(z)-f(w))d\nu(w)$.
		Then
		\begin{align*}
			\|f\|_{Z,p} & \leq \|\phi\circ\pi\|_{Z,p} + \left(\int_Z\left|\dashint_{A_{\pi(z)}}\left(f(z)-f(w)\right) d\nu(w)\right|^p d\nu(z)\right)^{1/p}
			\\ & \leq \|\phi\|_{Y,p} + \left(\int_Z \dashint_{A_{\pi(z)}} |f(z)-f(w)|^p\,d\nu(w)d\nu(z)\right)^{1/p}
			\\ & \leq \|\phi\|_{Y,p} + \left(\int_Z |\nabla_{2a} f|(z)^p d\nu(z)\right)^{1/p}
			\\ & = \|\phi\|_{Y,p} + \|\nabla_{2a} f \|_p\ .
		\end{align*}
		On the other hand, for any $ y,y' $ in $ Y $, $ \phi(y') $ is in the interval $\left[\inf_{A_{y'}}f, \sup_{A_{y'}} f\right]$, and each $ A_{y'} $ satisfying $ d(y,y')\leq a $ is contained in the ball $ B(y,a+2b) $. Then, we have \[ |\nabla_{a} \phi|(y)\leq |\nabla_{a+2b}f|(z)\leq |\nabla_{2a}f|(z),\quad\text{for any $y\in Y$ and $z\in A_y$}.\]
		
		We now prove the first inequality of Proposition~\ref{p_discretization}.
		If $h_{2a,p}(Z)\leq \frac12,$ then 
		for any $\epsilon\in(0,1/6)$ we can find $f$ as above so that
		\[
		\frac{2}{3}\geq \frac{1}{2} +\epsilon \geq h_{2a,p}(Z)+\epsilon
		\geq \frac{\|\nabla_{2a} f\|_p}{\|f\|_p}
		\geq \frac{\|\nabla_{2a} f\|_p}{\|\phi\|_p + \|\nabla_{2a} f\|_p}.
		\]
		Thus $\|\nabla_{2a} f\|_p \leq 2\|\phi\|_p$ and
		\[
		h_{2a}^p(Z)+\epsilon \geq \frac{\|\nabla_a \phi\|_p}{3\|\phi\|_p}
		\geq \frac{1}{3} h_a^p(Y).
		\]
		Since $\epsilon$ was arbitrary, $h_{a,p}(Y)\leq 3h_{2a,p}(Z)$. Moreover, from Lemma \ref{prop:linubd}, $h_{a,p}(Y)\leq 6$, so if $h_{2a,p}(Z)\geq \frac12$, then $h_{a,p}(Y) \leq 12 h_{2a,p}(Z)$.
		
		The other direction is easier: 
		given $\psi\in \ell^{\infty}(Y)$ such that $\dashint_Y \psi d\nu_Y=0,$ we define $$g\vcentcolon=\sum_{y\in Y} \psi(y)1_{A_y},$$ where $1_{A_y}$ denotes the characteristic function of $A_y$. We clearly have $\dashint g d\nu=0$ and $\|g\|_p=\|\psi\|_p$. 
		Hence we are left with comparing the gradients. 
		\begin{align*}
		\|\nabla_{a}g \|_p^p & =  \sum_Y\nu(A_y)\dashint_{A_y}\sup_{z',z''\in B(z,a)}|g(z')-g(z'')|^pd\nu(z) \\
		& \leq   \sum_Y\nu(A_y)\sup_{z',z''\in B(y,a+2b)}|g(z')-g(z'')|^p \\
		& \leq    \sum_Y \nu_Y(y) \sup_{y',y''\in B(y,a+4b)\cap Y} |\psi(y')-\psi(y'')|^p\\
		& =   \|\nabla_{3a}\psi\|_p^p. \qedhere
		\end{align*}
	\end{proof}
	We will need the following proposition to compare Poincar\'e constants at different scales.
	\begin{proposition}\label{prop:scales}(see \cite[Proposition 4.3]{humemackaytessera})
		Let 
		$ (Z,d,\nu) $ be a $ 1 $-geodesic metric measure space.
		Then for any $a\geq 3$ and all $p\in[1,\infty)$ we have  
		$$ \frac{\nu_{\min}(1/2)}{\nu_{\max}(2a)}\cdot h_{a,p}(Z) \leq h_{\frac32,p}(Z) \leq h_{a,p}(Z),$$
		where $ \nu_{\min}(1/2) $ denotes the minimal measure of a ball of $ Z $ of radius $ 1/2 $, and $ \nu_{\max}(2a) $ denotes the maximal measure of a ball of $ Z $ of radius $ 2a $.
	\end{proposition}
	\begin{proof}
		This is the same proof as in \cite{humemackaytessera}, where we detail the constants involved.
		
		The right-hand side inequality is obvious. Let us prove the left-hand side. Let $ f $ be a measurable function $ Z \to \R $.
		Let $z\in Z$, and let $x,y$ be two distinct points of $B(z,a)$.
		Then there exists $x=x_0,\ldots, x_n=y$ within $B(z,a)$ such that $ d(x_{i+1},x_i) \leq 1 $ for all $ i $, and $d(x,y)=\sum_{i=1}^n d(x_{i-1},x_i)$. Up to removing vertices, we can make the assumption that this sequence is minimal in the following sense:
		\[ \forall i,j\in\llbracket 0, n\rrbracket\,\left(\vert j - i\vert > 1\implies d(x_i,x_j)>1\right). \]
		Note that removing vertices may make the equality $d(x,y)=\sum_{i=1}^n d(x_{i-1},x_i)$ fail, but we keep the property that every $ x_i $ is at distance at most $ a/2 $ from $ x $ or $ y $.
		We claim that the following inequality is true:
		\[\stepcounter{equation}\tag{\theequation}\label{eq:lwbdgrdnt} \int_{z'\in B(z,2a)} \card{\nabla_{\frac32} f}d\nu \geq \nu_{\min}(1/2)\cdot\vert f(x)- f(y)\vert\]
		We consider two cases:
		\begin{itemize}
			\item if $ n $ is even, let us call $ Z_{x,y} $ the set of $ z'\in Z $ 
			that are in the $ \frac32 $-neighbourhood of both $ x_{2i-2} $ and $ x_{2i} $ for some integer $ i $ between $ 1 $ and $ n/2 $. Then, since $ a\geq 3 $, $ Z_{x,y} $ is contained in the ball $ B(z, 2 a) $. It contains the closed balls $ B(x_{2i-1},\frac12) $, for any such $ i $. From the minimality assumption that we have made on the path $ (x_i)_{0\leq i\leq n} $, these balls are pairwise disjoints. Then,
			\begin{align*}
				\int_{z'\in B(z,2a)} \card{\nabla_{\frac32} f}d\nu &\geq
				\int_{z'\in Z_{x,y}} \card{\nabla_{\frac32} f}d\nu\\
				&\geq \sum_{i=1}^{n/2} \int_{B(x_{2i-1},\frac12)} \card{\nabla_{\frac32} f}d\nu\\
				&\geq \sum_{i=1}^{n/2} \int_{B(x_{2i-1},\frac12)} \vert f(x_{2i})-f(x_{2i-2})\vert d\nu\\
				&\geq \nu_{\min}(1/2)\sum_{i=1}^{n/2}\vert f(x_{2i})-f(x_{2i-2})\vert\\
				&\geq \nu_{\min}(1/2)\cdot\vert f(x)- f(y)\vert\\
			\end{align*}
			\item if $ n $ is odd, let us call $ Z'_{x,y} $ the set of $ z'\in Z $ that are in the $ \frac32 $-neighbourhood of both $ x_{2i-2} $ and $ x_{2i} $ for some integer $ i $ between $ 1 $ and $ (n-1)/2 $, or that are in the $ \frac32 $-neighbourhood of both $ x_{n-1} $ and $ y $. Then, since $ a\geq 3 $, $ Z_{x,y} $ is contained in the balls $ B(z, 2 a) $. It contains the closed ball $ B(x_{2i-1},\frac12) $, for any $ i $ from $ 1 $ to $ (n+1)/2 $ (note that the last ball is centred at $ y $). From the minimality assumption that we have made on the path $ (x_i)_{0\leq i\leq n} $, these balls are pairwise disjoints. Then,
			\begin{align*}
				\int_{z'\in B(z,2a)} \card{\nabla_{\frac32} f}d\nu &\geq
				\int_{z'\in Z'_{x,y}} \card{\nabla_{\frac32} f}d\nu\\
				&\geq \sum_{i=1}^{(n+1)/2} \int_{B(x_{2i-1},\frac12)} \card{\nabla_{\frac32} f}d\nu\\
				&\geq \sum_{i=1}^{(n-1)/2} \int_{B(x_{2i-1},\frac12)} \vert f(x_{2i})-f(x_{2i-2})\vert d\nu + \int_{B(y,\frac12)} \vert f(x_{n-1})-f(y)\vert d\nu\\
				&\geq \nu_{\min}(1/2)\left(\sum_{i=1}^{(n-1)/2}\vert f(x_{2i})-f(x_{2i-2})\vert + \vert f(x_{n-1})-f(x_{n})\vert\right) \\
				&\geq \nu_{\min}(1/2)\cdot\vert f(x)- f(y)\vert\\
			\end{align*}
		\end{itemize} 
		Since the inequality \eqref{eq:lwbdgrdnt} is true for any $ x,y\in B(z,2a) $, we deduce
		\[ \int_{z'\in B(z,2a)} \card{\nabla_{\frac32} f}d\nu \geq \nu_{\min}(1/2)\cdot\card{\nabla_{a} f}(z). \]
		Integrating over $ z $, we get:
		\[ \int_{z\in Z}\left(\int_{z'\in B(z,2a)} \card{\nabla_{\frac32} f}(z')d\nu(z')\right)^p d\nu(z) \geq \nu_{\min}(1/2)^p\cdot\norm{\nabla_{a} f}^p.\]
		Moreover for any $ z $,
		$$\left(\int_{z'\in B(z,2a)} \card{\nabla_{\frac32} f}(z')d\nu(z')\right)^p \leq \nu(B(z,2a))^{p-1} \int_{z'\in B(z,2a)} \left(\card{\nabla_{\frac32} f}(z')\right)^p d\nu(z').$$
		Then,
		\begin{align*}
			\nu_{\min}(1/2)^p\cdot\norm{\nabla_{a} f}^p &\leq \int_{z\in Z} \nu(B(z,2a))^{p-1} \int_{z'\in B(z,2a)} \left(\card{\nabla_{\frac32} f}(z')\right)^p d\nu(z') d\nu(z)\\
			&\leq \nu_{\max}(2a)^{p-1}\int_{z\in Z}\int_{z'\in B(z,2a)} \left(\card{\nabla_{\frac32} f}(z')\right)^p d\nu(z') d\nu(z)\\
			&= \nu_{\max}(2a)^{p-1}\int_{z,z'\in Z} \mathbf{1}_{d(z,z')\leq 2a} \left(\card{\nabla_{\frac32} f}(z')\right)^p d\nu(z') d\nu(z)\\
			&= \nu_{\max}(2a)^{p-1}\int_{z'\in Z} \left(\card{\nabla_{\frac32} f}(z')\right)^p \left( \int_{z\in Z} \mathbf{1}_{z\in B(z',2a)} d\nu(z) \right)d\nu(z')\\
			&\leq \nu_{\max}(2a)^p\int_{z'\in Z} \left(\card{\nabla_{\frac32} f}(z')\right)^p d\nu(z')\\
			&= \nu_{\max}(2a)^p\norm{\nabla_{\frac32} f}^p.
		\end{align*}
		Finally,
		\[ \norm{\nabla_{2a} f} \leq \frac{\nu_{\max}(2a)}{\nu_{\min}(1/2)} \norm{\nabla_{\frac32} f}.\qedhere\]
	\end{proof}
	We now can prove Theorem \ref{thm:cheeggraphs}.
	\begin{proof}[Proof of Theorem \ref{thm:cheeggraphs}]
		We can assume without loss of generality that $\Gamma$ is connected, because otherwise $h_{a,p}(\Gamma)=h_{a,p}(\Lambda)=0$.
		
	Let $ (\tilde{\Gamma},d,\nu) $ be the ``measured'' simplicial complex obtained identifying each edge of $ \Gamma $ to the unit interval equipped with the Lebesgue measure. We define $ \iota\colon V\Gamma\to\tilde{\Gamma} $ the natural map that maps the vertices of $ \Gamma $ in the simplicial complex $ \tilde{\Gamma} $. For simplicity, for a given a vertex $ v $ of $ V\Gamma $, we will still denote $ v $ the corresponding vertex $ \iota(v) $ in the simplicial complex $ \tilde{\Gamma} $.
	
	By definition, $ V\Lambda $ is a maximal $ b $-separated subset of $ V\Gamma $. $ \iota(V\Lambda) $ is the subset of $ \tilde{\Gamma} $ corresponding to $ V\Lambda $. We claim that $ \iota(V\Lambda) $ is also maximal $ b $-separated. First, $ \iota(V\Lambda) $ is clearly $ b $-separated. Second, if $ x $ be a point of $ \tilde{\Gamma} $, there exists a vertex $ v $ at distance at most $ 1/2 $. By maximality, there exists $ w\in V\Lambda $ such that $ d(w,v) < b $, and, since both terms are integers, we have $ d(w,v) \leq b-1 $. Then we have $ d(x,w)\leq b - 1/2 < b $, which shows that $ \iota({V\Lambda}) $ is maximal.   
		Let $ \mathcal{A} = (A_v)_{v\in \iota(V\Lambda)} $ be a measurable partition of scale $ b $ satisfying that each $ A_v $ is $ b $-centred at $ v $.
		We can identify $V\Lambda$ and $\iota(V\Lambda)$, then we have two different metric measure structures on $ V\Lambda $ : 
			\begin{itemize}\setcounter{footnote}{0}
			\item The graph $\Lambda = (V\Lambda,E\Lambda)$, which is $b$-rescaling associated with $V\Lambda$ (Definition~\ref{d:rescaling}), endowed with the shortest-path metric and the counting measure,
			\item The $b$-discretization\footnote{We use the notation $\Lambda_b$ because this space is \textit{close} from being the same space as $\Lambda$, where the distances are multiplied by $b$.} $ \Lambda_b = (\iota(V\Lambda),d_{\vert \iota(V\Lambda)},\nu_b) $ associated with $\mathcal{A}$, that we will call $\Lambda_b$ (Definition~\ref{d:discretztion}).
			\end{itemize}
		Roughly speaking, the inequality~\eqref{ing1} below states that taking the appropriate scale, their $L^p$ Cheeger constant do not differ too much. Let us write $ \nu_{\min}(b) $ be the minimal measure of a ball in $ \tilde{\Gamma} $ of radius $ b $, and $ \nu_{\max}(2b) $ be the maximal measure of a ball in $ \tilde{\Gamma} $ of radius $ 2b $. We have
\[\label{ing1}\stepcounter{equation}\tag{\theequation}
\left(\frac{\nu_{\max}(2b)}{\nu_{\min}(b)}\right)^{-1/p}\times h_{2b,p}(\Lambda_b)
\leq h_{1,p}(\Lambda) 
\leq \left(\frac{\nu_{\max}(2b)}{\nu_{\min}(b)}\right)^{1/p}\times h_{2b,p}(\Lambda_b).\]		
Let us prove this inequality. By definition (see Definitions~\ref{d:rescaling}, \ref{d:discretztion}), for any $ v $ in $ V\Lambda $,
\newcounter{ineg}
		\[\stepcounter{ineg}\tag{\fnsymbol{ineg}}\label{eq:cdsgvdfgf}
		B_{\tilde{\Gamma}}(v,b)\subset A_v\subset B_{\tilde{\Gamma}}(v,2b).\] 				Therefore:
		\[\stepcounter{ineg}\tag{\fnsymbol{ineg}}\label{eq:messingltns}
		 \nu_{\min}(b)\leq\nu_b(\{v\})\leq \nu_{\max}(2b),\quad\text{for any $ v $ in $ V\Lambda $}. \]
		We can now prove \eqref{ing1}. Let $ \tilde{f} $ be a function from $\iota(V\Lambda)$ to $\R$. Let us write $f$  the corresponding function from $ V\Lambda $ to $ \R $ (it is roughly the same function). From the right-hand side of \eqref{eq:cdsgvdfgf}, we have $|\nabla_1 f|^p(v) \leq |\nabla_{2b} \tilde{f}|^p(v)$. Then, 
		\begin{align*}
			\norm{\nabla_{2b}\tilde{f}}^p &= \sum_{v\in V\Lambda}|\nabla_{2b} \tilde{f}|^p(v)\nu_b(\{v\})\\
			&\geq \sum_{v\in V\Lambda}|\nabla_{2b} \tilde{f}|^p(v)\nu_{\min}(b)\\
			&\geq \nu_{\min}(b)\sum_{v\in V\Lambda}|\nabla_1 f|^p(v)\\
			&= \nu_{\min}(b)\norm{\nabla_{1}f}^p
		\end{align*}
		Moreover, from right-hand side of \eqref{eq:messingltns}, we have ${\|\tilde{f}\|_p} \leq \norm{f}\times{\nu_{\max}(2b)^{1/p}},$ and the right-hand side of \eqref{ing1} follows. The left-hand side of \eqref{ing1} comes very similarly, we let the proof to the reader (we will not use this inequality).\\
		
From Proposition \ref{p_discretization}, we can deduce
\[\stepcounter{equation}\tag{\theequation}
h_{2b,p}(\Lambda_b)\leq 12 h_{4b,p}(\tilde{\Gamma}).\]
		From Proposition \ref{prop:scales}, we can deduce
\[\stepcounter{equation}\tag{\theequation}
		h_{4b,p}(\tilde{\Gamma})\leq \frac{\nu_{\max}(8b)}{\nu_{\min}(1/2)} h_{\frac32,p}(\tilde{\Gamma}).\]
		We claim that we have:
		\[\stepcounter{equation}\tag{\theequation}\label{ing4}
		h_{\frac32,p}(\tilde{\Gamma})\leq D^{2/p}h_{1,p}(\Gamma).\]
		Indeed, if $f\colon V\Gamma\to\R$, then we can find $\tilde{f}\colon \tilde{\Gamma}\to\R$ such that for any $x$, $\tilde{f}(x)=f(v)$, where $v$ is a vertex of $\Gamma$ at distance at most $1/2$ from $x$. Since the degree of every vertex in $\Gamma$ is between $1$ and $D$, every ball in $\tilde{\Gamma}$ of radius $1/2$, centred at vertices, have a measure between $1/2$ and $D/2$. The inequality \eqref{ing4} follows from:
		\begin{itemize}
			\item $\norm{\tilde{f}}^p = \sum_{v\in V\Gamma}\card{f(v)}^p\nu(B(v,1/2))\geq\frac12\norm{f}^p$.
			\item For any $z$ in $\tilde{\Gamma}$ that is not at the middle of an edge, let us write $v$ its closest vertex. Then $\card{\nabla_{\frac32}\tilde{f}(z)} \leq\card{\nabla_2f(v)}\leq \sum_{w\sim v} \card{\nabla_1f(w)} $, where the last sum is taken on the set of neighbours of $v$. Then,
\begin{align*}
\norm{\nabla_{\frac32}\tilde{f}}^p&=\int_{z\in\tilde{\Gamma}}\card{\nabla_{\frac32}\tilde{f}(z)}^p d\nu(z)\leq \sum_{v\in V\Gamma}\int_{z\in B(v,1/2)}\card{\nabla_{\frac32}\tilde{f}(z)}^p d\nu(z)\\
&\leq\sum_{v\in V\Gamma}\left(\sum_{w\sim v}\card{\nabla_1f(w)}\right)^p\nu(B(v,1/2)\\
&\leq\sum_{v\in V\Gamma}D^{p-1}\left(\sum_{w\sim v}\card{\nabla_1f(w)}^p\right)D/2\\
&=\frac{D^p}2\norm{\nabla_1 f}^p.
\end{align*}
		\end{itemize}
		Theorem \ref{thm:cheeggraphs} then follows from the chain of inequalities from~\eqref{ing1} to~\eqref{ing4}:
\begin{align*}	
h_{1,p}(\Lambda) 
&\leq \left(\frac{\nu_{\max}(2b)}{\nu_{\min}(b)}\right)^{1/p}\times h_{2b,p}(\Lambda_b)\\
&\leq k^{1/p}\times 12 h_{4b,p}(\tilde{\Gamma})\\
&\leq12k^{1/p}\times\frac{\nu_{\max}(8b)}{\nu_{\min}(1/2)} h_{\frac32,p}(\tilde{\Gamma})\\
&\leq12k^{1/p}\frac12kbD^{2/p}h_{1,p}(\Gamma)\\
&\leq\Big(6k^{\frac{p+1}p}D^{2/p}\Big)bh_{1,p}(\Gamma).\qedhere
\end{align*}
	\end{proof}

\bibliography{/home/prolecoz/bibliothaeque/bibliothaeque}{}
\bibliographystyle{alpha}
\end{document}